%% file: PRC,PpCfieldsandNTP2.tex
\documentclass[a4paper,12pt]{article}
\usepackage{amssymb}
\usepackage{amsthm}
\usepackage{amsxtra}
\usepackage{amsfonts}
\usepackage[all]{xy}
\usepackage[T1]{fontenc}
\newcommand{\PRCB}{T}
\newcommand{\PpCB}{T}

\newcommand{\LC}{\mathcal{L}_n}
\newcommand{\LCR}{\mathcal{L}}
\newcommand{\Li}{\mathcal{L}^{(i)}}
\newcommand{\LCp}{\mathcal{L}_{n}}
\newcommand{\LCRp}{\mathcal{L}}
\newcommand{\Lip}{\mathcal{L}^{(i)}}
\newcommand{\clos}[2]{#1^{(#2)}}
\newcommand{\closp}[2]{#1^{(#2)}}
\newcommand{\tp}{\mathrm{tp}}
\newcommand{\qftp}{\mathrm{qftp}}
\newcommand{\Lstp}{\mathrm{Lstp}}
\newcommand{\acl}{\mathrm{acl}}
\newcommand{\dcl}{\mathrm{dcl}}
\newcommand{\Gal}{\mathrm{Gal}}
\newcommand{\G}{\mathrm{G}}
\newcommand{\bdn}{\mathrm{bdn}}
\newcommand{\dpr}{\mathrm{dp}}

\newtheorem{thm}{Theorem}[section]
\newtheorem{cor}[thm]{Corollary}
\newtheorem{lem}[thm]{Lemma}

\newtheorem{prop}[thm]{Proposition}

\newtheorem{fact}[thm]{Fact}

\theoremstyle{definition}
\newtheorem{para}[thm]{}
\newtheorem{Types:}[thm]{Types:}

\newtheorem*{claim}{Claim}
\newtheorem*{claim1}{Claim 1}
\newtheorem*{claim2}{Claim 2}
\newtheorem*{claim3}{Claim 3}
\newtheorem*{claim4}{Claim 4}
\newtheorem*{claim5}{Claim 5}

\newtheorem{defn}[thm]{Definition}
\newtheorem{notation}[thm]{Notation}
\newtheorem{rem}[thm]{Remark}

\def\rest{{\lower      .25    em      \hbox{$\vert$}}}

\font\helpp=cmsy5
\def\semdp
{\hbox{$\times\kern-.23em\lower-.1em\hbox{\helpp\char'152}$}\,}

\def\dnfo{\,\raise.2em\hbox{$\,\mathrel|\kern-.9em\lower.35em\hbox{$\smile$}
$}}
\def\dnf#1{\lower1em\hbox{$\buildrel\dnfo\over{\scriptstyle #1}$}}
\def\dfo{\;\raise.2em\hbox{$\mathrel|\kern-.9em\lower.35em\hbox{$\smile$}
\kern-.7em\hbox{\char'57}$}\;}
\def\df#1{\lower1em\hbox{$\buildrel\dfo\over{\scriptstyle #1}$}}

\def\Ind#1#2{#1\setbox0=\hbox{$#1x$}\kern\wd0\hbox to 0pt{\hss$#1\mid$\hss}     
\lower.9\ht0\hbox to 0pt{\hss$#1\smile$\hss}\kern\wd0}
\def\Notind#1#2{#1\setbox0=\hbox{$#1x$}\kern\wd0\hbox to 0pt{\mathchardef
\nn=12854\hss$#1\nn$\kern1.4\wd0\hss}\hbox to
0pt{\hss$#1\mid$\hss}\lower.9\ht0 \hbox to
0pt{\hss$#1\smile$\hss}\kern\wd0}

\def\ind{\mathop{\mathpalette\Ind{}}}                       
\def\thind{\mathop{\mathpalette\Ind{}}^{\text{\th}}}       
\def\nthind{\mathop{\mathpalette\Notind{}}^{\text{\th}} }   

\begin{document}

\title{Pseudo real closed fields, pseudo $p$-adically closed fields and NTP$_2$}
\author{Samaria Montenegro\footnote{s.montenegro@uniandes.edu.co; present address: Universidad de los Andes}\\ Universit\'e Paris Diderot-Paris 7\footnote{ 
Partially supported by ValCoMo (ANR-13-BS01-0006) and the Universidad de Costa Rica.}}
\date{}
\maketitle

\begin{abstract}
The main result of this paper is a positive answer to the Conjecture 5.1 of \cite{CKS} by A. Chernikov, I. Kaplan and P. Simon: If $M$ is a PRC field, then $Th(M)$ is NTP$_2$ if and only if $M$ is bounded. In the case of P$p$C fields, we prove that if $M$ is a bounded P$p$C field, then $Th(M)$ is NTP$_2$.
We also generalize this result to obtain that, if $M$ is a bounded PRC or P$p$C field with exactly $n$ orders or $p$-adic valuations respectively, then $Th(M)$ is strong of burden $n$. This also allows us to explicitly compute the burden of types, and to describe forking.  

  
\end{abstract}

\tableofcontents

\input{Introduction}

\input{PreliminariesPRC}
\input{BoundedPRCfields}
\input{PRCstability}

\input{PreliminariesPpC}

\input{BoundedPpCfield}

\input{PpCstability}

\section{Acknowledgments}
 
The contents of this paper constitutes a part of my PhD thesis, I am grateful to my advisor Zo\'e Chatzidakis for her guidance and helpful suggestions, as well as for all the corrections and advice.
I also want to thank Thomas Scanlon and Frank Wagner for their comments and corrections, and Artem Chernikov for his suggestion that a preliminary version of Theorem \ref{PRCNTP2} could be generalized to obtain strong in addition to NTP$_2$, as well as multiple discussions around the topics of the paper.  
Part of the work presented in this paper was made during the program on Model Theory,
Arithmetic Geometry and Number Theory at MSRI, Berkeley, Spring 2014. I thank the MSRI for its hospitality. 
This paper was partially supported by ValCoMo (ANR-13-BS01-0006) and the Universidad de Costa Rica.

\end{document}

%% file: Introduction.tex
\section{Introduction} \label{Introduction}

A  \emph{pseudo algebraically closed field} (PAC field) is a field $M$ such that every absolutely irreducible affine variety defined over $M$ has an $M$-rational point.
The concept of a PAC field was introduced by J. Ax in \cite{Ax} and has been extensively studied.
The above definition of PAC field has an equivalent model-theoretic version: $M$ is existentially closed (in the language of rings) in each regular field extension of $M$. 

. 

The notion of PAC field has been generalized  by S. Basarab  in  \cite{Ba0} and then by A. Prestel in \cite{Pre} to ordered fields. 
Prestel calls a field $M$ \emph{pseudo real closed} (PRC) if $M$ is existentially closed (in the language of rings) in each regular field extension $N$ to which all orderings of $M$ extend. 
PRC fields were extensively studied by L. van den Dries in \cite{Van}, Prestel in \cite{Pre}, M. Jarden in \cite{J1}, \cite{J2} and \cite{J3},  Basarab in \cite{Ba} and \cite{Ba1}, and others.

In analogy to PRC fields, C. Grob \cite{Gr}, Jarden and D. Haran \cite{HJ} studied the class of pseudo $p$-adically closed fields. 
A field $M$ is called a \emph{pseudo $p$-adically closed} (P$p$C) if $M$ is existentially closed (in the language of rings) in each regular field extension $N$ to which all the $p$-adic valuations of $M$ can be extended by $p$-adic valuations on $N$. P$p$C fields have also been studied by I. Efrat and  Jarden in \cite{EJ}, Jarden in \cite{J4} and others.

 The class of PRC fields contains strictly the classes of PAC fields of characteristic 0 and real closed fields (RCF) and the class of P$p$C fields contains the $p$-adically closed fields ($p$CF). It is known that the theories RCF and $p$CF are NIP. 
Duret showed in \cite{Du} that the complete theory of a PAC field which is not separably closed is not NIP.
In \cite{ChPi} Z. Chatzidakis and A. Pillay proved that if $M$ is a bounded (i.e. for any $n \in \mathbb{N}$, $M$ has only finitely many extensions of degree $n$) PAC field, then $Th(M)$ is simple. In \cite{Cha0} Chatzidakis proved that if $M$ is a PAC field and $Th(M)$ is simple, then $M$ is bounded.
 PRC and P$p$C fields were extensively studied, but mainly from the perspective of algebra (description of absolute Galois group, etc), elementary equivalence, decidability, etc. Their stability theoretic properties had not been studied.
In this paper we study the stability theoretic properties of the classes of PRC and P$p$C fields.

In Theorem \ref{IPPRC} we generalize the result of Duret and we show that the complete theory of a PRC field which is neither algebraically closed nor real closed is not NIP. In Corollary \ref{nPpCIPC} we show that the complete theory of a bounded P$p$C which is not $p$-adically closed is not NIP. The general case for P$p$C is still in  progress, the main obstacle is that the $p$-adic valuations are not necessarily definable, and that algebraic extensions are not necessarily P$p$C.

The class of NTP$_2$ theories (theories without the tree property of the second kind, see Definition \ref{defIPTP2}) was defined by Shelah in \cite{Shel} in the 1980's and  contains strictly the classes of simple and NIP theories.
Recently the class of NTP$_2$ theories has been particularly studied and contains new important examples, A. Chernikov showed in \cite{Che1} that any ultra-product of $p$-adics is NTP$_2$.
A. Chernikov and M. Hils showed in \cite{CheHil} that a $\sigma$-Henselian valued difference field of equicharacteristic $0$
is NTP$_2$, provided both the residue difference field and the value group (as an ordered difference group) are NTP$_2$.
There are not many more examples of strictly $NTP_2$ theories. 
A. Chernikov, I. Kaplan and P. Simon conjectured in \cite[Conjecture 5.1]{CKS} that if $M$ is a PRC field then $Th(M)$ is NTP$_2$ if and only $M$ is bounded.
Similarly if $M$ is a P$p$C field. 

The main result of this paper is a positive answer to the conjecture by Chernikov, Kaplan and Simon for the case of PRC fields (Theorem \ref{PRCNTP2}).
In fact for bounded PRC fields we obtain a stronger result:
In Theorem \ref{PRCstrong} we show that if $M$ is a bounded PRC field with exactly $n$ orders, then $Th(M)$ is \emph{strong} of burden $n$.
We also show that $Th(M)$ is not rosy (Corollary \ref{PRCnonrosy}) and resilient (Theorem \ref{PRCresilient}). The class of resilient theories contains the class of NIP theories and is contained in the class of NTP$_2$ theories.  It is an open question to know whether NTP$_2$ implies resilience.

The case of P$p$C fields is more delicate, and we obtain only one direction of the conjecture. In Theorem \ref{PpCstrong} we show that the theory of a bounded P$p$C field with exactly $n$ $p$-adic valuations is strong of burden $n$ and in Theorem \ref{PpCresilient} we show that this theory is also resilient.
That unbounded P$p$C fields have $TP_2$ will be discussed in another paper. 
The problem arises again from the fact that an algebraic extension of a P$p$C field is not necessarily P$p$C.

Independently, W. Johnson \cite{WJ} has shown that the model companion of the theory of fields with several independent orderings has NTP$_2$, as well as characterized forking, extension bases, the burden, and several 
other results. He also obtains similar results for the class of existentially closed fields with several valuations, or with several $p$-adic valuations. 
Some of his results follow from ours, since his fields are bounded, and PRC or P$p$C in case of several orderings or $p$-adic valuations. 
His results on fields with several valuations however cannot be obtained by our methods.

The organization of the paper is as follows:
In section \ref{PreliminariesPRC} we give the required preliminaries on ordered fields and pseudo real closed fields. 
In section \ref{SectionPRC} we study the theory of bounded PRC fields from a model theoretic point of view.
We work in a fixed complete theory of a bounded PRC field, and we enrich the language adding constants for an elementary submodel. In section \ref{SDT1v} we study the one variable definable sets; to do that we work with a notion of interval for multi-ordered fields. We define multi-intervals, and a notion of multi-density in multi-intervals (Definition \ref{defmultiinterval}).
Using that we find a useful description of one variable definable sets in terms of multi-intervals and multi-density.
In section \ref{SDTPv} we show that the results obtained in \ref{SDT1v} can be easily generalized to several variables.
We define a notion of multi-cell (Definition \ref{defmulticell}), and we find (Theorem \ref{descomposition2}) a description of definable sets in terms of multi-cells and multi-density.
This ``density theorem'' is extremely important and plays a role in most proofs.
In \ref{AmalgamationTheorems} we show some theorems about amalgamation of types for bounded PRC fields.
The main difficulty is caused by the orderings. 
In section \ref{PRCstability} we study the stability theoretic properties of PRC fields and we show the main theorem on NTP$_2$ and strongness for bounded PRC fields.
In section \ref{BurdenPRCfields} we calculate explicitly the burden of complete types, in \ref{sectionreliance} we show that the theory of a bounded PRC field is resilient, and in section \ref{forkingdividingPRC} we give a description of forking.  
In \ref{Lascarsection} we study the Lascar strong types, in Theorem \ref{LascarPRC} we show that having the same Lascar type equals having the same types. We also show that if $a$ and $b$ have the same Lascar type, then the Lascar distance between $a$ and $b$ is less or equal to two. 
We see in Remark \ref{IndTheoNTP2} that in the case of PRC bounded fields the Amalgamation Theorem (Theorem \ref{thamalgamation}) implies the independence theorem for NTP$_2$ theories showed by Chernikov and Ben Yaacov in \cite{ChBY}.

In section \ref{PreliminariesPpC} are the preliminaries on $p$-adically closed fields, and pseudo $p$-adically closed fields.
In sections \ref{SectionPpC}, \ref{IPppC} and \ref{PpCClasification} we show that some results obtained in section \ref{SectionPRC} and \ref{PRCstability}, can be generalized easily to bounded P$p$C fields.

%% file: PreliminariesPRC.tex
\section{Preliminaries on pseudo real closed fields} \label{PreliminariesPRC}

In this section we will give all the preliminaries that are required throughout the paper. We assume that the reader is familiar with basic concepts in model theory and algebra. 

\begin{para} \textbf{Notations and Conventions.}\label{Notandconv}
Let $T$ be a theory in a language $\mathcal{L}$, $M$ a model of $T$ and $\phi(\bar{x}, \bar{y})$ an $\mathcal{L}$-formula. For $A \subseteq M$, $\mathcal{L}(A)$ denotes the set of $\mathcal{L}$-formulas with parameters in $A$. If $\bar{a}$ is a tuple in $M$, we denote by $\phi(M, \bar{a}):= \{\bar{b} \in M^{|\bar{x}|}: M \models \phi(\bar{b}, \bar{a})\}$  and by $\tp_{\mathcal{L}}^M(a/A)$ $(\qftp_{\mathcal{L}}^M(a/A))$ the set of $\mathcal{L}(A)$-formulas (quantifier-free $\mathcal{L}(A)$-formulas) $\varphi$, such that $M \models \varphi(a)$. 
Suppose $M$ and $N$ are $\mathcal{L}$-structures and $A \subseteq M,N$. We denote by $M\equiv_A N$ if $M$ is $\mathcal{L}(A)$-elementarily equivalent to $N$.
Denote by $\acl^M_{\mathcal{L}}$ and $\dcl^M_\mathcal{L}$ the model theoretic algebraic and definable closures in $M$. We omit $M$ or $\mathcal{L}$ when the structure or the language is clear.

We denote by $\mathcal{L}_{\mathcal{R}}:= \{+,-,\cdot,0,1\}$ the language of rings.
All fields considered will have characteristic zero.
If $M$ is a field, we denote by $M^{alg}$ its algebraic closure, by $\G(M):= \Gal(M^{alg}/M)$ the absolute Galois group of $M$.  

\end{para}

\subsection{Ordered fields} \label{Orderedfields}\label{PrRCF}

A field $M$ is called \emph{formally real} or just \emph{real} if $M$ can be ordered.
An ordered field $(M,<)$ is \emph{real closed} if it has no proper ordered algebraic extension. Every ordered field $(M,<)$ has an algebraic extension $(\overline{M}^r, \overline{<}^r)$ which is real closed, and unique up to isomorphism over $M$.  We call this extension the \emph{real closure} of $(M,<)$. The absolute Galois group of $\overline{M}^r$, $G(\overline{M}^r)$ is cyclic of order $2$. 
Conversely, given an involution $\sigma$ in $G(M)$, its fixed field $Fix_{M^{alg}}(\sigma)$ is a real closed field and it has a unique ordering $<$ for which the positive elements are exactly the non-zero squares. 
We refer to the restriction of this ordering to $M$ as: \emph{the ordering of $M$ induced by $\sigma$}. 
If $\tau$ and $\sigma$ in $\G(M)$  induce the same ordering on $M$, then $\sigma$ and $\tau$ are conjugates in $\G(M)$ and $Fix_{M^{alg}}(\sigma)$ and $Fix_{M^{alg}}(\tau)$ are isomorphic over $M$.

A field extension $N/M$ is called \emph{totally real} if each order on $M$  extends to some order on $N$.

\begin{para}\textbf{Amalgamation theorem for ordered fields:}\label{Amalord}
If $(M_1,<_1)$ and $(M_2, <_2)$ are extensions of an ordered field $(M,<)$, and if $M_1$, $M_2$ are linearly disjoint over $M$, then $M_1M_2$ has an ordering $<_3$ that extends both $<_1$ and $<_2$.
Moreover, if $M_1$ and $M_2$ are algebraic over $M$, then the extension $<_3$ is unique. A proof of the existence part is given by van den Dries in \cite{Van} (Lemma 2.5) and the uniqueness part is shown in section 1 of \cite{J1}.
\end{para}

\subsection{Pseudo real closed fields}\label{PrPRC}

\noindent \textbf{Regular extensions:} Let $N/M$ be a field extension. We say that $N$ is a \emph{regular extension} of $M$ if $N/M$ is separable and the restriction map: $\G(N)\rightarrow \G(M)$ is onto.  In characteristic $0$ this is equivalent to $N \cap M^{alg}=M$.   

\begin{fact} \label{PRC} \cite[Theorem 1.2]{Pre}
For a field $M$ the following are equivalent:
\begin{enumerate}
	\item $M$ is existentially closed (relative to $\mathcal{L_R}$) in every totally real regular extension $N$ of $M$.
	\item For every absolutely irreducible variety $V$ defined over $M$, if $V$ has a simple 
	 $\overline{M}^r-$rational point for every real closure $\overline{M}^r$ of $M$, then $V$ has an $M$-rational point.
\end{enumerate}
\end{fact}

Prestel showed in Theorem 1.7 of \cite{Pre} that if $M$ admits only a finite number of orderings, (1) of Fact \ref{PRC} implies that $M$ is existentially closed in $N$ in the language $\mathcal{L_R}$ augmented by predicates for each order $<$ of $M$.

\begin{defn} \label{defnPRC}
A field $M$ of characteristic $0$ satisfying the conditions of Fact \ref{PRC} is called \emph{pseudo real closed} (PRC). 
By Theorem 4.1 of \cite{Pre} we can axiomatize the class of PRC fields in  $\mathcal{L_R}$.
By Theorem 11.5.1 of \cite{EJ} PAC fields cannot be ordered. So in particular PAC fields of characteristic 0 are PRC fields.  
Observe also that the class of PRC fields contains  the class of real closed fields.
\end{defn}

\begin{fact}\label{PRCcaracte}
Let $M$ be a PRC field.
\begin{enumerate}
	\item If $<$ is an order on $M$, then $M$ is dense in $(\overline{M}^r, \overline{<}^r)$ (\cite[Proposition 1.4]{Pre}).
	\item If $<_i$ and $<_j$ are different orders on $M$, then $<_i$ and $<_j$ induce different topologies (\cite[Proposition 1.6]{Pre}).
	\item If $L$ is an algebraic extension of $M$, then $L$ is PRC (\cite[Theorem 3.1]{Pre}).
	\end{enumerate}
\end{fact}

\begin{lem}\label{PRCacl}
Let $M$ be a PRC field and $A \subset M$. Then $A^{alg}\cap M = \acl^{M}(A)$.
If in addition $M$ has an $\mathcal{L_R}$-definable order, then $\acl^{M}(A)= \dcl^{M}(A)$.
\begin{proof} 
It is clear that if $M$ has a definable order, then $\acl^{M}(A)= \dcl^{M}(A)$.
Obviously $A^{alg}\cap M \subseteq \acl^{M}(A)$.
Let $A_0 = A^{alg} \cap M$ and let $\widetilde{M}$ be a copy of $M$ by an $\mathcal{L_R}(A_0)$-isomorphism $f$,  such that $M$ is linearly disjoint from $\widetilde{M}$ over $A_0$. 
As $M/A_0$ and $\widetilde{M}/A_0$ are regular and $M$ is linearly disjoint from $\widetilde{M}$ over $A_0$, we obtain that $M \widetilde{M}/M$ and $M\widetilde{M}/\widetilde{M}$  are regular.
By \ref{Amalord} $M \widetilde{M}$ is a totally real regular extension of $M$ and of $\widetilde{M}$.
Then $M$ and $\widetilde{M}$ are existentially closed in $M \widetilde{M}$, so there is an elementary extension $M^{*}$ of $M$ and of $\widetilde{M}$ such that $M \widetilde{M}\subseteq M^*$.

Let $\alpha \in \acl^{M}(A)$ and $\widetilde{\alpha} = f(\alpha)$. 
Since $M, \widetilde{M}  \prec M^{*}$, we deduce that $\tp^{M^{*}}(\alpha/A_0) = \tp^{M^{*}}(\widetilde{\alpha}/A_0)$ and $\acl^{M^{*}}(A) \subseteq M$.
Then $\widetilde{\alpha} \in \acl^{M^{*}}(A) \subseteq M$, since $\alpha \in \acl^{M}(A)$ . 
It follows that $\widetilde{\alpha} \in A_0$, then $\alpha = \widetilde{\alpha}$, so $\alpha \in A^{alg} \cap M$. 
Therefore $\acl^{M}(A) \subseteq A^{alg}\cap M$. 
\end{proof}
\end{lem}

Thanks to the last lemma we get easily the exchange principle and we have a good notion of dimension given by the algebraic closure.

\subsection{The theory of PRC fields with $n$ orderings}

\begin{defn}
Let $M$ be a field, $n \geq 1$, and $<_1, \ldots, <_n$ be $n$ orderings on $M$. We call the structure $(M, <_1, \ldots, <_n)$  an \emph{$n$-fold ordered field}. 

An $n$-fold ordered field $(M,<_1, \ldots, <_n)$ is $n$-\emph{pseudo real closed} ($n$-PRC) if:
\begin{enumerate}
 \item $M$ is a PRC field,
 \item if $i \not = j$, then $<_i$ and $<_j$ are different orders on $M$,
 \item $<_1, \ldots, <_n$ are the only orderings on $M$.
\end{enumerate}
Observe that an $n$-PRC field with $n=0$ is a PAC field.
\end{defn}

\begin{fact} \label{simpoint}\cite[Proposition 1.4]{J2}
Let $(M, <_1, \ldots, <_n)$ be an $n$-PRC field and let $V$ be an absolutely irreducible variety defined over $M$.
Denote by $\clos{M}{i}$ a fixed real closure of $M$ with respect to $<_i$. For every $1 \leq i \leq n$ take $q_i \in V(\clos{M}{i})$ a simple point. For each $i \in \{1, \ldots, n\}$, let $U_i$ be an $<_i$-open set such that $q_i \in U_i$.
Then $V$ has an $M$-rational point $q \in \displaystyle{\bigcap_{i=1}^n U_i}$
\end{fact}

\begin{fact} \label{embedding2}\cite[Theorem 3.2]{J1}
Let $(M, <_1, \ldots, <_n)$ and $(N, <_1', \ldots, <'_n)$ be two $n$-PRC fields. Let $\xi_i$ and $\sigma_i$ be involutions in $\G(M)$ and $\G(N)$ that induce $<_i$ and $<'_i$ on $M$ and $N$ respectively. Let $L$ be a common subfield of $M$ and $N$. Suppose further that there exists an isomorphism $\varphi: \G(N) \rightarrow \G(M)$ such that:
\begin{enumerate}
	\item[(a)] $\varphi(\sigma)|_{L^{alg}}= \sigma|_{L^{alg}}$ for every $\sigma \in \G(N)$.
	\item[(b)] $\varphi(\sigma_i)= \xi_i$ for $i\in \{1, \ldots, n\}$. 
\end{enumerate}
Then $(M, <_1, \ldots, <_n) \equiv_L (N, <'_1, \ldots, <'_n)$.
\end{fact}

\begin{cor}\label{EqelemPRC} \cite[Corollary 3.3]{J1}
Let $(M, <_1, \ldots, <_n) \subseteq (N, <'_1, \ldots, <'_n)$ two $n$-PRC fields. If the restriction map $\G(N)\rightarrow \G(M)$ is an isomorphism, then $(M, <_1, \ldots, <_n) \prec (N, <'_1, \ldots, <'_n)$.
\end{cor}

%% file: BoundedPRCfields.tex
\section{Bounded pseudo real closed fields} \label{SectionPRC}

In this section we study some model theoretic properties of bounded  PRC fields. In section \ref{DensityTheorem} we give a useful description of definable sets, this description is in some way a generalization to multi-ordered fields of cellular decomposition for real closed fields. In \ref{AmalgamationTheorems} we show some results about amalgamation of types. 

\begin{para}\label{ApTh} 
\textbf{Approximation Theorem} \cite[Theorem 4.1]{PRZI} Let $F$ be a field  and $\tau_1, \ldots, \tau_n$ be different topologies on $F$ which are induced by orders or non trivial valuations. For each $i \in \{1,\ldots,n\}$, let $U_i$  be a non-empty $\tau_i$-open subset of $F$.
Then $\displaystyle{\bigcap_{i=1}^n U_i \not = \emptyset}$.
\end{para}

\begin{rem}\label{bddfiniteorders}
If $M$ is a bounded PRC field, then $M$ has only finitely many orders.
\begin{proof}
Let $m \in \mathbb{N}$ be such that $M$ has exactly $m$ extensions of degree $2$. 
Suppose by contradiction that there exists $k > m$ and $\{<_i\}_{1 \leq i \leq k}$, distinct orderings on $M$. For each $1 \leq i \leq k$ choose $a_i$ such that $a_i>_i 0$, $a_i<_j 0$ for all $j \not = i$ (they exist by  \ref{ApTh} and Fact \ref{PRCcaracte}). Then the extensions $M(\sqrt{a_i})$, $1 \leq i \leq k$, are proper and linearly disjoint over $M$. Indeed, $\sqrt{a_i}$ belongs to the real closure $\clos{M}{i}$ of $M$ with respect to $<_i$, but does not belong to $\clos{M}{j}$ for $j \not = i$. This contradicts the fact that $M$ has exactly $m$ extensions of degree $2$.
\end{proof}
\end{rem}

\begin{para} \textbf{Notation \& Setting.}\label{PRCB}
In this section we fix a bounded PRC field $K$, which is not real closed and a countable elementary substructure $K_0$ of $K$.
By Lemma 1.22 of \cite{Cha0} the restriction map: $\G(K)\rightarrow \G(K_0)$ is an isomorphism, and $K_0^{alg}K= K^{alg}$.

By Remark \ref{bddfiniteorders} there exists $n \in \mathbb{N}$ such that $K$ has exactly $n$ distinct orders.
So $K$ is an $n$-PRC field. 
In this section we will work over $K_0$, thus we denote by $\mathcal{L}$ the language of rings with constant symbols for the elements of $K_0$, $\Li:= \mathcal{L} \cup \{<_i\}$ and $\mathcal{L}_n:= \mathcal{L} \cup \{<_1, \ldots, <_n\}$.
We let $\PRCB:=Th_{\LC}(K)$.
If $M$ is a model of $T$, we denote by $\clos{M}{i}$ the real closure of $M$ with respect to $<_i$.

Observe that if $n= 0$, $K$ is PAC. In this case the properties of $\PRCB$ are well known: by Corollary 4.8 of \cite{ChPi} $\PRCB$ is simple and by Corollary 3.1 of \cite{Hrus} $\PRCB$ has elimination of imaginaries. \bf{Therefore we will suppose always that $n \geq 1$}.
\end{para}

\begin{para}\label{MCPRC}
Observe that $\PRCB$ is model complete:  Let $M_1, M_2 \models \PRCB$ such that $M_1 \subseteq M_2$. 
As the restriction maps $\G(M_1)\rightarrow \G(K_0)$ and $\G(M_2)\rightarrow \G(K_0)$ are isomorphisms, it follows that the restriction map $\G(M_2)\rightarrow \G(M_1)$ is also an isomorphism. 
Then by Corollary $\ref{EqelemPRC}$ $M_1 \prec M_2$.
\end{para}

\begin{lem} \label{Deforders}
Let $(M,<_1, \ldots, <_n)$ be a model of  $\PRCB$. Then for all $i \in \{1, \ldots, n\}$ we can define the order $<_i$ by an existential $\LCR$-formula.
\begin{proof}
Let $i \in \{1, \ldots, n\}$. Then $M\models a>_i0$ if and only if $\clos{M}{i} \models \exists \alpha (\alpha^2 = a \wedge \alpha \not = 0)$.
Let $\sigma_i \in \Gal(M^{alg}/\clos{M}{i})$, $\sigma_i \neq id$. 
Define $M_2$ as the composite field of all the extensions of $M$ of degree $2$ and let $\widetilde\sigma_i = \sigma_i|_{M_2}$.  
In $M$ we can interpret without quantifiers in the language $\LCR$ the structure $({M}_2, +, \cdot,  \widetilde{\sigma_{1}}, \ldots, \widetilde{\sigma}_n)$, with the action of the automorphism $\widetilde{\sigma_i}$, for all $1 \leq i \leq n$. 
The reader can refer to Appendix 1 of $\cite{Cha}$ for more details.
Therefore we can define the formula ``$a >_i 0$" as follows:

\[M\models a >_i 0 \;\mbox{if and only if} \; M_2 \models \displaystyle{\exists \alpha (\alpha \not = 0 \wedge \widetilde{\sigma_i}(\alpha)=\alpha \wedge {\alpha}^2 = a}).\]
\end{proof}
\end{lem}

Observe that Lemma \ref{Deforders} implies that if $A \subseteq M$ and $a$ is a tuple in $M$, then $\tp_{\LCR}(a/A) \vdash \tp_{\LC}(a/A)$ and $ \acl_{\LCR}(A) =\acl_{\LC}(A) = \dcl_{\LC}(A)= \dcl_{\LCR}(A)$. 

\begin{cor}\label{PRCRmc}
$Th_{\LCR}(K)$ is model complete.
\begin{proof}
Use \ref{MCPRC} and observe that: $x<_i y \leftrightarrow (y-x)>_i 0$, and $\neg(x <_i y) \leftrightarrow y <_i x \vee y = x$.
 
\end{proof}

\end{cor}

\begin{para}\label{typePrc2}
\textbf{Types.} By Fact \ref{embedding2} we can describe the types in $\PRCB$ in a simple form:
Let $M$ be a model of $\PRCB$, $A$ a subfield of $M$ (containing $K_0$) and $a$ and $b$ tuples from $M$. As $K_0^{alg}M=M^{alg}$ and $K_0 \subseteq A$, we obtain that $(A(a))^{alg}= A^{alg}\acl(A(a))$ and $(A(b))^{alg}= A^{alg}\acl(A(b))$. 
It follows that $\tp^M(a/A)=\tp^M(b/A)$ if and only if there is an $\LCR$-isomorphism $\varphi$ between $\acl(A(a))$ and $\acl(A(b))$, which sends $a$ to $b$ and is the identity on $A$. 
\end{para}

\subsection{Density theorem for PRC bounded fields}  \label{DensityTheorem}
\begin{lem} \label{simple}(Folklore)
Let $F$ be a large algebraically closed field, let $d \in \mathbb{N}$, denote by $\mathbb{A}^d$ the $d$-dimensional affine space. 
Let $M$ be a small subfield of $F$, let $W \subseteq \mathbb{A}^d$  be a Zariski closed set defined over $M$. Then there exist $m \in \mathbb{N}$ and absolutely irreducible varieties $W_1, \ldots, W_m$ defined over $M$ such that:
\[ \bar{x} \in W(M) \; \;  \mbox{if and only if} \;  \; \bar{x} \in  \bigcup_ {j \in \{1, \ldots, m\}} W_j^{sim}(M),\]  
where  $W_j^{sim}(M):= \{\bar{x} \in W(M): \bar{x} \; \mbox{is a simple point of } \; W  \}$. 
\begin{proof}

We can suppose that $W(M)$ is Zariski dense in $W$ since if $V = \overline{W(M)}^{z}$ is the Zariski closure of $W(M)$, then $V$ is defined over $M$ and $V(M)=W(M)$. 
Let $W_1, \ldots, W_m$ be the absolutely irreducible components of $W$. 
Then $W(M) = \displaystyle \bigcup_{j=1}^m{W_j(M)}$, and $W_j= \overline{W_j(M)}^z$.

Any $M$-automorphism of $F$ will fix $W_j(M)$ pointwise for all $j \in \{1, \ldots, m\}$, hence fix $\overline{W_j(M)}^z$ setwise. Therefore for all $\sigma \in Aut(F/M)$,  $\sigma(W_j)=\sigma(\overline{W_j(M)}^z)= \overline{W_j(M)}^z = W_j$, so $W_j$ is defined over $M$ for all $j \in \{1, \ldots, m\}$.  

Let $W_j^{sing}(M) = \{\bar{x} \in W_j(M): \bar{x} \; \mbox{is a singular point of } \; W_j  \}$. 
Fix $j \in \{1, \ldots, m \}$, we have that $W_j(M) =W_j^{sim}(M)  \cup W_j^{sing}(M)$, $W_j^{sing}(M)$ is a closed set in the Zariski topology and its dimension is less than the dimension of $W_j(M)$.

The result follows by induction on the dimension.
  
\end{proof}
\end{lem}

\begin{rem}\label{simplerem}
Note that if $W$ is defined over $A = A^{alg}\cap M$, then the $W_j$ are defined over $A$ for all $j \in \{1, \ldots, m\}$: The absolutely irreducible components $W_j$ of $W$ are defined over $A^{alg}$, and by Lemma \ref{simple} are also defined over $M$, hence the $W_j$ are defined over $A^{alg}\cap M=  A$.
\end{rem}

\subsubsection{Density theorem for one variable definable sets} \label{SDT1v}

\begin{defn}\label{defmultiinterval}
Let $(M, <_1, \ldots, <_n)$ be a model of $\PRCB$ (see \ref{PRCB}).
\begin{enumerate}
\item A subset of $M$ of the form $I= \displaystyle{\bigcap_{i=1}^n (I^i\cap M)}$ with $I^i$ a non-empty $<_i$-open interval in $\clos{M}{i}$ is called a \emph{multi-interval}.
Observe that by  \ref{ApTh} (Approximation Theorem) and Fact \ref{PRCcaracte} every multi-interval is non empty.

\item A definable subset $S$ of a multi-interval $I = \displaystyle{\bigcap_{i=1}^n (I^i\cap M)}$ is called \emph{multi-dense} in $I$ if for any multi-interval $J \subseteq I$, $J \cap S \not = \emptyset.$
Note that multi-density implies $<_i$- density in $I^i$, for all $i \in \{1, \ldots,n\}$.

\end{enumerate}

\end{defn}

\begin{rem}\label{definM}
Let $(M, <_1, \ldots, <_n)$ be a model of $\PRCB$. Let $i \in \{1, \ldots,n\}$ and $a \in \clos{M}{i} \setminus M$ such that $a \in \acl^{\clos{M}{i}}(c)$, with $c$ a tuple in $M$. Then $A= \{x \in M: x <_i a\}$ is definable  in $M$ by a quantifier-free $\Li(c)$-formula.
\begin{proof}
By quantifier elimination of the theory of real closed fields (RCF) and the fact that $\acl^{\clos{M}{i}}(c)= \dcl^{\clos{M}{i}}(c)$, we can find a quantifier-free $\Li$-formula $\phi(x,c)$, such that $\clos{M}{i}\models \forall x (x <_i a \leftrightarrow \phi(x,c))$.
Then $x \in A$ if and only if $M\models \phi(x,c)$.
 \end{proof}

\begin{prop} \label{thmdensite}
Let $(M, <_1, \ldots, <_n)$ be a model of $\PRCB$. 
Let $\phi(x, \bar{y})$ be an $\LC$-formula, $\bar{a}$ a tuple in $M$ and $b \in M$ such that $M\models \phi(b, \bar{a})$ and $b \notin \acl(\bar{a})$.
Then there is a multi-interval $I= \displaystyle{\bigcap_{i=1}^n (I^i\cap M)}$ such that:

\begin{enumerate}
	\item $b \in I$,
	\item $\{x \in I: M \models \phi(x, \bar{a})\}$ is multi-dense in $I$,
	\item $I^i \subseteq \clos{M}{i}$ has its extremities in $\dcl^{\clos{M}{i}}_{\Li} (\bar{a})\cup \{\pm \infty\}$, for all $1 \leq i \leq n$, 
        \item the set $I^i \cap M$ is definable in $M$ by a quantifier-free $\Li(\bar{a})$-formula, for all $1 \leq i \leq n$.
\end{enumerate}   

\begin{proof}

By Corollary \ref{PRCRmc} there exists a quantifier-free $\LCR(\bar{a})$-formula $\psi(x, \bar{y})$  such that $M \models\forall x( \phi(x, \bar{a}) \leftrightarrow \exists\bar{y}\psi(x, \bar{y}))$. 

As we can define the relation $\neq$ with an existential formula, we can suppose that $\psi(x, \bar{y})$ is a positive formula, i.e defines an algebraic set $W$ defined over $\acl(\bar{a})$. 

Then $M\models \phi(x, \bar{a})$ is equivalent to $\exists \bar{y} \;(x, \bar{y})\in W(M)$.

Let $d$ be the arity of $\bar{y}$. 
As $M\models \phi(b, \bar{a})$ we can find $\bar{y_0} \in M^{d}$ such that $(b,\bar{y_0}) \in W(M)$. 
By Lemma \ref{simple} and Remark \ref{simplerem}, there exists an absolutely irreducible variety $V$ defined over $\acl(\bar{a})$ such that $(b, \bar{y_0})$ is a simple point of $V$ and $V(M)\subseteq W(M)$. 

For each $i \in \{1, \ldots,n\}$ we define:
\[A_i := \{x \in \clos{M}{i}: \exists(y_1, \ldots, y_{d}) \in {(\clos{M}{i})}^d (x, y_1 \ldots, y_{d}) \; \mbox{is a simple point of} \; V \}.\]

Observe that $b \in A_i$ and that $A_i$ is $\Li$-definable in $\clos{M}{i}$ with parameters in $\acl(\bar{a})$. 
By o-minimality of $\clos{M}{i}$ there is an $<_i$-interval $I^i \subseteq \clos{M}{i}$, with extremities in $\dcl^{\clos{M}{i}}(\bar{a}) \cup \{\pm \infty\}$, such that $b \in I^i$ and $I^i \subseteq A_i$. 
By Remark \ref{definM}, $I^i \cap M$ is definable in $M$ by a quantifier-free $\Li$-formula. 

Define $I := \displaystyle{\bigcap_{i=1}^n}(I^i\cap M)$ and $S:= \displaystyle{\{x \in I: M \models \phi(x, \bar{a})\}}.$

\begin{claim} $S$ is multi-dense in $I$:
\begin{proof}
Let $J \subseteq I$ be a non-empty multi-interval; we need to show that $J \cap S \not = \emptyset$.
Let $z \in J$; since $z \in A_i$ for all $i \in \{1, \ldots, n\}$, there are $y^{(i)} \in (\clos{M}{i})^d$ such that each $q_i:= (z, y^{(i)})$ is a simple point of $V$.
By Fact $\ref{simpoint}$ we can find $q_0:= (z_0, y_0) \in V(M)$ such that $q_0$ is arbitrary $<_i$-close to $q_i$ for all $i\in \{1, \ldots, n\}$.
In particular we can find $z_0 \in J$.
Then we obtain that $\exists \bar{y} \; (z_0,\bar{y}) \in V(M)$, and then $M \models \phi(z_0, \bar{a})$. 
\end{proof}
\end{claim}

\end{proof}
\end{prop}

\begin{thm} \label{descomposition}
Let $(M, <_1, \ldots, <_n)$ be a model of $\PRCB$,  let $\phi(x, \bar{y})$ be an $\LC$-formula and let $\bar{a}$ be a tuple in $M$.
Then there are a finite set $A\subseteq \phi(M, \bar{a})$, $m \in \mathbb{N}$ and $I_1, \ldots, I_m$, with $I_j = \displaystyle{\bigcap_{i=1}^n (I^i_{j}\cap M)}$ a multi-interval such that:
\begin{enumerate}
\item $A \subseteq \acl(\bar{a})$,
\item $\phi(M, \bar{a}) \subseteq \displaystyle{\bigcup_{j=1}^m I_ j \cup A},$
\item$\{x\in I_j:M \models \phi(x, \bar{a})\}$ is multi-dense in $I_{j}$, for all  $1 \leq j \leq  m$,
\item $I^i_{j} \subseteq \clos{M}{i}$ has its extremities in $\dcl^{\clos{M}{i}}_{\Li} (\bar{a})\cup \{\pm \infty\}$, for all  $1 \leq j \leq  m$  and $1 \leq i \leq n$, 
\item the set $I^i_{j}\cap M$ is definable in $M$ by a quantifier-free $\Li(\bar{a})$-formula, for all $1 \leq j \leq  m$ and $1 \leq i \leq n$.

 \end{enumerate}
\end{thm}
\begin{proof}

As in Proposition \ref{thmdensite} using Corollary \ref{PRCRmc} and Lemma \ref{simple}, there are  $r \in \mathbb{N}$ and  absolutely irreducible varieties $W_1, \ldots, W_r$ defined over $\acl(\bar{a})$ such that:

\[M \models \forall x (\phi(x, \bar{a}) \leftrightarrow  \exists \bar{y} \; \;  (x, \bar{y}) \in \bigcup_ {j \in \{1, \ldots, r\}} W_j^{sim}(M) ).\]

Working with each $W_j$ separately, we can suppose that there is an absolutely irreducible variety  $W$ defined over $\acl(\bar{a})$ such that:
\[M \models \forall x (\phi(x, \bar{a}) \longleftrightarrow \exists \bar{y}(x, \bar{y}) \in W^{sim}(M)).\]

Let $d= |\bar{y}|$, for each $i\in \{1, \ldots,n\}$ define:
\[A_i := \{x \in \clos{M}{i}: \exists\bar{y} \in {(\clos{M}{i})}^d (x,\bar{y}) \; \mbox{is a simple point of} \; W \}.\]

By o-minimality of $\clos{M}{i}$ there are $r_i \in \mathbb{N}$, $I^i_{1}, \ldots, I^i_{r_i}$ $<_i$-open intervals in $\clos{M}{i}$ with extremities in $\dcl^{\clos{M}{i}}_{\Li}(\bar{a}) \cup \{\pm \infty\}$, and a finite set $C_i$ such that $A_i =\displaystyle{\bigcup_{j=1}^{r_i}I^i_{j}}\cup C_i$ and $C_i \subseteq{\acl(\bar{a})} $. By Remark \ref{definM}, the set $I^{i}_j \cap M$, with $1 \leq i \leq n$, $1 \leq j \leq r_i$, is definable in $M$ by a quantifier-free $\Li(\bar{a})$-formula.

By the proof of Proposition \ref{thmdensite}, $\phi(M, \bar{a}) \subseteq \displaystyle{\bigcup_{\sigma \in J} \bigcap_{i=1}^n (I^i_{\sigma(i)}\cap M) \cup \bigcup_{i=1}^n (C_i\cap M)}$, where $J = \{\sigma:\{1, \ldots, n\} \rightarrow \{1, \ldots, \max\{r_1 ,\ldots, r_n\}\}, \sigma(i)\leq r_i\}$.

For each $\sigma \in J$, define $I_{\sigma}:= \displaystyle {\bigcap_{i=1}^n (I^i_{\sigma(i)}\cap M)}$ and $S_{\sigma}:=  \{x \in  I_{\sigma}:M \models \phi(x,\bar{a})\}$.
\begin{claim}
For all $\sigma \in J$, $S_{\sigma}$ is multi-dense in $I_{\sigma}:$
\begin{proof}
Fix $\sigma \in J$. Let $U_{\sigma}$ be a multi-interval such that $U_{\sigma} \subseteq I_{\sigma}$, we need to show that $U_{\sigma} \cap S_{\sigma} \not = \emptyset.$
Let $z \in U_{\sigma}$. Then $z \in A_i$ for all $i \in \{1, \ldots, n\}$. 
So there is $y^{(i)} \in (\clos{M}{i})^{d}$, such that $q_i:=(z, y^{(i)})$ is a simple point of $W$. 
By Fact $\ref{simpoint}$ we can find $q_0:= (z_0, \bar{y_0}) \in W(M)$ such that $q_0$ is arbitrary $<_i$-close to $q_i$ for all $i\in \{1, \ldots, n\}$, in particular we can find $z_0 \in U_{\sigma}$  which satisfy $\phi$.
\end{proof}
\end{claim}
\end{proof}

\end{rem}

\subsubsection{Density theorem for several variable definable sets.} \label{SDTPv}

The proof of Proposition \ref{thmdensite} and Theorem \ref{descomposition2} can be easily generalized to several variables.
We assume the reader is familiar with the concept of cells and cell decomposition in o-minimal theories. 
See chapter 3 of \cite{Van2} for more details.

\begin{defn}
If $(M,<)$ is an ordered field, and $r \in \mathbb{N}$, then a \emph{box in $M^r$} is a set of the form $I_1 \times\ldots \times I_r$, where $I_j$ is an non-empty $<$-open interval, for all  $j\in \{1, \ldots, r\}$.
\end{defn}

\begin{rem}\label{ApThC}
Let $F$ be a field and $\tau_1, \ldots, \tau_n$ distinct topologies on $F$ induced by orders or valuations. Let $r \in \mathbb{N}$.
For each $i \in \{1, \ldots,n\}$, let $U^i$ be a non-empty $\tau_i$-open set in $F^r$ (endowed with the product topology). 
Then $\displaystyle \bigcap_{i=1}^n U^i \not = \emptyset.$
\begin{proof}
Let $i \in \{1, \ldots,n\}$. 
As $U^i$ is $\tau_i$-open in $F^r$ there exist $I^i_1, \ldots, I^i_r$ non-empty $\tau_i$-open subsets of $F$ such that $I^i_1 \times \ldots \times I^i_r \subseteq U^i$.
By the Approximation Theorem (\ref{ApTh}) for all $t \in \{1, \ldots,r\}$, $\displaystyle{\bigcap_{i=1}^nI_t^i} \not = \emptyset$.
If $V_t= \displaystyle{\bigcap_{i=1}^nI_t^i}$, then $\emptyset \not =V_1 \times \ldots \times V_r \subseteq \displaystyle \bigcap_{i=1}^n U^i.$  
\end{proof}
\end{rem}

\begin{defn}\label{defmulticell}
Let $(M, <_1, \ldots, <_n)$ be a model of $\PRCB$ and let $r \in \mathbb{N}$.
\begin{enumerate}
 \item A subset of $M^r$ of the form $C= \displaystyle{\bigcap_{i=1}^n (C^i\cap M^r)}$ with $C^i$ a non-empty $<_i$-open cell in $(\clos{M}{i})^r$ is called a \emph{multi-cell in $M^r$}. Observe that by Remark \ref{ApThC} every multi-cell is not empty.
 \item If $C= \displaystyle{\bigcap_{i=1}^n (C^i\cap M^r)}$ is a multi-cell in $M^r$, then $C$ is called a \emph{multi-box in $M^r$} if $C^i$ is a $<_i$-box in $(\clos{M}{i})^r$, for all $i \in \{1, \ldots,n\}$.
 \item A definable subset $S$ of a multi-cell $C = \displaystyle{\bigcap_{i=1}^n (C^i\cap M^r)}$ in $M^r$ is called \emph{multi-dense} in $C$ if for any multi-box $J \subseteq C$ in $M^r$, $J \cap S \not = \emptyset.$
Note that multi-density in $C$ implies $<_i$- density in $C^i$, for all $i \in \{1. \ldots,n\}$.
\end{enumerate}
\end{defn}

\begin{thm}\label{descomposition2} \label{thmdensite2}
Let $(M, <_1, \ldots, <_n)$ be a  model of  $\PRCB$ and let $r \in \mathbb{N}$.
Let $\phi(x_1, \ldots, x_r, \bar{a})$ be an $\LC$-formula.
Then there are a set $V$, $m \in \mathbb{N}$, and $C_1, \ldots, C_m$ with $C_j= \displaystyle{\bigcap_{i=1}^n (C^i_j\cap M^r)}$ a multi-cell in $M^r$ such that:

\begin{enumerate}
\item $\phi(M, \bar{a}) \subseteq \displaystyle{\bigcup_{j=1}^m C_ j \cup V }$,
\item the set $V$ is contained in some proper Zariski closed subset of $M^r$ which is definable over $\acl(\bar{a})$,
\item $\{(x_1, \ldots,x_r) \in C_j:M \models \phi(x_1, \ldots,x_r, \bar{a})\}$ is multi-dense in $C_j$, for all $1 \leq j \leq m$,
\item the set $C^i_j\cap M^r$ is definable in $M$ by a quantifier-free $\Li(\bar{a})$-formula, for all $1 \leq j \leq m$, $1 \leq i \leq n$.
 \end{enumerate}
 \begin{proof}
  As in the proof of Theorem \ref{descomposition} we can suppose that there is an absolutely irreducible variety $W$ defined over $\acl(\bar{a})$ such that:
\[M \models \forall (x_1, \ldots,x_r) (\phi(x_1, \ldots,x_r, \bar{a}) \longleftrightarrow \exists \bar{y}(x_1, \ldots,x_r, \bar{y}) \in W^{sim}(M)).\]

Let $d = |\bar{y}|$, for each $i \in \{1, \ldots,n\}$ we define 
\[A_i := \{(x_1, \ldots, x_r) \in (\clos{M}{i})^r: \exists(\bar{y}) \in {(\clos{M}{i})}^d (x_1, \ldots, x_{r}, \bar{y}) \; \mbox{is a simple point of} \; W \}.\]

Observe that $A_i$ is $\Li(\bar{a})$-definable. 
By cell decomposition in $\clos{M}{i}$, there are $r_i\in \mathbb{N}$, pairwise disjoint $<_i$-open cells $C^i_1, \ldots, C^i_{r_i}$, and proper Zariski closed subsets $V^i$ of $(\clos{M}{i})^r$ such that:
\begin{enumerate}
 \item $A_i =\displaystyle{\bigcup_{j=1}^{r_i}C^i_j}\cup V^i$,
\item $ \displaystyle{\bigcup_{j=1}^{r_i}C^i_j} \cap V^i = \emptyset$,
\item the sets $V^i,C^i_1, \ldots C^i_{r_i}$ are quantifier-free $\Li(\bar{a})$-definable.
 \end{enumerate}

Let $V:=\displaystyle{\bigcup_{i=1}^n (V^i\cap M^r)}$. Then $V$ is $\LC(\bar{a})$-definable in $M$ and $\dim(V)<r$.

Let $J = \{\sigma:\{1, \ldots, n\} \rightarrow \{1, \ldots, \max\{r_1 ,\ldots, r_n\}\}, \sigma(i)\leq r_i\}$. 
For each $\sigma \in J$ define $C_{\sigma}:= \displaystyle{\bigcap_{i=1}^n (C^i_{\sigma(i)}\cap M^r)}$. 
Then $\phi(M, \bar{a}) \subseteq \displaystyle{\bigcup_{\sigma \in J}C_{\sigma} \cup V}$.

Exactly as the proof in Theorem \ref{descomposition} we have that:

$\{(x_1, \ldots,x_r) \in C_{\sigma}: M \models \phi(x_1, \ldots,x_r,\bar{a}) \}$ is multi-dense in $C_{\sigma}$ for all $\sigma \in J$.
 \end{proof}
\end{thm}

\begin{lem}\label{lemqftdense}
Let $(M, <_1, \ldots, <_n)$ be a  model of  $\PRCB$.
Let $A \subseteq M$ and let $\bar{a}$ be a tuple of $M$ such that $trdeg(A(\bar{a})/A)= |\bar{a}|$.
For all $i \in \{1, \ldots,n\}$, let $\bar{b}_i \in M^{|\bar{a}|}$ be such that $\qftp_{\Li}(\bar{b_i}/A)= \qftp_{\Li}(\bar{a}/A)$, and let $U^i$ be a non-empty $<_i$-open set in $(\clos{M}{i})^{|\bar{a}|}$ such that $\bar{b_i} \in U^i$.
Then the type $p(\bar{x}):= \{\bar{x} \in \displaystyle{\bigcap_{i=1}^n U^i}\} \cup \tp_{\LC}(\bar{a}/A) \cup \{\bar{x} \not = \bar{a}\}$ is consistent. 
\begin{proof}
Let $U:= \displaystyle{\bigcap_{i=1}^n U^i}$.
By compactness it is enough to show that if $\psi(\bar{x}) \in \tp_{\LC}(\bar{a}/A)$, then $M \models \exists \bar{x}(\bar{x} \in U \wedge \psi(\bar{x})\wedge \bar{x} \not= \bar{a})$.
As $trdeg(A(\bar{a})/A)= |\bar{a}|$, by Proposition \ref{thmdensite2} there exists a multi-cell $C:= \displaystyle{\bigcap_{i=1}^n (C^i \cap M^{|\bar{a}|})}$ in $M^{|\bar{a}|}$ such that:
\begin{enumerate}
 \item $\bar{a} \in C$,
 \item $\{\bar{x} \in C: M \models \psi(\bar{x})\}$ is multi-dense in $C$,
 \item $C^i \cap M^{|\bar{a}|}$ is definable in $M$ by a quantifier free $\Li(A)$-formula, for all $i \in \{1, \ldots, n\}$.
\end{enumerate}

Let $i \in \{1, \ldots,n\}$. Since $\bar{a} \in C^i  \cap M^{|\bar{a}|}$, $C^i  \cap M^{|\bar{a}|}$ is quantifier free $\Li(A)$-definable and $\qftp_{\Li}(\bar{b_i}/A)= \qftp_{\Li}(\bar{a}/A)$, we deduce that $\bar{b_i} \in C^i \cap M^{|\bar{a}|}$.
Let $V^i:=  U^i \cap C^i \cap  M^{|\bar{a}|}$. Observe that $\bar{b_i} \in V^i$ and that $V^i$ is an $<_i$-open set in $M^{|\bar{a}|}$. Let $V:= \displaystyle{\bigcap_{i=1}^n{V^i}}$, by Remark \ref{ApThC} $V \not = \emptyset$. Since $V \subseteq C$ and every $V^i$ is $<_i$-open in $M^{|\bar{a}|}$, by multi-density of $\psi(\bar{x})$ in $C$ there exists $\bar{c} \in V$, $\bar{c} \not = \bar{a}$ such that $M \models \psi(\bar{c})$. Then $M \models \bar{c} \in U \wedge \psi({\bar{c}}) \wedge \bar{c} \not = \bar{a}$, since $V \subseteq U$.
\end{proof}
\end{lem}

\subsection{Amalgamation theorems for PRC bounded fields} \label{AmalgamationTheorems}

\begin{lem}\label{Calg}
Let $F_1$ and $F_2$ be regular extensions of a field $k$. Let $H\leq \G(F_1F_2)$ be a closed subgroup, let $\pi:\G(F_1F_2)\to \G(F_1)$ be the restriction map, and suppose that $\pi|_{H}$ is an isomorphism.
Let $\rho\in \G(F_1F_2)$ be such that $\rho$ fixes $F_1^{alg}\cap Fix(H)$ and assume that for every $\sigma\in H$, $\pi(\sigma)=\pi(\rho^{-1}\sigma\rho)$.
Then there is $\tilde{\rho}\in \G(F_1^{alg}F_2)$ such that $\tilde \rho^{-1}\sigma\tilde {\rho}=\rho^{-1}\sigma\rho$ for all $\sigma\in H$.
\begin{proof}
The condition $\pi(\sigma)=\pi(\rho^{-1}\sigma\rho)$ for all $\sigma\in H$ implies that $\pi(\rho)$ centralizes $\pi(H)$.
We have that $Fix(\pi(H))=Fix(H)\cap F_1^{alg}$ and that $Fix(\rho) \cap F_1^{alg}= Fix(\pi(\rho))$. By hypothesis $Fix(H)\cap F_1^{alg} \subseteq Fix(\rho) \cap{F_1^{alg}}$. So $\pi(\rho) \in \pi(H)$. 
Let $\tau\in H$ be such that $\pi(\rho)=\pi(\tau)$, and consider $\tilde \rho=\tau^{-1}\rho$. Since $\pi(\tau)$
centralizes $\pi(H)$ and $\pi$ is an isomorphism, $\tau\sigma\tau^{-1}= \sigma$ for all $\sigma \in H$, and therefore ${\tilde\rho}^{-1}\sigma\tilde{\rho} = \rho^{-1}\sigma\rho$, for all $\sigma \in H$ and clearly $\pi(\tilde{\rho})= 1.$
\end{proof}
\end{lem}

\begin{prop}\label{lemamalgamation}
 Let $(M, <_1, \ldots, <_n)$ be a model of $\PRCB$ and $E = \acl(E) \subseteq M$. Let $a_1, a_2, c_1,c_2$ be tuples of $M$ such that $E(a_1)^{alg}\cap E(a_2)^{alg}=E^{alg}$ and $\tp_{\LC}(c_1/E)=\tp_{\LC}(c_2/E)$. 
 Assume that there is $c$ realizing $\qftp_{\LC}(c_1/E(a_1)) \cup \qftp_{\LC}(c_2/E(a_2)) \cup \tp_{\LC}(c_1/E)$, such that $c$ is $ACF$-independent from $\{a_1,a_2\}$ over $E$.
Then $\tp_{\LC}(c_1/E(a_1)) \cup \tp_{\LC}(c/E(a_2)) \cup \qftp_{\LC}(c/E(a_1, a_2))$ is consistent. 
\begin{proof}

Take $c$ in some elementary extension $N$ of $M$ such that $c$ is $ACF$-independent from $\{a_1,a_2\}$ over $E$ and realizes $\qftp^N_{\LC}(c_1/E(a_1)) \cup \qftp^N_{\LC}(c_2/E(a_2)) \cup \tp^N_{\LC}(c_1/E)$.
For all $1 \leq i \leq n$, fix a real closure $\clos{N}{i}$ of $N$ for $<_i$. If $A \subset \clos{N}{i}$, then we set $\clos{A}{i}= A^{alg} \cap N^{(i)}$.

As $\qftp_{\LC}(c/E(a_1))= \qftp_{\LC}(c_1/E(a_1))$ there exists an $\LC$-isomorphism  
$\Phi:E(a_1,c_1) \rightarrow E(a_1, c)$, which fixes $E(a_1)$ and sends $c_1$ to $c$. 

We denote $C:= \acl(E(c))$, $A_1:=\acl(E(a_1))$, $A_2:=\acl(E(a_2))$ and $C_1:=\acl(E(c_1))$.  

\begin{claim1}
We can extend $\Phi$ to an $\LCR(A_1^{alg})$-isomorphism $\tilde{\Phi}: (A_1C_1)^{alg}\rightarrow (A_1C)^{alg}$ such that $\tilde{\Phi}|_{A_1C_1}$ is an $\LC$-isomorphism and $\tilde{\Phi}(C_1)=C$. 
\end{claim1}
\begin{proof} 
In $\clos{N}{i}$ $\acl_{\Li}=\dcl_{\Li}$, so we can extend $\Phi$ uniquely to an $\Li$-isomorphism $\Phi^{(i)}: \clos{E(a_1,c_1)}{i} \rightarrow \clos{E(a_1,c)}{i}$.
Note that for all $1 \leq i \leq n$, $\Phi^{(i)}$ is the identity on $\clos{A_1}{i}$.

Since $\tp_{\LC}(c_1/E)=\tp_{\LC}(c/E)$, there is an elementary $\LC$-isomorphism $\xi: C_1 \rightarrow C$ which fixes $E$ and sends $c_1$ to $c$. As above we can extend $\xi$ to an $\Li$-isomorphism $\xi^{(i)}: \clos{C_1}{i} \rightarrow \clos{C}{i}$.

As in $\clos{N}{i}$ $\acl_{\Li}=\dcl_{\Li}$, we get that $\Phi^{(i)}|_{\clos{C_1}{i}}=\xi^{(i)}$ for all $1 \leq i \leq n$. This implies that for all $1 \leq i \leq n$, $\Phi^{(i)}|_{C_1} = \xi$, it follows that $\Phi^{(i)}|_{A_1C_1} = \Phi^{(j)}|_{A_1C_1}$ for all $ i,j \in \{1, \ldots, n\}$.
Therefore $\Phi^{(i)}|_{A_1C_1}: A_1C_1 \rightarrow A_1C$ is an $\LC$-isomorphism and $\Phi^{(i)}(C_1)=C$.

Since $C_1/E$ and $C/E$ are regular extensions and $A_1$ is ACF-independent from $C_1$ and from $C$ over $E$, we have that $C_1A_1$ and $CA_1$ are linearly disjoint from $A_1^{alg}$ over $A_1$. So we can extend $\Phi^{(i)}|_{A_1C_1}$ to an $A_1^{alg}$-isomorphism $\tilde{\Phi}: (A_1C_1)^{alg} \rightarrow (A_1C)^{alg}$.
\end{proof}

Let $D:= \tilde{\Phi}(\acl(A_1C_1))\subseteq (A_1C)^{alg}$. 
The map $\Psi: \G(\acl(A_1C_1))\rightarrow \G(D)$ defined by $\Psi(\tau)= \tilde\Phi\tau\tilde\Phi^{-1}$, is an isomorphism inducing the identity on $\G(A_1)$.

Let  $\pi: \G(\acl(A_1C_1))\rightarrow \G(E)$ the restriction map.
As $\G(D) \cong \G(\acl(A_1C_1)) \cong \G(E)$, there exists $\theta: \G(E)\rightarrow \G(D)$ such that $\Psi = \theta \circ \pi$.

Let $K:= A_1C$, $F:= \acl(A_1A_2)\acl(A_2C)$, $\widetilde{K}:= A_1^{alg}C$, $\widetilde{F}:=(A_1A_2)^{alg}(A_2C)^{alg}$ and $\widetilde{D}:= (A_1C)^{alg}$. 

Define $S := \{(\sigma, \theta(\sigma|_{{E}^{alg}})): \sigma \in \Gal(\widetilde{F}/F)\}$.

\begin{claim2} $S$ is isomorphic to a subgroup $S'$ of $\Gal(\widetilde{F}\widetilde{D}/FD)$ which projects onto $\Gal(\widetilde{F}/F)$ and onto $\G(D)$.

\begin{proof}
By Lemma 2.5 (2) of \cite{Cha} $(A_1A_2)^{alg}(A_2C)^{alg}\cap (A_1C)^{alg}= (C^{alg}(A_1^{alg}\cap A_2^{alg}))A_1^{alg} = A_1^{alg}C^{alg}$.
Since $E^{alg}M = M^{alg}$ it follows that $E^{alg}C= C^{alg}$.
Therefore $\widetilde{F} \cap \widetilde{D} = A_1^{alg}C= \widetilde{K}$. As $\widetilde{D}/\widetilde{K}$ is a Galois extension, $\widetilde{F}$ and $\widetilde{D}$ are linearly disjoint over $\widetilde{K}$.

Since $F/A_1$ is a regular extension, $F$ is linearly disjoint from $A_1^{alg}$ over $A_1$. As $A_1 \subseteq K \subseteq F$, $F$ is linearly disjoint from $KA_1^{alg}$ over $K$. 
Since $KA_1^{alg}= A_1^{alg}C=\widetilde{K}$, $F$ is linearly disjoint from $\widetilde{K}$ over $K$.
Similarly, as $D/A_1$ is a regular extension and $A_1 \subseteq K \subseteq D$, we get that $D$ is linearly disjoint from $\widetilde{K}$ over $K$.

By Lemma 2.6 of [8] applied to $(K,F,D,\widetilde{K},\widetilde{F},\widetilde{D})$ in place of $(K,L,M,K_1,L_1,M_1)$ we obtain that:
$\Gal(\widetilde{F}\widetilde{D}/FD)\simeq \{(\sigma, \tau)\in \Gal(\widetilde{F}/F) \times \Gal(\widetilde{D}/D): \sigma|_{\widetilde{K}}= \tau|_{\widetilde{K}} \}$.

Let $\sigma \in \Gal(\widetilde{F}/F)$.
Since $\widetilde\Phi(x)=x$ for all $x \in A_1^{alg}$, and $\sigma$ and $\theta(\sigma|_{E^{alg}})$ are the identity on $C$, we deduce that $\sigma$ and $\theta(\sigma|_{E^{alg}})$ agree on $K_1$.
Then $S$ is isomorphic to a subgroup $S'$ of $\Gal(\widetilde{F}\widetilde{D}/FD)$.

\end{proof}

\end{claim2}

By claim $2$, $L= Fix(S')$ is a regular extension of $D$ and $F$.
As $N/\acl(A_1A_2)$ and $N/\acl(A_2C)$ are regular extensions and contain $F$, it follows that $L / \acl(A_1A_2)$ and $L/\acl(A_2C)$  are regular extensions.

\begin{claim3}
 Each of the $n$ orders on $F$ extends to an order on $L$.
\begin{proof}
 Fix $1 \leq i \leq n$ and let $H:= \G(\clos{N}{i}) \cong \G(\clos{F}{i})$.
It suffices to show that $L$ it is contained in $Fix(\widetilde{H})$ with $\widetilde{H}$ a conjugate of $H$.
Let $\pi_1$ and $\pi_2$ be the restriction maps of $\G(N)$ to $\G(A_1C_1)$ and $\G(A_1C)$ respectively.

Denote by $H_1:= \{\pi_1(\sigma): \sigma \in H\}$ and $H_2:= \{\pi_2(\sigma): \sigma \in H\}$.
Since $\widetilde{\Phi}|_{A_1C_1}$ is an $\LC(A_1)$-isomorphism and $\widetilde{\Phi}(C_1)=C$, it follows that $\Psi(H_1)$ and $H_2$ are conjugate in $\G(A_1C)$. In fact, $\widetilde\Phi(\clos{A}{i}_1 C_1) = \clos{A}{i}_1 C$, since $\widetilde{\Phi}|_{A_1^{alg}}= id$. 
So we can find $\rho \in \G(\clos{A}{i}_1 C)$ such that $\rho^{-1}\pi_2(\sigma)\rho = \Psi(\pi_1(\sigma)) =  \theta(\sigma|_{E^{alg}})$, for all $\sigma \in H$. Then $(\rho^{-1}\sigma \rho)|_{A_1^{alg}}= \sigma|_{A_1^{alg}}$, for all $\sigma \in H.$

Applying Lemma $\ref{Calg}$ with $F_1= A_1$ and $F_2=C$, we can suppose that $\rho \in \G(A_1^{alg}C)$.
By Lemma 2.5 (2) of \cite{Cha} we obtain that $(A_1A_2)^{alg}(A_2C)^{alg}$ and $(A_1C)^{alg}$ are linearly disjoint over $A_1^{alg}C$, therefore we may extend $\rho$ to an element $\tilde{\rho} \in \G((A_1A_2)^{alg}(A_2C)^{alg})$ (recall that $A_1^{alg}C \supset C^{alg}$).
Let  $\widetilde{H}:= \tilde{\rho}^{-1}H\tilde{\rho}$.

Then we have for all $\sigma \in H$ we have that:

\[\pi_1(\tilde{\rho}^{-1}\sigma\tilde{\rho})= \rho^{-1}\pi_1(\sigma)\rho = \Psi(\pi_2(\sigma)) = \theta(\sigma|_{E^{alg}}).\]

The image of $\widetilde{H}$ by the restriction map inside $\Gal((A_1A_2)^{alg}(A_2C)^{alg}(A_1C)^{alg}/A_1A_2C)$ is contained in $S$. Hence $Fix(\widetilde{H}) \supseteq Fix(S)=L$.
\end{proof}
\end{claim3}

Let $\widetilde{<_1}, \ldots, \widetilde{<_n}$ be $n$ orders on $L$ extending the $n$ orders on $F$. Since $L/\acl(A_1A_2)$ and $M/\acl(A_1A_2)$ are regular, we can suppose that they are linearly disjoint, so by Fact \ref{Amalord}, for all $1 \leq i \leq n$, $<_i$ and $\widetilde{<_i}$ have a common extension to an ordering on $LM$. Hence $LM$ is a totally real extension of $M$.

Since $LM/ M$ is regular, by Fact $\ref{PRC}$, $M$ is existentially closed in $LM$, so there is an elementary extension $M^*$ of $M$ such that $M \subseteq LM \subseteq M^*$.

\begin{claim4}
\begin{enumerate}
\item $D = (A_1C)^{alg} \cap M^* = \acl^{M^*}(A_1C)$
\item $\acl^{N}(A_1C_1) = (A_1C_1)^{alg} \cap N = (A_1C_1)^{alg} \cap M^* = \acl^{M^*}(A_1C_1)$
\item $\acl^{N}(A_2C) = (A_2C)^{alg} \cap N = (A_2C)^{alg} \cap M^* = \acl^{M^*}(A_2C)$
\end{enumerate}
\begin{proof} 
(1): We have that $D \subseteq D^{alg} \cap M^*$. Suppose that $D \not= D^{alg} \cap M^*$, and let $\alpha \in D^{alg} \cap M^*$, $\alpha \not\in D$. Since $D^{alg}= E^{alg}D$, there exists $\beta \in E^{alg}\setminus E$ such that $D(\beta)= D(\alpha) \subset M^*$, so $\beta \in M^*$ and this is a contradiction. 
Hence $D = D^{alg} \cap M^*$. As $D \subseteq (A_1C)^{alg}$ the claim follows.

(2) and (3): Clear because $M\prec N, M^{*}$ and by claim 2
\end{proof}
\end{claim4}

\begin{claim5}
We have that $c$ realizes $\tp_{\LC}^{M^*}(c_1/E(a_1)) \cup \tp_{\LC}^{M^*}(c/E(a_2)) \cup \qftp_{\LC}^{M^*}(c/E(a_1,a_2))$.
\end{claim5}
\begin{proof}

Define $\Phi':= \widetilde{\Phi}|_{\acl^N(A_1C_1)}$. Then $\Phi': \acl^N(A_1C_1)\rightarrow D$ satisfies that $\Phi'(c_1)=c$ and $\Phi'|_{A_1}= id$.  
By Claim 4 (1) and (2) $\Phi': \acl^{M^*}(A_1C_1) \rightarrow \acl^{M^*}(A_1C)$, therefore by \ref{typePrc2} $\tp^{M^*}(a_1/E(a_1))= \tp^{M^*}(c/E(a_1))$. Then  $c$ realizes $\tp^{M^{*}}(c_1/E(a_2))$.

\end{proof}
\end{proof}
\end{prop}

\begin{thm}\label{thamalgamation}
 Let $(M, <_1, \ldots, <_n)$ be a model of  $\PRCB$. Let $E = \acl(E) \subseteq M$. Let $a_1, a_2, c_1,c_2$ be tuples of $M$ such that $E(a_1)^{alg}\cap E(a_2)^{alg}=E^{alg}$ and $\tp_{\LC}(c_1/E)=\tp_{\LC}(c_2/E)$. Assume that there is $c$ $ACF$-independent of $\{a_1,a_2\}$ over $E$ realizing $\qftp_{\LC}(c_1/E(a_1)) \cup \qftp_{\LC}(c_2/E(a_2))$.
Then $\tp_{\LC}(c_1/Ea_1) \cup \tp_{\LC}(c_2/Ea_2) \cup \qftp_{\LC}(c/E(a_1,a_2))$ is consistent.
\begin{proof} 
Observe that $trdeg(E(c_j)/E) = trdeg(E(c)/E)= trdeg(E(c,a_j)/E(a_j))$, for $j =1,2$.
Remark that if $d \in E(c)^{alg}\cap M = \acl(Ec) = \dcl(Ec)$, then $\tp(d/Ec)$ is isolated by a quantifier-free $\Li$-formula.
Then we can suppose that $trdeg(E(c)/E)= |c|= |c_j|$.    

\begin{claim}
We can suppose that $\tp_{\LC}(c/E)= \tp_{\LC}(c_1/E)$.
\begin{proof}
Suppose that $M$ is sufficiently saturated. 
We need to show that $\qftp_{\LC}(c/E(a_1)) \cup \qftp_{\LC}(c/E(a_2)) \cup \tp_{\LC}(c_1/E)$ is realized by some $c^*$ $ACF$-independent of $\{a_1,a_2\}$ over $E$. 
By compactness it is enough to show that if $\psi_j(x,a_j) \in \qftp_{\LC}(c/E(a_j))$, for $j=1,2$ and $\phi(x) \in \tp_{\LC}(c_1/E)$, then $\psi_1(x,a_1) \wedge \psi_2(x,a_2) \wedge \phi(x)$ is realized for some $c^*$  which is $ACF$-independent from $\{a_1, a_2\}$ over $E$.

Since $\phi(x) \in \tp_{\LC}(c_1/E)$, there is a multi cell $C:= \displaystyle{\bigcap_{i=1}^n(C^i \cap M^{|c|})}$ in $M^{|c|}$ such that $c_1 \in C$, $\{x \in C: M \models \phi(x) \}$ is multi-dense in $C$ and the set $C \cap M^{|c|}$ is definable in $M$ by a quantifier-free $\LC(E)$-formula. Then $c \in C$, since $\qftp_{\LC}(c/E)= \qftp_{\LC}(c_1/E)$.

For $j=1,2$, as $\psi_j(x, a_j) \in \qftp_{\LC}(c/E(a_j))$, there is a multi-cell $U_{j} := \displaystyle{\bigcap_{i=1}^n{(U^i_{j}}\cap M^{|c|})}$ in $M^{|c|}$ such that $c \in U_{j}$ and $U_j \subseteq \psi_j(M, a_j)$.
Then for all $i \in \{1, \ldots,n\}$,  $\displaystyle{U^i_1 \cap U^i_{2}} \cap C^i \not = \emptyset$. By saturation we can find for all $i \in \{1, \ldots,n\}$ an $<_i$-cell $D^i$ in $({\clos{M}{i}})^{|c|}$ such that $D^i \subseteq \displaystyle{U^i_1 \cap U^i_{2}} \cap C^i$.

Let $D := \displaystyle{\bigcap_{i=1}^n (D^i \cap M^{|c|})}$. By multi-density there is $c^* \in D$, $ACF$-independent of $\{a_1, a_2\}$ over $E$, and such that $M \models \phi(c^*)$. Then $M \models \phi(c^*) \wedge \psi(c^*, a_1) \wedge \psi(c^*,a_2)$.

\end{proof}
\end{claim}

Since $c$ realizes $\qftp_{\LC}(c_1/E(a_1)) \cup \qftp_{\LC}(c_2/E(a_2)) \cup \tp_{\LC}(c_1/E)$, and is $ACF$-independent of $\{a_1,a_2\}$ over $E$, by Proposition \ref{lemamalgamation} there is $c'$ realizing $\tp_{\LC}(c_1/E(a_1)) \cup \tp_{\LC}(c/E(a_2)) \cup \qftp_{\LC}(c/E(a_1,a_2)).$

In particular $c'$ realizes $\qftp_{\LC}(c_2/E(a_2)) \cup \qftp_{\LC}(c_1/E(a_1)) \cup \tp(c_2/E)$. By Proposition \ref{lemamalgamation} again there is $c''$ $ACF$-independent of $\{a_1,a_2\}$ over $E$, realizing $\tp_{\LC}(c_2/E(a_2)) \cup \tp_{\LC}(c'/E(a_1)) \cup \qftp(c'/E(a_1,a_2))$. Therefore $c''$ realizes $\tp_{\LC}(c_2/E(a_2)) \cup \tp_{\LC}(c_1/E(a_1)) \cup \qftp_{\LC}(c/E(a_1,a_2))$.
\end{proof}
\end{thm}
 
\begin{cor}\label{IT2}
Let $(M, <_1, \ldots, <_n)$ be a model of  $\PRCB$. Let $E = \acl(E) \subseteq M$ and $a_1, a_2, c$ tuples of $M$ such that:
$\tp_{\LC}(a_1/E)= \tp_{\LC}(a_2/E)$, $c$ is ACF-independent of $\{a_1,a_2\}$ over $E$, and $\qftp_{\LC}(c,a_1/E) = \qftp_{\LC}(c,a_2/E) $. 
Suppose that $E(a_1)^{alg}\cap E(a_2)^{alg}= E^{alg}$.
Then there exists a tuple $c^{*}$ in some elementary extension $M^{*}$ of $M$ such that:
\begin{enumerate} 
\item $\qftp_{\LC}(c^*/E(a_1,a_2)) = \qftp_{\LC}(c/E(a_1,a_2)),$
\item $\tp_{\LC}(c^{*}, a_1/E) = \tp_{\LC}(c^{*}, a_2/E),$ 
\item $\tp_{\LC}(c^{*}, a_1/E)= \tp_{\LC}(c, a_1/E).$
\end{enumerate}
\begin{proof}
As $\tp_{\LC}(a_1/E)= \tp_{\LC}(a_2/E)$, we can find $c_2$ such that $\tp_{\LC}(c_2,a_2/E)= \tp_{\LC}(c,a_1/E)$.
Since $c$ realizes $\qftp_{\LC}(c/E(a_1)) \cup \qftp_{\LC}(c_2/E(a_2))$, by Theorem \ref{thamalgamation} there is some $c^{*}$ realizing $\tp_{\LC}(c/E(a_1)) \cup \tp_{\LC}(c_2/E(a_2)) \cup \qftp_{\LC}(c/E(a_1,a_2))$.  This implies that $\tp_{\LC}(c^{*}, a_1/E) = \tp_{\LC}(c^{*}, a_2/E),$  and $\tp_{\LC}(c^{*}, a_1/E)= \tp_{\LC}(c, a_1/E)$.
\end{proof}

\end{cor}

%% file: PRCstability.tex
 \section{PRC fields and their stability theoretic properties} \label{sectionNTP2PRC} \label{PRCstability}

We give all necessary preliminaries about NIP, NTP$_2$ and strong theories and also some useful lemmas about indiscernible sequences.
For more details on NTP$_2$ and strong theories see  \cite{Che1} and \cite{CK}.

Fix $\mathcal{L}$ a language and $T$ a complete $\mathcal{L}$-theory. We work inside a monster model $\mathbb{M}$ of $T$.

\begin{defn}\label{defIPTP2} 
Let $\phi(\bar{x}, \bar{y})$ be an $\mathcal{L}$-formula.
\begin{enumerate}
 \item We say that  $\phi(\bar{x}, \bar{y})$ has the \emph{independence property $(IP)$} if for any $m \in \mathbb{N}$ there is a family of tuples $\{b_l: l < m\}$ in $\mathbb{M}^{|\bar{y}|}$ such that for each $A \in \mathcal{P}(m)$ there is a tuple $a_A \in \mathbb{M}^{|\bar{x}|}$, such that $\mathbb{M} \models \phi(a_A, b_l)$ if and only if  $l \in A$.
A formula $\phi(\bar{x}, \bar{y})$ is \emph{NIP} if it does not have the IP. A theory is called \emph{NIP} if no formula has IP.
\item We say that  $\phi(\bar{x}, \bar{y})$ has TP$_2$ if there are $(a_{l ,j})_{l, j < \omega}$ and $k \in \omega$ such that:
\begin{enumerate}
	\item $\{\phi(\bar{x}, a_{l, j})_{j \in \omega}\}$ is $k$-inconsistent for all $l< \omega .$
	\item For all $f:\omega \rightarrow \omega, \{\phi(\bar{x}, a_{l, f(l)}): l \in \omega\}$ is consistent.	
\end{enumerate}
A formula $\phi(\bar{x}, \bar{y})$ is \emph{NTP$_2$}  if it does not have TP$_2$.
A theory is called \emph{NTP$_2$} if no formula has TP$_2$.
\end{enumerate}
By Proposition 5.31 of \cite{Sim} if $T$ is NIP, then it is NTP$_2$.
\end{defn}

\begin{defn}
Let $p(x)$ be a (partial) type. An \emph{inp-pattern of depth $\lambda$ in $p(x)$}, where $\lambda$ is a finite or infinite cardinal, consists of $(\bar{a_l}, \phi_{l}(x,y_l),k_l)_{l <\lambda}$ with $\bar{a_l}=(a_{l,j})_{j \in \omega}$ and $k_l \in \omega$ such that:
\begin{enumerate}
\item $\{\phi_l(x, a_{l,j})\}_{j <\omega}$ is $k_{l}$-inconsistent, for each $l< \lambda$.
\item $\{\phi_l(x, a_{l,f(l)})\}_{l <\lambda} \cup p(x)$ is consistent, for any $f: \lambda \rightarrow \omega$.
\end{enumerate}
The \emph{burden} of a partial type $p(x)$ is the supremum of the depths of inp-patterns in it. 
We denote the burden of $p$ by $bdn(p)$ and by $bdn(\bar{a}/A)$ the burden of $\tp(\bar{a}/A)$.

\begin{para} \label{factbdntypes} By  Theorem 2.5 of \cite{Che1} if $bdn(b/A) < \kappa$ and $bdn(a/Ab) < \lambda$, with $\kappa$ and $\lambda$ finite or infinite cardinals, then $bdn(a,b/A)< \lambda \times \kappa$.
\end{para}
\end{defn}

\begin{defn}
Consider a set of sequences ${(a_l)}_{l \in \omega}$, with $a_l = (a_{l, j})_{j < \kappa}$.  We say that they
are \emph{mutually indiscernible} over a set $C$ if $a_l$ is indiscernible over $\{C(a_{l'})\}_{l' \not = l}$.
\end{defn}

\begin{fact}\label{Strong1} \cite[Lemma 2.2]{Che1}\label{strongindis}
For a (partial) type $p(x)$ over a small set $E$, the following are equivalent:
\begin{enumerate}
\item There is an inp-pattern of depth $\lambda$ in $p(x)$.
\item There is an array $(\bar{a}_{l})_{l<\lambda}$ with rows mutually indiscernible over $E$ and $\phi_l(x, y_{l})$ for $l <\lambda$ such that:
\begin{enumerate}
\item $\{\phi_{l}(x,a_{l,j})\}_{j<\omega}$ is inconsistent for every $l <\lambda$.
\item $p(x) \cup \{\phi_l(x, a_{l,f(l)})\}_{l < \lambda}$ is consistent, for any $f: \lambda \rightarrow \omega$.
\end{enumerate}
\end{enumerate}
\end{fact}

\begin{defn} 
$T$ is called \emph{strong} if there is no inp-pattern of infinite depth in it. Clearly, if $T$ is strong then it is NTP$_2$.
\end{defn}

\begin{fact}

\begin{enumerate}\label{strongari}\label{strondisjun}
 \item In the definition of strong it is enough to look at types in one variable. \cite[Theorem 2.5]{Che1} 
 \item If $(\bar{a_l}, \phi_{l,0}(x,y_{l,0}) \lor \phi_{l,1}(x,y_{l,1})   ,k_l)_{l < \lambda}$ is an inp-pattern, then $(\bar{a_l}, \phi_{l,f(l)}(x,y_{l,f(l)}),k_l)_{l < \lambda}$ is an inp-pattern for some $f: \lambda \rightarrow\{0,1\}$. \cite[Lemma 7.1]{Che1}
\end{enumerate}
 
\end{fact}

\begin{lem}(Folklore)\label{interindis} 
Let $({a_l})_ {l \in \omega}$ be an indiscernible sequence over $E$. Then $\displaystyle{\bigcap_{l \in \omega}} \dcl(E(a_l)) = \dcl(E(a_0)) \cap \dcl(E(a_1))$.
\begin{proof}

Clearly $\displaystyle{\bigcap_{l \in \omega}\dcl(E(a_l))} \subseteq \dcl(E(a_0)) \cap \dcl(E(a_1))$.

Let $\alpha \in \dcl(E(a_0)) \cap \dcl(E(a_1))$. Let $\varphi_1(x, a_0)$ and $\varphi_2(x, a_1)$ be the formulas that define $\alpha$ with parameters in $E(a_0)$ and $E(a_1)$ respectively.

Then $\mathbb{M} \models \varphi_1(\alpha, a_0) \wedge \varphi_2(\alpha, a_1) \wedge \exists^{=1}x\varphi_1(x, a_0) \wedge \exists^{=1}x\varphi_2(x, a_1)$.

By indiscernibility $\mathbb{M} \models \exists x (\varphi_1(x, a_0) \wedge \varphi_2(x, a_k)) \wedge \exists^{=1}x\varphi_1(x, a_0) \wedge \exists^{=1}x\varphi_2(x, a_k)$, for all $k \geq 1.$ 
Since $\mathbb{M} \models \varphi_1(\alpha, a_0) \wedge \exists^{=1}x\varphi_1(x, a_0)$, we get that $\mathbb{M} \models \varphi_2(\alpha, a_k)$, for all $k \geq 1$. Then $\alpha \in \dcl(E(a_k))$ for all $k \geq 1$, since $\mathbb{M} \models \exists^{=1}x\varphi_2(x, a_k)$.
\end{proof}
\end{lem}

Note that in particular this implies for all $0< l< j$ in $\mathbb{N}$:
\[\dcl(E(a_l))\cap \dcl(E(a_j)) = \dcl(E(a_0)) \cap \dcl(E(a_l)).\]

\begin{lem}(Folklore)\label{algdisj} 
Let $(a_l)_{l \in \omega}$ be an indiscernible sequence over $E$. Let $F = \dcl(E(a_{0}))\cap \dcl(E(a_{1}))$. 
Then $(a_{l})_{l \in \omega}$ is indiscernible over $F$.
\begin{proof}
If $k_0 < \ldots < k_r$, we want to show that: \[\tp(a_{0},\ldots, a_{r}/F)= \tp(a_{k_0},\ldots, a_{k_r}/F).\]
Let $\bar{\beta} \subseteq F$, and $\varphi(\bar{x},\bar{\beta})\in \tp(a_{0}, \ldots, a_{r}/F)$.
Then $\bar{\beta} \subseteq \dcl(E(a_{0}))\cap \dcl(E(a_{1}))$. 
By Lemma \ref{interindis}, $\bar{\beta} \subseteq \dcl(E(a_{k}))$ for all $k \in \omega$.

Let $\psi(\bar{x}, a_{0})$ the formula that defines $\bar{\beta}$ over $E(a_{0})$. Then by indiscernibility $\psi(\bar{x}, a_{k})$ defines $\bar{\beta}$ over $E(a_{k})$ for all $k \in \omega$.
Let $k > k_r$, Then
\[\mathbb{M}\models \exists \bar{x} (\varphi(a_{0}, \ldots, a_{r}, \bar{x}) \wedge \psi(\bar{x}, a_{k}) ).\]
Then by indiscernibility: 
\[\mathbb{M}\models \exists \bar{x} (\varphi(a_{k_0}, \ldots, a_{k_r}, \bar{x}) \wedge \psi(\bar{x}, a_{k}) ).\]
As $\bar{\beta}$  is the only tuple that satisfies $\psi(\bar{x}, a_{k})$, we have $\mathbb{M}\models\varphi(a_{k_0}, \ldots, a_{k_r}, \bar{\beta})$. 
\end{proof}
\end{lem}

\subsection{Independence property, NTP$_2$ and strength}

\begin{thm}\label{IPPRC}
Let $M$ be a PRC field which is neither algebraically closed nor real closed. Then $Th_{\mathcal{L_R}}(M)$ has the independence property. More generally $Th_{\mathcal{L_R}}(M)$ has the $IP_n$ property, for all $n \in \mathbb{N}$ (see Definition 2.1 of \cite{Hem}).
\begin{proof}
By Properties \ref{PRCcaracte} (3) $M(\sqrt{-1})$ is a PRC field. Since $M(\sqrt{-1})$ has no orderings, it is a PAC field. 
Observe that as $M$ is not real closed, $M(\sqrt{-1})$ is neither separably closed nor real closed. By Corollary 6.5 of \cite{Du} $M(\sqrt{-1})$ has the IP, more generally by Corollary 7.4 of \cite{Hem} $M(\sqrt{-1})$ has the $IP_n$ property. 
Then $M$ has the $IP_n$ property, since $M(\sqrt{-1})$ is interpretable in $M$.
\end{proof}
\end{thm}

\begin{lem}\label{NIPreducts}
Let $\mathcal{L}$ be a language and let $T$ be an $\mathcal{L}$-theory. Suppose that for any $\mathcal{L}$-atomic formula $\varphi(\bar{x}, \bar{y})$, there are $\mathcal{L}' \subseteq \mathcal{L}$ and an $\mathcal{L'}$-theory $T'$ such that: $\varphi(\bar{x}, \bar{y})$ is an $\mathcal{L}'$-formula, $T'$ is NIP and for all $M \models T$, there is $M' \models T'$ such that $M|_{\mathcal{L'}} \subseteq M'$.
Then in $T$ every quantifier-free $\mathcal{L}$-formula is NIP.
 \begin{proof}
Every quantifier-free formula $\phi{(\bar{x}, \bar{y})}$ is a Boolean combination of atomic formulas.
Since every Boolean combination of NIP formulas is NIP (Lemma 2.9 of \cite{Sim}), it is enough to show that every atomic formula is NIP.

Let $\varphi(\bar{x}, \bar{y})$ be an atomic formula and let $M$ be a model of $T$. Let $\mathcal{L}'$, $T'$ and $M'$ satisfy the hypothesis of the Lemma for the formula $\varphi(\bar{x}, \bar{y})$. 
Suppose that $\varphi(\bar{x}, \bar{y})$ has the independence property.
Then for all $m \in \mathbb{N}$, there is a family of tuples $\{b_l: l<m\}$ and $\{a_A: A \in \mathcal{P}(m)\}$ in $M$ such that $M\models \varphi(a_A, b_l)$ if and only if $l \in A$.  
Then $M'\models \varphi(a_A, b_l)$ if and only if $l \in A$, since $M|_{\mathcal{L'}} \subseteq M'$. That contradicts the fact that $T'$ is NIP.
\end{proof}
\end{lem}

\begin{cor}\label{SQNIP}
Let $n \geqslant 1$. In $n$-PRC every quantifier-free $\LC$-formula is NIP.  
\begin{proof}
By Lemma \ref{NIPreducts} using the fact that in $n$-PRC the atomic formulas are of the form $p(\bar{x}, \bar{y}) >_{i} 0$,  with $i \in \{1, \ldots, n\}$ and $p(\bar{x}, \bar{y}) \in \mathbb{Q}[\bar{x}, \bar{y}]$, and that RCF is NIP. 
\end{proof}
\end{cor}

\begin{fact} \label{factNIP} \cite[Proposition 2.8]{Sim}
The formula $\varphi(x, y)$ is NIP if and only if for any indiscernible sequence $(a_l : l \in \mathbb{N})$ and any tuple $b$, there is some $k \in \mathbb{N}$ such that $\varphi(a_l, b) \leftrightarrow \varphi(a_j, b)$, for all $l, j \geq k$.
\end{fact}

\begin{lem}\label{lemsq}
Let $n \geq 1$ and let $M$ be a bounded PRC field with exactly $n$ orders, which is not real closed. Let $T:=Th_{\LC}(M)$ (see \ref{PRCB} for the definition of $\LC$).
Let $E =\acl(E)\subset M$ and $(a_j)_{j\in \omega}$ an indiscernible sequence over $E$.
Let $\phi(x, \bar{y})$ be an $\LC$-formula and $I$ a multi-interval definable
over $E$ such that $\{x \in I:M \models \phi(x, a_0)\}$ is multi-dense in $I$.
Then for all multi-intervals $J\subseteq I$ definable over $E$, $\{x \in J\} \cup \{\phi(x, a_0) \wedge \phi(x,a_1)\}$ is consistent.
\begin{proof}

\begin{claim1} We can suppose that $E(a_{0})^{alg} \cap E(a_{1})^{alg} = E^{alg}$.
\begin{proof}

Denote by $A_0:=\acl(E(a_0))$ and by $A_1:= \acl(E(a_1))$.

By Lemma \ref{algdisj} and the fact that $\acl=\dcl$ we can suppose that $A_0\cap A_1 = E$.

Since $A_0A_1$ is a regular extension of $A_0$ and of $A_1$, by Lemma 2.1 of \cite{Cha} $A_0^{alg} \cap {A_1}^{alg} = {(A_0 \cap A_1)}^{alg}$.
 
As $A_0^{alg} = E(a_{0})^{alg}$, $A_1^{alg}= E(a_{1})^{alg}$ and $A_0\cap A_1=E$, we get that $E(a_{0})^{alg} \cap E(a_{1})^{alg} = E^{alg}$.     
\end{proof}
\end{claim1}

Let $J = \displaystyle\bigcap_{i=1}^n{(J^i\cap M)}$ be a multi-interval definable over $E$ such that $J \subseteq I$.

\begin{claim2}\label{indis0}
There exists $\widetilde{c}$ in some elementary extension $N$ of $M$ such that $\widetilde{c} \in J$,  $\widetilde{c} \notin \acl(Ea_{j}:j\in \omega)$ and for all $j \in \omega$, $\qftp_{\LC}(\widetilde{c}, a_{0}/E)= \qftp_{\LC}(\widetilde{c}, a_{j}/E)$.
\begin{proof}
 By compactness it is enough to show that if $\psi_1(x,y), \ldots, \psi_m(x,y)$ are quantifier-free $\LC(E)$-formulas, then the type:
 
 \[ \{x \in J\} \cup \{p(x)\neq0\}_ {p \in E(a_j:j\in \omega)[x], p\not=0} \cup  \{ \psi_l(x,a_{0}) \leftrightarrow \psi_l(x,a_{j})\}_{_{0 < j, 1\leq l \leq m }}\]
 is consistent.
 
Let $d \in J$, $d \notin \acl(Ea_{j}: j\in \omega)$; by Corollary \ref{SQNIP} and Fact \ref{factNIP}, for all $1 \leq l \leq m$ there exists $k_l \in \mathbb{N}$ such that $\psi_l(d,a_{j_1}) \leftrightarrow \psi_l(d,a_{j_2})$ for all $k_l \leq j_1 < j_2$. Let $k= \max\{k_1, \ldots, k_m\}$.

Then $d$ realizes the type: 
\[ \{x \in J\} \cup \{p(x)\neq0\}_ {p \in E(a_j:j\in \omega)[x], p\not=0} \cup  \{ \psi_l(x,a_{j_1}) \leftrightarrow \psi_l(x,a_{j_2})\}_{_{k \leq j_1<j_2, 1 \leq l\leq m}}\]

Since $J$ is definable with parameters in $E$, $p(x) \in E(a_j:j\in \omega)[x]$ and ${(a_{j})}_{j \in \omega}$ is indiscernible over $E$, it follows that the type:
\[ \{x \in J\} \cup \{p(x)\neq0\}_ {p \in E(a_j:j\in \omega)[x], p\not=0} \cup  \{ \psi_l(x,a_{0}) \leftrightarrow \psi_l(x,a_{j})\}_{_{0 < j, 1 \leq l\leq m }}\]
is consistent.
\end{proof}
\end{claim2}

\begin{claim3}
There exists $c$ in some elementary extension $N$ of $M$ such that $c \in J$,  $c \notin \acl(Ea_{j}:j\in \omega)$, $N \models \phi(c, a_{0})$, and for all $j \in \omega$, $\qftp_{\LC}(c, a_{0}/E)= \qftp_{\LC}(c, a_{j}/E)$.

\end{claim3}
\begin{proof}

By compactness it is enough to show that if $\psi_1(x,y), \ldots, \psi_m(x,y)$ are quantifier-free $\LC(E)$-formulas, then the type:
\[ \{x \in J\} \cup \{p(x)\neq0\}_ {p \in E(a_j:j\in \omega)[x], p\not=0}  \cup  \{ \psi_l(x,a_{0}) \leftrightarrow \psi_l(x,a_{j})\}_{_{0 < j, 1 \leq l\leq m}} \cup \{\phi(x,a_0)\}\]
is consistent.
Define $\theta(x,y_1,y_2):= \displaystyle{\bigwedge_{l=1}^m \psi_l(x,y_1) \leftrightarrow \psi_l(x,y_2)}$.

Let $N \succeq M$ be $|E|^+$-saturated, and $\widetilde{c} \in N$ satisfying claim 2. 

Then $N \models \theta(\widetilde{c},a_0,a_j)$, for all $0 < j \in \omega.$ 
As $\theta(x, a_0,a_j)$ is a quantifier-free $\LC$-formula and $\widetilde{c}\notin \acl(E(a_j): j\in \omega)$, there exist for all $i \in \{1, \ldots, n\}$, $<_i$-intervals  $B_{j}^{i}\subseteq \clos{N}{i}$ qf-definable over $E(a_0,a_j)$, such that $\widetilde{c} \in \displaystyle{\bigcap_{i = 1}^n(B_{j}^{i}\cap N)}$ and  $\displaystyle{\bigcap_{i = 1}^n(B_{j}^{i}}\cap N) \subseteq \theta(N,a_0,a_j)$.

Let $1 \leq i\leq n$; as $N$ is $|E|^+$-saturated, there exists an $<_i$-interval $B^{i}\subseteq \clos{N}{i}$ containing $\widetilde{c}$, such that $B^{i}\cap N \subseteq \displaystyle{\bigcap_{j \in \omega} (B_{j}^{i}\cap N)}$.

As $\widetilde{c}\in J :=\displaystyle{\bigcap_{i=1}^n (J^{i}\cap N)}$, we can assume that $B^{i}\subseteq J^{i}$, for all $1 \leq i \leq n$. 

Then $B :=\displaystyle{\bigcap_{i = 1}^n (B^{i}\cap N)} \subseteq J \subseteq I$ and $\displaystyle{B} \subseteq \displaystyle{\bigcap_{l=1}^m\theta(N,a_0, a_j)}$,  for all $j \in \omega$.

By multi-density of $\phi(x, a_0)$ in $I$ and saturation, there exists $c \in B$, $c \notin \acl(Ea_{j}:j \in \omega)$, such that $N \models \phi(c, a_0)$. 
As $c \in B$, $\psi_l(c,a_0) \leftrightarrow  \psi_l(c,a_j)$, for all $j \in \omega$, $1 \leq l \leq m$.

\end{proof}

 By Corollary \ref{IT2} there is $c^{*}$ in some elementary extension $N^*$ of $N$, such that $\tp(c^{*},a_0/E)= \tp(c^{*},a_1/E)$ and $\tp(c^{*},a_0/E)=\tp(c,a_0/E)$.
So $N^*\models \phi(c^{*},a_0)\wedge \phi(c^{*},a_1)$ and since $c \in J$, $J$ is definable with parameters in $E$, and $tp(c^*/E)=tp(c/E)$, we obtain that $c^{*} \in J$.
Then $c^*$ realizes $\{x \in J\} \cup \{\phi(x, a_0) \wedge \phi(x, a_1)\}$.

\end{proof}
\end{lem}

\begin{thm}\label{strongdensity}
Let  $n \geq 1$ and let $M$ be a bounded PRC field with exactly $n$ orders which is not real closed, and let $T:=Th_{\LC}(M)$.
Let $E =\acl(E)\subset M$ and $(a_j)_{j\in \omega}$ an indiscernible sequence over $E$.
Let $\phi(x, \bar{y})$ be an $\LC(E)$-formula and $I:= \bigcap_{i=1}^n I^i$ a multi-interval definable
over $E$ such that $\{x \in I:M \models \phi(x, a_0)\}$ is multi-dense in $I$.
Then the type $p(x):=\{\phi(x, a_{j})\}_{j \in \omega}$ is consistent.

\begin{proof}

Define $\psi(x, y_1,y_2):= \phi(x, y_1) \wedge \phi(x,y_2)$.
By Theorem \ref{descomposition} there are a finite set $B \subseteq M$, $m \in \mathbb{N}$ and $I_1,\ldots, I_m$, with $I_j:= \displaystyle{\bigcap_{i=1}^n I_j^i}$ a multi-interval such that:
\begin{enumerate}
 \item $\psi(M,a_0,a_1) \subseteq \displaystyle\bigcup_{j=1}^m{I_j} \cup B$,
 \item $\{x \in I_j: M \models \psi(x, a_0,a_1)\}$ is multi-dense in $I_j$, for all $1 \leq j \leq m$.
\end{enumerate}
Denote by $\clos{E}{i}:= E^{alg} \cap \clos{M}{i}$.

\begin{claim}
There is $j \in \{1, \ldots, m\}$ such that for all $i \in \{1, \ldots, n\}$, $|I^i_j \cap \clos{E}{i}|\geq 2$.
\begin{proof}
Suppose that for all $j \in \{1, \ldots, m\}$, there is $i_j \in \{1, \ldots, n\}$ such that $|I^{i_j}_j \cap \clos{E}{i_j}|\leq 1$. So we can find an $<_{i_j}$-interval $J^{i_j}\subseteq I$ with extremities in $\clos{E}{i_j}$ such that $J^{i_j} \cap I^{i_j} =  \emptyset$. 
Observe that if $i_{j}= i_{l}$, for some $j,l \in \{1, \ldots, m\}$, then we can choose $J^{i_{j}}= J^{i_{l}}$.

Let $J:= \displaystyle{\bigcap_{j=1}^m J^{i_j}} \cap \displaystyle{\bigcap_{i \not = i_j} \clos{M}{i}}$.
Then $J$ is a multi interval, definable over $E$, and $J \subseteq I$.
By Lemma \ref{lemsq}  there exists $c \in J$ such that $M \models \psi(c, a_0,a_1)$. 
This implies that $J \cap \displaystyle{\bigcup_{j=1}^m {I_j}} \not = \emptyset$.
Thus there is $j \in \{1, \ldots, m\}$ such that $J \cap I_j \not = \emptyset$, whence $J^{i_j} \cap I_j^{i_j} \not = \emptyset$ which gives the desired contradiction.

\end{proof}
\end{claim}

Take $x^i_1 \not = x^i_2 \in I^i_j \cap \clos{E}{i}$ for each $1 \leq i \leq n$. Define $y_i:= \min_{<_i}\{x_1, x_2\}$, $z_i:= \max_{<_i}\{x_1, x_2\}$ and let $\widetilde{I_1}:= \displaystyle\bigcap_{i=1}^n((y_i,z_i)_i\cap M)$. 
Then $\widetilde{I_1}$ is a multi-interval definable over $E$.

Define $b_{j}:= (a_{2j},a_{2j+1})$; then $(b_j)_{j\in \omega}$ is indiscernible over $E$. 
Then $\{x \in \widetilde{I_1}: M \models \psi(x, b_0)\}$ is multi-dense in $\widetilde{I_1}$, since $\widetilde{I_1}  \subseteq I_j$. 

Repeating this process with the formula $\psi(x, b_0)$ and the multi-interval $\widetilde{I_1}$,  we find a multi-interval $\widetilde{I_2}$, definable over $E$ such that that $\{x \in \widetilde{I_2}:M \models \phi(x, a_0) \wedge \phi(x,a_1) \wedge \phi(x, a_2) \wedge \phi(x,a_3)\} $ is multi-dense in $\widetilde{I_2}$, etc.
This shows that $p(x)$ is finitely consistent.   
\end{proof}
\end{thm}

Theorem \ref{strongdensity} can be easily generalized to several variables:

\begin{thm}\label{lemsq+v}
Let $n\geq 1$ and let $M$ be a bounded PRC field with exactly $n$ orders, which is not real closed. Let $T:=Th_{\LC}(M)$. Let $E =\acl(E)\subset M$ and $(a_j)_{j\in \omega}$ an indiscernible sequence over $E$.
Let $\phi(x_1, \ldots,x_r, \bar{y})$ be an $\LC(E)$-formula and $C$ a multi-cell in $M^r$ definable over $E$ such that $\{(x_1, \ldots,x_r) \in C: \phi(x_1, \ldots,x_r, a_0)\}$ is multi-dense in $C$.
Then $p(x_1, \ldots,x_r):= \{\phi(x_1, \ldots,x_r, a_j)\}_{j \in \omega}$ is consistent.
\begin{proof}
We will first show that Lemma \ref{lemsq} can be generalized to several variables.
\begin{claim}
If $J\subseteq C$ is a multi-box in $M^r$ definable over $E$, then there exists $c=(c_1, \ldots, c_r)$ in some elementary extension $N$ of $M$ such that: $c \in J$, $trdeg(E(c)/E)=r$, $N \models \phi(c,a_0)$ and for all $j \in \omega$, $\qftp_{\LC}(c, a_{0}/E)= \qftp_{\LC}(c, a_{j}/E)$.

\begin{proof}
Let $J:=\displaystyle{\bigcap_{i=1}^n (J^{i}\cap M^r)} \subseteq C$ be a multi-box in $M^r$ definable over $E$.
As in Claim 1 of Lemma \ref{lemsq} we can suppose that $E(a_0)^{alg} \cap E(a_1)^{alg}= E^{alg}$.

Let $\psi_1(\bar{x}, \bar{y}), \ldots, \psi_m(\bar{x}, \bar{y})$ be quantifier-free $\LC(E)$-formulas, and consider the type:
\[q(\bar{x}):= \{\bar{x} \in J\} \cup \{p(\bar{x})\neq0\}_ {p \in E(a_j:j\in \omega)[\bar{x}], p\not=0}  \cup  \{ \psi_l(\bar{x},a_{0}) \leftrightarrow \psi_l(\bar{x},a_{j})\}_{_{0 < j, 1 \leq l\leq m}} \cup \{\phi(\bar{x},a_0)\}\]
By compactness it is enough to show that $q(\bar{x})$ is consistent.

Define $\theta(\bar{x},y_1,y_2):= \displaystyle{\bigwedge_{l=1}^m \psi_l(\bar{x},y_1) \leftrightarrow \psi_l(\bar{x},y_2)}$, for $1 \leq l \leq m$ .

Exactly as in Claim 2 of Lemma \ref{lemsq} there exists $\tilde{c}= (\tilde{c_1}, \ldots,\tilde{c_r})$ in some elementary extension $N$ of $M$ such that: $\tilde{c} \in J$, $trdeg(E(\tilde{c})/E)=r$, and for all $j \in \omega$, $\qftp(\tilde{c},a_0/E)= \qftp(\tilde{c},a_j/E)$.
Suppose that $N$ is $E^+$-saturated.

Then $N \models \displaystyle{\bigcap_{l=1}^m\theta(\tilde{c},a_0,a_j)}$, for all $j \in \omega.$ 

As $\theta(\bar{x}, a_0,a_j)$ is quantifier-free, using cell decomposition in each real closure $\clos{N}{i}$ for all $1 \leq i \leq n$ and the fact that $trdeg(E(\tilde{c})/E)=r$,  there exists for all $i \in \{1, \ldots, n\}$, $<_i$-open cells  $B_{j}^{i}\subseteq (\clos{N}{i})^r$ definable over $E(a_0,a_j)$, such that $\widetilde{c} \in \displaystyle{\bigcap_{i = 1}^n(B_{j}^{i}\cap N^r)}$ and  $\displaystyle{\bigcap_{i = 1}^n(B_{j}^{i}\cap N^r)} \subseteq \theta(N,a_0,a_j)$. 
By saturation and the fact that each $(B^i_{j}\cap N^r)$ is $<_i$-open (in the product topology), there exists an $<_i$-box $B^{i}$ containing $\widetilde{c}$, such that $B^{i}\cap N^r \subseteq \displaystyle{\bigcap_{j \in \omega}(B_{j}^{i}\cap N^r)}$.

As $\widetilde{c}\in J^i$, for all $1 \leq i \leq n$, we can assume by taking the intersection that $B^{i}\subseteq J^{i}$. 

Then $B :=\displaystyle{\bigcap_{i = 1}^n (B^{i}\cap N^r)} \subseteq J \subseteq C$ and $\displaystyle{B} \subseteq \displaystyle{\bigcap_{l=1}^m\theta(N,a_0, a_j)}$,  for all $j \in \omega$. 

By multi-density of $\phi(\bar{x}, a_0)$ in $C$ and saturation, there exists $c=(c_1, \ldots,c_r) \in B$ such that $trdeg(E(c)/E)=r$, and $N \models \phi(c, a_0)$. 
As $c \in B$, we get that $\psi_l(c,a_0) \leftrightarrow  \psi_l(c,a_j)$, for all $j \in \omega$, $1 \leq l \leq m$.
Then $c$ realizes $q(\bar{x})$. 
\end{proof}
\end{claim}
As in Lemma \ref{lemsq} using Theorem \ref{IT2}, for all multi-box $J \subseteq C$ definable over $E$, $\{\bar{x} \in J\} \cup \{\phi(\bar{x}, a_0) \wedge \phi(\bar{x}, a_1)\}$ is consistent.

Define $\psi(\bar{x}, y_1,y_2):= \phi(\bar{x}, y_1) \wedge \phi(\bar{x},y_2)$.
By Theorem \ref{descomposition2} there are a set $V \subseteq M^r$, $m \in \mathbb{N}$, and $C_1, \ldots, C_m$ with $C_j= \displaystyle{\bigcap_{i=1}^n (C^i_j\cap M^r)}$ a multi-cell such that:
\begin{enumerate}
\item $\psi(M, a_0,a_1) \subseteq \displaystyle{\bigcup_{j=1}^m C_ j \cup V }$,
\item the set $V$ is contained in some proper Zariski closed subset of $M^r$, which is definable over $\acl(a_0,a_1)$,
\item $\{x \in C_j: M \models \psi(\bar{x}, a_0,a_1)\}$ is multi-dense in $C_j$ for all $1 \leq j \leq m$,
\item the set $C^i_j \cap M^r$ is definable in $M$ by a quantifier-free $\Li(a_0,a_1)$-formula, for all $1 \leq j \leq m$.
\end{enumerate}

If $J \subseteq C$ is a multi-box in $M^r$ definable over $E$, then there is $\bar{x} \in J$ such that $\phi(\bar{x}, a_0) \wedge \phi(\bar{x},a_1)$.
So there exists $j \leq m$ such that $J \cap C_j \not = \emptyset$.
As in Theorem \ref{strongdensity} there exists $j \leq m$ and multi-cell $J \subseteq C_j$, definable over $E$ such that $\psi(\bar{x}, a_0,a_1)$ is multi-dense in $J$.  
The rest of the proof is as in Theorem \ref{strongdensity}.
\end{proof}
\end{thm}

\begin{defn}(\cite[Definition 2.1]{Ons})
\begin{enumerate}
\item We say that a formula $\delta(x, a)$ \emph{strongly divides} over $A$ if $\tp(a/A)$ is non algebraic and $\{\delta(x, a')\}_{a' \models \tp(a/A)}$ is $k$-inconsistent for some $k \in \mathbb{N}$.
\item  A formula $\delta(x,a)$ \emph{\th-divides}(\emph{thorn divides}) over $A$ if we can find some tuple $c$ such that $\delta(x,a)$ strongly divides over $Ac$.
\item A formula \emph{\th-forks} (\emph{thorn forks}) over $A$ if it implies a (finite) disjunction of formulas which \th-divide over $A$.
\item The type $p(x)$ \emph{\th-divides} if there is a formula in $p(x)$ which \th-divides; similarly for \th-forking.
\item We say that \emph{$a$ is \th-independent of $b$} over $A$, denoted by $a \thind_A b$, if $\tp(a/Ab)$ does not \th-fork over $A$.   

\end{enumerate}
Observe that in (1), $k$-inconsistency means: if $a_1, \ldots, a_k$ realize the $\tp(a/A)$ and the sets $\delta(M, a_1), \ldots, \delta(M,a_k)$ are distinct, then their intersection is empty.

\end{defn}

\begin{defn}\cite[Theorem 3.7]{EaOn}
 A theory is called \emph{rosy} if there is some $\kappa$ such that there are no \th-forking chains of length $\kappa$. That is for all $b$ one can not find $(a_j)_{j \in \kappa}$ such that for $\alpha < \kappa$ one has $\tp(b/ (a_j)_{j \leq \alpha})$ \th-forks over $(a_j)_{j < \alpha}$.
\end{defn}

\begin{thm} \label{PACNTP2}
If $M$ is an unbounded PAC field, then $Th_{\mathcal{L_R}}(M)$ has TP$_2$ and is not rosy.
\begin{proof}
Suppose that $M$ is an unbounded PAC field which is sufficiently saturated. 
As in the proof of Theorem 3.9 of $\cite{Cha0}$, we can assume that there are infinitely many finite algebraic extensions $\{L_j\}_{j \in \omega}$ of $M$, which are linearly disjoint over $M$ and with Galois group over $M$ isomorphic to some fixed simple group $G$. Let $r= |G|$.
For each $j \in \omega$, take $\alpha_j$ such that $L_j=M(\alpha_j)$.   Let $g(\bar{Y},X)\in \mathbb{Z}[\bar{Y},X]$  and $\bar{a}_j:= (a_{1j}, \ldots, a_{rj}) \in M^r$ be such that $g(\bar{a}_{j}, X)= X^r+ a_{1j}X^{r-1}+ \ldots + a_{rj}$ is the minimal polynomial of $\alpha_j$ over $M$. Define $\widetilde{L}$ as the field composite of $\{L_j: j \in \omega\}$. 

Choose an element $t$ transcendental over $M$ and $n>4$ such that $G$ embeds into $A_n$ (the group of even permutations of $n$ elements). 
Observe that such an $n$ exists: $G$ embeds into $S_r$ the group of permutations on $r$ elements,  which in turn embeds into $A_{2r}$. 

Let $k$ be the prime field of $M$.  Let $\{b_{l}: l \in \omega\} \subseteq M$ be algebraically independent over $k$.
As in Lemma 3.8 of $\cite{Cha0}$, using Theorem A of \cite{Pp}, we can find an algebraic extension $E_1$ of $k(b_1,t)$ such that $E_1/k(b_1)$ and $E_1/k(t)$ are regular, and $\Gal(E_1/k(b_1,t))\cong A_n$. Take $\beta_1$ integral over $k[b_1, t]$ such that $E_ 1=k(b_1,t)(\beta_1)$. Let $p(b_1,t, X) \in k[b_1,t, X]$ be the minimal polynomial of $\beta_1$ over $k(b_1,t)$. 
For all $l \in \omega$ define $E_l$ as the field extension of $k(b_l,t)$ generated by a root of $p(b_l,t, X)$.

Since $E_l$ is a regular extension of $k(b_l)$, and is algebraically independent from $M$ over $k(b_l)$ the field $M_l:=ME_l$ is a regular extension of $M$ and $\Gal(M_l/M(t)) \cong A_n$. 
Define $\widetilde{M}$ as the field composite of $\{M_l: l \in \omega\}$.
As in Lemma 3.8 of $\cite{Cha0}$, $\widetilde{M}$ is a regular extension of $M$.

Let $\varphi(t, b_l, \bar{a_j})$ be the formula that says ``The extension generated by a root of $p(b_l, t, X)$ contains a root of $g(\bar{a_{j}}, X)=0"$. Note that such a formula exists, because in $M$ we can interpret finite Galois extensions of $M$ (see Appendix 1 of $\cite{Cha}$ for more details).

Observe that for all $l \in \omega$, $\{\varphi(t, b_l, \bar{a}_j): j \in \omega\}$ is $(\frac{n!}{2r}+1)$- inconsistent: 
Otherwise there would exist $l \in \omega$ and $j_1, \ldots ,j_s$ with $s = (\frac{n!}{2r}+1)$ such that $M_l \supseteq  L_{j_1}\cdots L_{j_s}$. But $[M_l: M(t)]= \frac{n!}{2}$, $[L_j:M]= r$ for all $l,j \in \omega$ and  $\{L_j\}_{j \in \omega}$ is linearly disjoint over $M$. Then $sr< \frac{n!}{2}$ which contradicts the definition of $s$.

\begin{claim}
If $f: \omega \rightarrow \omega$, then $\{\varphi(t, b_l, \bar{a}_{f(l)}): l \in \omega\}$ is consistent.
\begin{proof}
For each $l$, fix an embedding $h_l: \Gal(L_{f(l)}/M)\rightarrow \Gal(M_l/M(t))$ and let $S_l:= \{(\sigma, h_l(\sigma)): \sigma \in \Gal(L_{f(l)}/M)\}$. Then $S_l$ is a subgroup of $\Gal(L_{f(l)}M_l/M(t)) \cong \Gal(L_{f(l)}/M)\times \Gal(M_l/M(t))$, because $M^{alg} \cap M_l = M$.
Define $P_l:= Fix(S_l) \subseteq L_{f(l)}M_l$. 
By definition of $S_l$, $L_{f(l)}M_l= L_{f(l)}P_l=M_lP_l$.

Then $\Gal(\widetilde{L}\widetilde{M}/M(t)) \cong \displaystyle{\prod_{j \in \omega}\Gal(L_j/M) \times \prod_{l \in \omega}\Gal(M_l/M(t))}$.

Define $S:= \{((\sigma_j)_{j \in \omega}, (\tau_{l})_{l \in \omega})  \in \Gal(\widetilde{L}\widetilde{M}/M(t)):\forall l \in \omega, \tau_l= h_l(\sigma_{f(l)})\}$.
If $\widetilde{P}$ is the field composite of $\{P_l: l \in \omega\}$, then $\widetilde{P}= Fix(S)$.

Note that $\Gal(\widetilde{P}/M(t))$ projects onto $\Gal(\widetilde{L}/M)$.
Indeed, let $(\sigma_j)_{j\in \omega} \in \Gal(\widetilde{L}/M)$ and for each $l \in \omega$ let $\tau_{l}= h_l(\sigma_{f(l)})$. 
Then $((\sigma_j)_{j \in \omega}, (\tau_l)_{l \in \omega}) \in \Gal(\widetilde{P}/M(t))$ is an extension of $(\sigma_j)_{j\in \omega}$.

This implies that $\widetilde{P}$ is a regular extension of $M$. 
Then by Fact \ref{PRC}  there exists an elementary extension $M^*$ of $M$ such that $\widetilde{P} \subseteq M^*$.

Since $\widetilde{P}M_l \supseteq L_{f(l)}$ and $\widetilde{P} \subseteq M^*$, it follows that $M^* \models \exists t \displaystyle{\bigwedge_{l \in \omega}\varphi(t, b_l, a_{f(l)})}$. Then $\{\varphi(t, b_l, a_{f(l)}): l \in \omega\}$ is consistent in $M$, since $M^*$ is an elementary extension of $M$.
\end{proof}
\end{claim}

Therefore the formula $\varphi(t; x, y)$ has TP$_2$.

To see that $Th_{\mathcal{L_R}}(M)$ is not rosy we need to show that for any cardinal $\kappa$, there exists $c$ and $(d_l)_{l < \kappa}$ such that for all $\alpha < \kappa$, $c\nthind_{(d_l)_{l < \alpha}}d_{\alpha}$.

Let $\kappa$ be a cardinal, construct as before $(b_l)_{l \in \kappa}$ and $(\bar{a_j})_{j \in \kappa}$. Let $\phi(t,b_l,\bar{a}_j):= \varphi(t, b_l, \bar{a}_j) \wedge$ ``the extension generated by a root of $g(\bar{a}_j,X)$ is Galois, 
with Galois group isomorphic to $G$''.
Let $\alpha < \kappa$. 

We claim that $\{\phi(t, b_{\alpha}, \bar{a}): \bar{a} \models \tp(\bar{a}_\alpha/b_{l},a_{l}: l< \alpha)\}$ is inconsistent: observe that if $\beta_1, \beta_2 \models \tp(\bar{a}_\alpha/b_{l},a_{l}: l< \alpha)$ and $\phi(t, b_{\alpha}, \beta_1) \not = \phi(t, b_{\alpha}, \beta_2)$, then $L_{\beta_1} \not = L_{\beta_2}$, where $L_{\beta_1}$ and $L_{\beta_2}$ are extensions generated by a root of $g(\beta_1,X)=0$ and $g(\beta_2,X)=0$ respectively; but as above $M_\alpha$ can only contain finitely many distinct sub extensions with Galois group $G$.

By definitions and properties of rosy theories, the formula $\phi(t,b_{\alpha}, \bar{a}_{\alpha})$ strongly divides over $(b_{j}a_{l}: j \leq \alpha, l< \alpha)$.
Let $(d_l)_{l \in \kappa}:= (b_l, \bar{a}_l)_{l \in \kappa}$. As in the Claim, $\{\phi(t, d_l): l \in \kappa\}$ is consistent, and we let $c$ realize $\{\phi(t, d_l): l \in \kappa\}$.
Since $\phi(t,d_{\alpha}) \in \tp(c/ (d_l)_{l \leq \alpha})$; $c\nthind_{(d_l)_{l < \alpha}}d_{\alpha}$.

\end{proof}
\end{thm}

\begin{cor}\label{PRCnonrosy}
If $M$ is an unbounded PRC field, then $Th_{\mathcal{L_R}}(M)$ is not rosy.
\begin{proof}
As before  $M(\sqrt{-1})$ is a PAC field. As $M(\sqrt{-1})$ is unbounded and interpretable in $M$, by Theorem $\ref{PACNTP2}$ it is not rosy, and so $M$ is not rosy. 
\end{proof}
\end{cor}

\begin{fact}\label{RCFstrong} 
The theory of real closed fields ($RCF$) is $\dpr$-minimal (\cite[Theorem A.6]{Sim}) and the $\dpr$-rank coincides with the burden ($\dpr$-rank coincides with $\bdn$ in any NIP theory, see \cite{Adl}). In particular $RCF$ is strong. 
Since $\dpr$-rank is sub-additive (Corollary 4.2 of \cite{KOU}) the burden satisfies the following:
if $r \in \mathbb{N}$ and $p(x_1, \ldots,x_r):= \{x_1= x_1, \ldots, x_r=x_r\}$, then $\bdn(p(x_1, \ldots,x_r))=r$.

\end{fact}

\begin{thm} \label{PRCstrong}
Let $n \geq 1$, let $M$ be a bounded PRC field with exactly $n$ orders. Then $Th_{\mathcal{L_R}}(M)$ is strong and $\bdn(\{x=x\})=n$.
\begin{proof}
If $M$ is real closed, by Fact \ref{RCFstrong} $Th_{\mathcal{L_R}}(M)$ is strong of burden $1$. Suppose that $M$ is not real closed. By Lemma \ref{Deforders} it is enough to show that $\PRCB= Th_{\LC}(M)$ is strong of burden $n$.
We can suppose that $M$ is sufficiently saturated. 
For $l \in \{0, \ldots,n-1\}$, define the formula $\varphi_l(x,y):= y <_{l+1}x<_{l+1}y+1$. 
Take $((a_{l,j})_{j \in \omega})_{l \leq n-1}$, such that $a_{l,j+1}= a_{l,j}+1.$ 
Using the Approximation Theorem ($\ref{ApTh}$),  $(\bar{a_l}, \varphi_l(x,y), 2)_{0 \leq l<n}$, with $\bar{a_l}= (a_{l,j})_{j\in \omega}$, is an inp-pattern of depth $n$. It follows that the burden is greater than or equal to $n$.

Suppose that there is an inp-pattern $(\bar{a_l}, \phi_{l}(x,y) ,k_l)_{0 \leq l < n+1}$ of depth $n+1$; by compactness we can take $\bar{a}_l:= (a_{l,j})_{j \in \kappa}$, with $\kappa$ a sufficiently large cardinal.
We can suppose that for all $0 \leq l < n+1$, $\phi_l(x, y)$ has parameters in a countable set $E$ with $E = \acl(E)\subseteq M$ so that $\G(E) \cong \G(M)$. 

By Fact \ref{strongindis}  we can suppose that the array $(\bar{a}_l)_{l< n+1}$ has rows mutually indiscernible over $E$.
It follows from Fact \ref{strondisjun}, Theorem \ref{descomposition} and indiscernibility, that we can suppose that for all $0 \leq l <  n+1, j < \kappa$ there is a multi-interval $I_{l,j}= \displaystyle{\bigcap_{i = 1}^n (I^{i}_{l, j}\cap M)}$ such that:

\begin{enumerate}
  \item [(a)] $\phi_l(M, a_{l, j})\subseteq I_{l,j}$,
	\item [(b)] $\{ x \in I_{l,j}:M \models \phi_l(x, a_{l, j})\}$ is multi-dense in $I_{l, j}$, 
	\item[(c)] $I^{i}_{l,j} \cap M$ is definable in $M$ by a quantifier-free $\Li$-formula with parameters in $E(a_{l, j})$.
\end{enumerate}

As $(\bar{a_l}, \phi_{l}(x,y) ,k_l)_{0 \leq l < n+1}$ is an inp-pattern, for all $f: \{0, \ldots, n\} \rightarrow \kappa$, $\displaystyle{\bigcap_{l=0}^nI_{l,f(l)}} \not = \emptyset$.

\begin{claim}
There exists $0\leq l \leq n$ such that $\displaystyle{\bigcap_{j \in \kappa}{I^{i}_{l,j}}} \neq \emptyset$, for all $1\leq i \leq n.$

\begin{proof}
Define for all $i  \in\{1, \ldots,n\}$, $A_i:= \{l \in \{0, \ldots,n\}: \displaystyle{\bigcap_{j \in \kappa}I^i_{l,j}}= \emptyset\}$.
As the burden of $RCF$ is $1$ (Fact \ref{RCFstrong}), $|A_i|\leq 1$  for all $i \in \{1, \ldots,n\}$.
Then $\displaystyle{|\bigcup_{i=1}^n A_i|} \leq n$, so there is $l \in \{0, \ldots, n\}$ such that $l \not \in \displaystyle{\bigcup_{i=1}^n A_i}$. 
Then $\displaystyle{\bigcap_{j \in \kappa}{I^{i}_{l,j}}} \neq \emptyset$, for all $1\leq i \leq n$.
\end{proof}
\end{claim}

Let $l \in \{0, \ldots, n\}$ satisfy the claim. We will only consider the row $l$ of the array so we denote $(a_j)_{j \in \omega}:= (a_{l,j})_{j \in \omega}$.

It follows by saturation that for all $i \in \{1, \ldots,n\}$ there exists a non-empty $<_i$-open interval $I^i \subseteq \clos{M}{i}$ such that:

\[I^i \cap M \subseteq \displaystyle\bigcap_{j \in \kappa} (I^i_{l,j}\cap M).\]
By density of $M$ in every real closure (Fact \ref{PRCcaracte}) we can suppose that $I^i:= (c_i,d_i)_i$, with $c_i,d_i \in M$.
Let $I := \displaystyle{\bigcap_{i=1}^n (I^i\cap M)}$; by the Approximation Theorem (\ref{ApTh}) $I \not = \emptyset$. 
Let $t> k_l$, we can suppose that $\kappa$ is large enough. By Erd\H{o}s-Rado we can find a countable sequence $(b_j)_{j \in \omega}$, indiscernible over $E':=\acl(E(c_i,d_i): 1 \leq i \leq n)$ and such that the first $t$ elements are in $(a_j)_{j \in \kappa}$.
So we have that $\{\phi_l(x, b_j)\}_{j \in \omega}$ is $k_l$-inconsistent.

Observe that for all $j \in \omega$, $\{x \in I: M \models \phi_l(x, b_{j}) \}$ is multi-dense in $I$.
Since $I$ is $\LC(E')$-definable and $(b_j)_{j \in \omega}$ is indiscernible over $E'$, by Theorem \ref{strongdensity} $\{\phi_l(x, b_j)\}_{j \in \omega}$ is consistent.
This contradicts the $k_l$-inconsistency.
\end{proof}
\end{thm}

\begin{thm}\label{PRCNTP2}
Let $M$ be a PRC field. Then $Th_{\mathcal{L_R}}(M)$ is NTP$_2$ if and only if $M$ is bounded. 
\begin{proof}
$(\Rightarrow)$ Suppose by contradiction that $M$ is unbounded. Then $M(\sqrt{-1})$ is a PAC field. As $M(\sqrt{-1})$ is unbounded and interpretable in $M$, by Theorem $\ref{PACNTP2}$ it has TP$_2$, and so $M$ has $TP2$. This contradicts the fact that $Th_{\mathcal{L_R}}(M)$ is NTP$_2$.

$(\Leftarrow)$ Let $M$ be a bounded PRC field. 
By Remark \ref{bddfiniteorders} there exists $n \in \mathbb{N}$ such that $M$ has exactly $n$ different orders. 
If $n= 0$, $M$ is a PAC field, so by Corollary 4.8 of \cite{ChPi} $Th(M)$ is simple and therefore it is NTP$_2$.
If $n \geq 1$ by Theorem \ref{PRCstrong} $Th(M)$ is strong and so is NTP$_2$.
\end{proof}
\end{thm}

\subsection{Burden of types in PRC fields} \label{BurdenPRCfields} 

Let $n \geq 1$ and fix a bounded PRC field $K$ with exactly $n$ orders and which is not real closed. Let $\PRCB:= Th_{\LC}(K)$ (see \ref{PRCB}).
For the rest of the section we are going to work inside a monster model $(M, <_1,\ldots, <_n)$ of $T$.

\begin{lem} \label{bdntrivialtype}
Let $r \in \mathbb{N}$. Let $p(\bar{x}):= \{x_1= x_1 \wedge \ldots \wedge x_r=x_r\}$. Then $\bdn(p(\bar{x}))= nr$.
\begin{proof}
The proof is a generalization of the proof of Theorem \ref{PRCstrong} to several variables.
For all $(i,l) \in \{1, \ldots,n\} \times \{1, \ldots,r\}$, define the formula $\varphi_{(i,l)}(x_1, \ldots,x_r, y):= y <_i x_l <_i y+1$.
For each $(i,l)$ take $(a^i_{l,j})_{j \in \omega}$ such that $a^i_{l,j+1}= a^i_{l,j}+ 1$.
By Remark \ref{ApThC} $(\bar{a}^i_{l}, \varphi_{(i,l)}(x_1, \ldots, x_r,y),2)_{1 \leq i \leq n, 1 \leq l \leq r}$, with $\bar{a}^i_{l}= (\bar{a}^i_{l,j})_{j \in \omega}$ is an inp-pattern of depth $nr$ in $p(\bar{x})$. It follows that $\bdn(p({\bar{x}}))\geq nr$.

Suppose that there is an inp-pattern $(\bar{a}_l, \phi_{l}(x_1, \ldots, x_r, \bar{y}), k_l)_{0 \leq l \leq nr}$ of depth $nr +1$.

By compactness we can take $\bar{a}_l:= (a_{l,j})_{j \in \kappa}$, with $\kappa$ a sufficiently large cardinal.
We can suppose that for all $0 \leq l \leq nr$, $\phi_l(x_1, \ldots, x_r, \bar{y})$ has parameters in a countable set $E$ with $E = \acl(E)\subseteq M$ and  $\G(E) \cong \G(M)$. 
By Fact \ref{strongindis} we can suppose that the array $(\bar{a}_l)_{0 \leq l \leq nr}$ has rows mutually indiscernible over $E$.

It follows from Fact \ref{strondisjun}, Theorem \ref{descomposition2} and indiscernibility that we can suppose that for all $0 \leq l \leq  nr, \; j < \kappa$ there is $C_{l,j}= \displaystyle{\bigcap_{i = 1}^n (C^{i}_{l, j}\cap M^r)}$ a multi-cell in $M^r$ such that:

\begin{enumerate}
  \item [(a)] $\phi_l(M, a_{l, j})\subseteq C_{l,j}$
	\item [(b)] $\{ (x_1, \ldots, x_r) \in C_{l,j}: M \models \phi_l(x_1, \ldots, x_r, a_{l, j})\}$ is multi-dense in $C_{l, j}$. 
	\item[(c)] $C^{i}_{l,j} \cap M^r$ is definable in $M$ by a quantifier-free $\Li$-formula with parameters in $E(a_{l, j})$.
\end{enumerate}

\begin{claim}
There exists $0\leq l \leq nr$ such that $\displaystyle{\bigcap_{j \in \kappa}{C^{i}_{l,j}}} \neq \emptyset$, for all $1\leq i \leq n.$
\end{claim}
\begin{proof}

Define for all $i  \in\{1, \ldots,n\}$, $A_i:= \{l \in \{0, \ldots,nr\}: \displaystyle{\bigcap_{j \in \kappa}C^i_{l,j}}= \emptyset\}$.
The fact that the burden in the theory of real closed fields of the type $\{x_1=x_1 \wedge \ldots \wedge x_r=x_r\}$ is $r$ (Fact \ref{RCFstrong}), implies that for all $i \in \{1, \ldots,n\}$, $|A_i|\leq r$.
Then $\displaystyle{|\bigcup_{i=1}^n A_i|} \leq nr$, so there is $l \in \{0, \ldots, nr\}$ such that $l \not \in \displaystyle{\bigcup_{i=1}^n A_i}$. 
Then $\displaystyle{\bigcap_{j \in \kappa}{C^{i}_{l,j}}} \neq \emptyset$, for all $1\leq i \leq n$.
\end{proof}

Let $l \in \{0, \ldots, nr\}$ satisfy the claim. We will only consider the row $l$ of the array so we denote $(a_j)_{j \in \omega}:= (a_{l,j})_{j \in \omega}$.
It follows by saturation that for all $i \in \{1, \ldots, n\}$ there exists an $<_i$-box $U^i$ in $(\clos{M}{i})^r$ such that:
\[U^i \cap M^r \subseteq \displaystyle{\bigcap_{j \in \kappa}C_{l,j}^i \cap M^r} \]

By density of $M$ in every real closure (Fact \ref{PRCcaracte}) we can suppose that $U^i$ is definable with parameters $\bar{c}_i$ in $M$.
Denote by $U:= \displaystyle{\bigcap_{i=1}^n (U^i\cap M^r)}$.

Let $t> k_l$, we can suppose that $\kappa$ is large enough. By Erd\H{o}s-Rado we can find a countable sequence $(b_j)_{j \in \omega}$, indiscernible over $E':=\acl(E(\bar{c_i}): i \leq n)$ and such that the first $t$ elements are in $(a_j)_{j \in \kappa}$.
So we have that $\{\phi_l(x, b_j)\}_{j \in \omega}$ is $k_l$-inconsistent.

Observe that for all $j \in \omega$, $\{\bar{x} \in U: M \models \phi_{l}(\bar{x}, b_{j})\}$ is multi-dense in $U$.
Since $U$ is $\LC(E')$-definable and $(b_j)_{j \in \omega}$ is indiscernible over $E'$, by Theorem \ref{lemsq+v} $\{\phi_l(\bar{x}, b_j): j \in \omega\}$ is consistent. 
This is a contradiction with the $k_l$-inconsistency.
\end{proof}
\end{lem}

\begin{lem} \label{bdntypesind}
Let $r \in \mathbb{N}$, $A \subseteq M$ and $\bar{a}:= (a_1, \ldots, a_r) \in M^r$ such that $trdeg(A(\bar{a})/A)=r$.
Then $\bdn(\bar{a}/A) = nr$.
\begin{proof}
Let $p(\bar{x}):= \{x_1= x_1 \wedge \ldots \wedge x_r=x_r\}$.
By Lemma \ref{bdntrivialtype} $\bdn(p(\bar{x})) = nr$. As $p(\bar{x}) \subseteq \tp(\bar{a}/A)$, we obtain that $\bdn(\bar{a}/A) \leq nr$. 
We will show that $\bdn(\bar{a}/A)\geq nr$.

Since $\tp(a_1/A(a_2, \ldots,a_r))$ is not algebraic, there is  a sequence $(b_{1,j})_{j \in \omega}$  in $M$ such that:
\begin{enumerate}
 \item $\tp(b_{1,j}/A(a_2, \ldots,a_r))= \tp(a_1/A(a_2, \ldots,a_r))$, for all $j \in \omega$,
 \item $(b_{1,j})_{j \in \omega}$ is indiscernible over $A(a_2, \ldots,a_r)$.
\end{enumerate}

Then using exchange property of $\acl$, we have that $\tp(a_2/ A(a_3, \ldots,a_r, b_{1,j}: j \in \omega))$ is not algebraic.
By induction we can find for all $l \in \{1, \ldots, r\}$ a sequence $(b_{l,j})_{j \in \omega}$ such that:
\begin{enumerate}
\item $\tp(b_{l,j}/A(a_{l+1},...,a_r, b_{1,j},..., b_{l-1,j}: j \in \omega))= \tp(a_l/A(a_{l+1},...,a_r, b_{1,j},..., b_{l-1,j}: j \in \omega))$, for all $j \in \omega$,
\item $(b_{l,j})_{j \in \omega}$ is indiscernible over $A(a_{l+1}, \ldots,a_r, b_{1,j}, \ldots, b_{l-1,j}: j \in \omega)$.
\end{enumerate}
\begin{claim1}
 For all $j_1, \ldots,j_r \in \omega$, $\tp(b_{1,j_1}, \ldots, b_{r,j_r}/A)= \tp(a_{1}, \ldots, a_{r}/A)$.
 \begin{proof}
 If $\varphi(x_1, \ldots,x_r) \in \tp(a_1, \ldots, a_r/A)$, then $\varphi(x_1, a_2, \ldots, a_r) \in \tp(a_1/Aa_2, \ldots a_r)$. 
 Then $\varphi(b_{1,j_1}, x_2, a_3, \ldots, a_r) \in \tp(a_2/ Aa_3, \ldots,a_r,b_{1,j}: j \in \omega) = \tp(b_{2,j_2}/ Aa_3, \ldots,a_r,b_{1,j}: j \in \omega)$, since $\tp(a_1/Aa_2, \ldots a_r) = \tp(b_{1,j_1}/Aa_2, \ldots,a_r)$.
 Iterate the procedures to get the result.

  \end{proof}
\end{claim1}

Fix $(i,l) \in \{1, \ldots,n\} \times \{1, \ldots,r\}$. There is a sequence  $(I^i_{l,j})_{j \in \omega}$ of $<_i$-open intervals in $M$ such that $b_{l,j} \in I^i_{l,j}$, and if $j_1 \not = j_2$, then $I^i_{l,j_1}\cap I^i_{l,j_2}= \emptyset$.
Let $\alpha^i_{l,j}$ be the pair of extremities of $I^i_{l,j}$ and $\varphi_{(i,l)}(x_1, \ldots,x_r, \bar{y})$ the $\Li$-formula such that:
\[M \models \varphi_{(i,l)}(x_1, \ldots, x_r, \alpha^i_{l,j}) \leftrightarrow x_{l} \in I^{i}_{l,j}\].

\begin{claim2}
Let $\beta_l^i:= (\alpha^i_{l,j})_{j \in \omega}$.
Then $(\beta_l^i, \varphi_{(i,l)}(\bar{x},\bar{y}),2)_{1 \leq i \leq n, 1\leq l \leq r}$ is an inp-pattern of depth $nr$ in $\tp(\bar{a}/A)$.
\begin{proof}
\begin{enumerate}
 \item Fix $(i,l) \in \{1, \ldots,n\}\times \{1, \ldots,r\}$. Then $\{\varphi_{(i,l)}(\bar{x}, \alpha^i_{l,j})\}_{j \in \omega}$ is $2$-inconsistent: Clear from the fact that $I^i_{l,j_1}\cap I^i_{l,j_2}= \emptyset$.
\item Let $f: \{1, \ldots,n\}\times \{1, \ldots,r\}\rightarrow \omega$. Then $\{\varphi_{(i,l)}(\bar{x}, \alpha^i_{l, f(i,l)})\}_{1 \leq i \leq n, 1\leq l\leq r} \cup \tp(\bar{a}/A)$ is consistent:

$\{\bar{x}:\varphi_{(i,l)}(\bar{x}, \alpha^i_{l, f(i,l)})\}_{1 \leq i \leq n, 1\leq l\leq r}= \{(x_1, \ldots, x_r): x_l \in I^i_{l, f(i,l)}\}_{1 \leq i \leq n, 1\leq l\leq r} =$ 

$\{(x_1, \ldots, x_r): (x_1, \ldots, x_r) \in \displaystyle{\bigcap_{i=1}^n(I^i_{1, f(i,1)}\times \ldots \times I^i_{r,f(i,r)})}\}$.

Let $\bar{b}_i:= (b_{1,f(i,1)}, \ldots, b_{r,f(i,r)})$ and $C^i:= I^i_{1, f(i,1)}\times \ldots \times I^i_{r,f(i,r)}$. 
By definition $\bar{b}_i \in  C^i$ and by Claim 1 $\tp(\bar{b}_i/A)= \tp(\bar{a}/A)$. 
Then by Lemma \ref{lemqftdense} $q(\bar{x}):= \{x \in \displaystyle{\bigcap_{i=1}^n C^i}\} \cup \{\tp(\bar{a}/A)\}$ is consistent.
\end{enumerate}
\end{proof}
\end{claim2}

Claim 2 implies that $\bdn(\bar{a}/A)\geq nr$. Hence $\bdn(\bar{a}/A) = nr$.
\end{proof}
\end{lem}

\begin{thm}\label{bdntypesgen}
 Let $r \in \mathbb{N}$ and $\bar{a}:=(a_1, \ldots, a_r) \in M^r$. Then $\bdn(\bar{a}/A)= n \cdot trdeg(A(\bar{a})/A)$.
 Therefore the burden is additive $(i.e. \, \bdn(\bar{a}\bar{b}/A)= \bdn(\bar{a}/A) + \bdn(\bar{b}/A\bar{a}))$.
 \begin{proof}
Let $k = trdeg(A(\bar{a})/A)$.
As before we can easily build an inp-pattern of depth $nk$ in $\tp(\bar{a}/A)$, and so $\bdn(\bar{a}/A) \geq nk$.

Suppose without loss of generality that $\{a_1, \ldots,a_k\}$ is a  transcendence basis of $A(\bar{a})/A$.
Then $\bdn(a_{k+1}, \ldots, a_r/A(a_1,\ldots, a_k))=0$, since $a_{k+1}, \ldots, a_r \in \acl(A(a_1, \ldots, a_k))$. 
By Lemma \ref{bdntypesind}, $\bdn(a_1, \ldots, a_k/A)= nk$, and by \ref{factbdntypes} $\bdn(\bar{a}/A) \leq nk$.

Since the transcendence degree is additive, the burden is additive.
 
 \end{proof}

\end{thm}

\subsection{Resilience of PRC fields}\label{sectionreliance}
\begin{defn} \label{defresilient} \cite[Definition 4.8]{ChBY}
Let $\mathcal{L}$ be a language and let $T$ be a complete $\mathcal{L}$-theory.
We say that $T$ is \emph{resilient} if we cannot find indiscernible sequences $\bar{a}= (a_j)_{j \in \mathbb{Z}}$, $\bar{b}= (b_l)_{l \in \mathbb{Z}}$, and a formula $\phi(x,y)$ such that:
\begin{enumerate}
 \item $a_0= b_0$,
 \item $\bar{b}$ is indiscernible over $(a_j)_{j \not = 0}$,
 \item $\{\phi(x, a_j)\}_{j \in \mathbb{Z}}$ is consistent, 
 \item $\{\phi(x, b_l)\}_{l \in \mathbb{Z}}$ is inconsistent. 
 \end{enumerate}
\end{defn}

\begin{rem}\label{remresilient}
\begin{enumerate}
\item It follows by compactness that we get an equivalent definition if we replace $\mathbb{Z}$ by $\kappa$ for the sequence $(b_l)$, where $\kappa\geq \omega$. 
\item \cite{Che3} If $\bar{a}= (a_j)_{j \in \mathbb{Z}}$, $\bar{b}= (b_l)_{l \in \mathbb{Z}}$ and $\phi(x,y)$ satisfy the conditions of Definition \ref{defresilient} and $\phi= \phi_1 \vee \phi_2$, then there is $t \in \{1, 2\}$, $\bar{a'}= (a'_j)_{j \in \mathbb{Z}}$ and $\bar{b'}= (b'_j)_{j \in \mathbb{Z}}$ such that $\phi_t(x, y)$, $\bar{a'}$ and $\bar{b'}$ satisfy the conditions of Definition \ref{defresilient}.
\item \cite{Che3} If $T$ is not resilient we can find $\bar{a}$, $\bar{b}$ and $\phi(x,y)$ with $|x|=1$ satisfying the conditions of definition \ref{defresilient}.

\end{enumerate}
\end{rem}

\begin{fact} \cite[Proposition 4.11]{ChBY} \label{factresilient}
\begin{enumerate}
 \item If $T$ is NIP, then it is resilient.
 \item If $T$ is simple, then it is resilient.
 \item If $T$ is resilient, then it is NTP$_2$.
\end{enumerate}
\end{fact}

\begin{thm}\label{PRCresilient}
Let $n \in \mathbb{N}$, let $M$ be a bounded PRC field with exactly $n$ orders and let $T:= Th_{\LC}(M)$. Then $T$ is resilient
\begin{proof}
If $n = 0$ $M$ is PAC, so $T$ is simple and by Fact \ref{factresilient} it is resilient.
If $M$ is real closed, then $T$ is NIP and by Fact \ref{factresilient} it is resilient.
Suppose that $M$ is neither PAC nor real closed. Suppose that $M$ is sufficiently saturated.
Let $\kappa$ be a sufficiently large cardinal. Suppose by contradiction that there exists $E \subseteq M$, an $\LC(E)$-formula $\phi(x,y)$, and indiscernible sequences over $E$, $\bar{a}= (a_j)_{j \in \mathbb{Z}}$, $\bar{b}= (b_l)_{l \in \kappa}$ such that:
\begin{enumerate}
 \item $a_0= b_0$,
 \item $\bar{b}$ is indiscernible over $(E(a_j):{j \in \mathbb{Z} \setminus \{0\}})$,
 \item $\{\phi(x, a_j)\}_{j \in \mathbb{Z}}$ is consistent, 
 \item $\{\phi(x, b_l)\}_{l \in \kappa}$ is inconsistent. 
 \end{enumerate}

We can suppose that $E = \acl(E)$ and by Remark \ref{remresilient} that $|x|=1$.
 
By Theorem \ref{descomposition}, $\phi(M, a_0)= \displaystyle{\bigvee_{t=0}^m \phi_t(M, a_0)}$ where $\phi_0(M,a_0)$ is finite, and for each $\phi_t(M, a_0)$, $t>0$, there is a multi-interval $I_t(a_0)= \displaystyle{\bigcap_{i=1}^n I^i_t(a_0)\cap M}$ such that:
\begin{enumerate}
\item $\phi_t(M, a_0) \subseteq I_ t(a_0)$,
\item$\{x\in I_t(a_0): M \models \phi(x,a_0)\}$ is multi-dense in $I_{t}(a_0)$ for all  $1 \leq t \leq  m$,
\item the set $I^i_{t}(a_0)\cap M$ is definable in $M$ by a quantifier-free $\Li(Ea_0)$-formula, for all $1 \leq t \leq  m$ and $1 \leq i \leq n$.
\end{enumerate}

By indiscernability, the same is true for all $a_j$, as $a_0= b_0$ and $(b_l)_{l \in \kappa}$ is indiscernible over $E$, the same is also true for all $b_l$.
By Remark \ref{remresilient}(2), we may assume $\phi(M,a_0)= \phi_t(M, a_0)$, for some $1 \leq t \leq m$. If $t=0$ there is nothing to prove so suppose that $t>0$.
Denote by $I(a_j)= I_t(a_j)$ and $I(b_l)= I_t(b_l)$. 

As $\{\phi(x, a_j)\}_{j \in \mathbb{Z}}$ is consistent, $\displaystyle{\bigcap_{j \in \mathbb{Z}}I(a_j)} \not = \emptyset$. 
Then for all $i \in \{1, \ldots, n\}$, $\displaystyle{\bigcap_{j \in \mathbb{Z}}I^i(a_j)} \not = \emptyset$. 
Since $Th(\clos{M}{i})$ is NIP, by Fact \ref{factresilient} it is resilient. This implies that for all $i \in \{1, \ldots, n\}$, $\displaystyle{\bigcap_{l \in \kappa}I^i(b_l)} \not = \emptyset$.
By density of $M$ in each real closure (Fact \ref{PRCcaracte}) and saturation of $M$, there exists a non-empty $<_i$-open interval $I^i \subseteq \clos{M}{i}$,  with extremities in $M$ such that $I^i \subseteq \displaystyle{\bigcap_{l \in \kappa}I^i(b_l)}$. Let $I:= \displaystyle{\bigcap_{i=1}^n I^i \cap M}$. 
Let $c_i,d_i$ be the extremities of $I^i$.

Let $k$ be such that $\{\phi(x, b_l)\}_{l \in \kappa}$ is $k$-inconsistent. By Erd\H{o}s-Rado we can find a countable sequence $(c_l)_{l \in \omega}$, indiscernible over $E':=\acl(Ec_i,d_i: i \leq n)$ and such that the first $k$ elements are in $(b_l)_{l \in \kappa}$.
So we have that $\{\phi_l(x, c_l)\}_{l \in \omega}$ is $k$-inconsistent.

Observe that for all $l \in \omega$, $\{x \in I: M \models \phi(x, c_l)\}$ is multi-dense in $I$.
Since $I$ is $\LC(E')$-definable and $(c_j)_{j \in \omega}$ is indiscernible over $E'$, by Theorem \ref{strongdensity} $\{\phi(x, c_l)\}_{l \in \omega}$ is consistent.
This contradicts the inconsistency and shows that $T$ is resilient.

 \end{proof}
\end{thm}

\subsection{Forking and dividing in bounded PRC fields} \label{forkingdividingPRC} 

\begin{defn}
We fix a theory $T$ and a monster model $\mathbb{M}$ of $T$. Let $A \subseteq \mathbb{M}$ be a small subset and let $a$ be a tuple in $\mathbb{M}$.
\begin{enumerate}
 \item We say that \emph{the formula $\psi(x,a)$ divides over $A$} if there exists $k \in \mathbb{N}$ and an indiscernible sequence over $A$, $(a_j)_{j \in \omega}$, such that:
 $a_0= a$ and $\{\psi(x,a_j): j \in \omega\}$ is $k$-inconsistent.
\item We say that \emph{the formula $\phi(x,a)$ forks over $A$} if there is a number $m \in \mathbb{N}$ and formulas $\psi_j(x,a_j)$ for $j <m$ such that $\phi(x,a) \vdash \displaystyle{\bigvee_{j <m}\psi_j(x,a_j)}$ and $\psi_j(x,a_j)$ divides over $A$ for every $j<m$.   
\item A \emph{type $p$ forks (divides) over $A$} if it implies a formula which forks (divides) over $A$.
\item We say that $A$ is \emph{an extension base} if for all tuples $a$ in $\mathbb{M}$, $tp(a /A)$ does not fork over $A$ .
\end{enumerate}
\end{defn}

 Denote by $a \dnfo_{A} b$ if $\tp(a/Ab)$ does not fork over $A$.

\begin{fact}\label{propforking}
 The following properties are satisfied by the relation $\dnfo$ in any theory:
 \begin{enumerate}
  \item \cite[Remark 2.14]{CK} if $a \dnfo_{Ab}c$ and $b \dnfo_A c$, then $ab \dnfo_A c$,
  \item \cite[Lemma 3.21]{CK} if forking equals dividing over $A$, then $a \dnfo_{A}b$ iff $a \dnfo_{\acl(A)}b$ iff $\acl(Aa) \dnfo_{A}b$ iff $a \dnfo_{A}\acl(Ab)$.
 \end{enumerate}
\end{fact}

\begin{cor}\label{propforking2}
Suppose that for any set $A$, every $1$-type over $A$ does not fork over $A$; then every $A$ is an extension base.
\begin{proof}
This follows immediately from (1) of Fact \ref{propforking} by induction on the arity of the type.
\end{proof}
\end{cor}

\begin{fact}\cite[Theorem 1.1, Corollary 1.3]{CK} \label{NTP2forking2}
Let $T$ be a NTP$_2$ theory.
\begin{enumerate}
 \item Forking equals dividing over any extension base (in particular over any model).
 \item If all sets are extensions bases, then forking equals dividing.
\end{enumerate}

\end{fact}

\begin{fact}\cite[Corollary 2.6]{Sim0} \label{RCFextensionsbases}
In the theory of real closed fields all sets are extensions bases and forking equals dividing.
\end{fact}

\begin{notation}\label{notTheory}
Let $n \geq 1$, as in \ref{PRCB} we fix a bounded PRC field $K$, which is not real closed and has exactly $n$ orders and let $\PRCB:=Th_{\LC}(K)$. 
Let $M$ be a monster model of $\PRCB$, let $a$ be a tuple of $M$ and $A, B \subseteq M$. 
Denote by $a\dnfo^i_AB$ if $\tp^{\clos{M}{i}}(a/AB)$ does not fork over $A$ and by $a\dnfo^{ACF}_AB$ if $a$ is $ACF$-independent of $B$ over $A$.
 \end{notation}

\begin{thm}\label{forkequalsdiv}
In $\PRCB$ all sets are extension bases and forking equals dividing.

\begin{proof}
By Theorem \ref{PRCNTP2} $T$ is NTP$_2$, and by Fact \ref{NTP2forking2} it is enough to show that all sets are extensions bases.
Suppose by contradiction that there exists $A \subseteq M$ and a tuple $a$ in $M$ such that $\tp(a/A)$ forks over $A$.
We can suppose that $a \not \in \acl(A)$. 
By Corollary \ref{propforking2} we can also suppose that $|a|=1$.

Then there are $\phi(x) \in \tp(a/A)$, $m \in \mathbb{N}$ and  $\psi_j(x,a_j)$ for $j <m$ such that:
$\phi(x) \vdash \displaystyle{\bigvee_{j <m}\psi_j(x,a_j)}$ and $\psi_j(x,a_j)$ divides over $A$ for every $j<m$.
Observe that since $a \not \in \acl(A)$ and $\phi(x) \in \tp(a/A)$, we obtain that $|\phi(M)|= \infty$ and therefore there is $j<m$ such that $|\psi_j(M,a_j)|= \infty$.

For each $j < m$, by Theorem \ref{descomposition} there are a finite set $A_j \subseteq \psi_j(M,a_j)$, $t_j \in \mathbb{N}$ and multi-intervals $I_1, \ldots, I_{t_j}$, definable with parameters in $\acl(A(a_j))$ such that:
$A_j \subseteq  \acl(Aa_j)$, $\psi_j(M, a_j) \subseteq \displaystyle{\bigcup_{l=1}^{t_j}I_l} \cup A_j$ and $\{x \in I_l: M \models \psi_j(x,a_j)\}$ is multi-dense in $I_l$, for all $l \leq t_j$.
Then $\psi_j(x,a_j)$ is equivalent to $\displaystyle{\bigvee_{l=1}^{t_j}(\psi_j(x,a_j) \land x \in I_l}) \vee x \in A_j$. 
Observe that since $\psi_j(x,a_j)$ divides over $A$, $\psi_j(x,a_j) \wedge x \in I_l $ divides over $A$ for all $l \leq t_j$.

Therefore we can suppose that $\phi(x) \vdash \displaystyle{\bigvee_{j <m}\psi_j(x,a_j)} \vee x \in B$ where the following is satisfied:
\begin{enumerate}
 \item $B$ is a finite subset of $\displaystyle{\bigcup_{j<m}\acl(Aa_j)}$,
 \item for all $j<m$ $\psi_j(x,a_j)$ divides over $A$, 
 \item for all $j<m$ there is a multi-interval $I_j$, definable in $M$ with parameters in $\acl(A(a_j))$ such that: $\psi_j(M, a_j) \subseteq I_j$ and $\{x \in I_j: M \models \psi_j(x,a_j)\}$ is multi-dense in $I_j$.
\end{enumerate}

\begin{claim}
There is $j<m$ and a multi-interval $J$ definable over $A$ such that $J \subseteq I_j$.
\begin{proof}
As $a \not \in \acl(A)$ and $\phi(x) \in \tp(a/A)$, by Proposition \ref{thmdensite} there exists a multi-interval $I$, definable over $A$ such that $a \in I$ and $\{x \in I: M \models \phi(x)\}$ is multi-dense in $I$.

Let $J \subseteq I$ be a multi-interval definable over $A$. 
Since $B$ is a finite set, $\phi(x) \vdash \displaystyle{\bigvee_{j <m}\psi_j(x,a_j)} \vee x \in B$, and by multi-density of $\phi(x)$ in $I$, it follows that $J \cap \displaystyle{\bigcup_{j<m}I_j} \not = \emptyset$.
There are infinitely many multi-intervals $J \subseteq I$ definable over $A$:
for all $i \in \{1, \ldots, n\}$, let $I^i$ be an $<_i$-open interval in $\clos{M}{i}$ such that $I = \displaystyle{\bigcap_{i=1}^nI^i \cap M}$. Since $I$ is definable over $A$, then $I^i = (a^i, b^i)_i:= \{x \in \clos{M}{i}: a^i <_i x <_i b^i\}$, with $a^i, b^i \in acl^{\clos{M}{i}}(A)$. Let $m \in \mathbb{N}$ be such that $(a^i+ \frac{1}{m}, b^i - \frac{1}{m})_i \subset (a^i, b^i)_i$, for all $1 \leq i \leq n$. 
Then for all $k \geq m$, ${\bigcap_{i=1}^n (a^i+ \frac{1}{k}, b^i - \frac{1}{k})_i}$ is a multi-interval in $I$, definable over $A$.

Then there exists $j<m$, such that $|I_j \cap \acl(A)| = \infty$.
Thus there is a multi-interval $J$, definable over $A$ such that $J \subseteq I_j$.

\end{proof}
\end{claim}

As $\psi_j(x, a_j)$ divides over $A$, there are $k \in \mathbb{N}$ and an indiscernible sequence over $A$, $(a_{j,l})_{l \in \omega}$ such that $a_{j,0}= a_j$ and $\{\psi_j(x, a_{j,l}): l \in \omega\}$ is $k$-inconsistent.
As $J \subseteq I_j$ and $\{x \in I_j: M \models \psi_j(x,a_{j,0})\}$ is multi-dense in $I_j$, we have that $\{x \in J: M \models \psi_j(x,a_{j,0})\}$ is multi-dense in $J$.
As $J$ is definable over $A$, by Theorem \ref{strongdensity} $\{\psi_j(x,a_{j,l}): l \in \omega\}$ is consistent.
This contradicts the $k$-inconsistency. 

 \end{proof}
\end{thm}

\begin{thm}\label{forkingPRC}
Let $M$ be a model of $\PRCB$, let $a$ be a tuple in $M$ and $A \subseteq B \subseteq M$.
Then  $a \dnfo_AB$ if and only if $a \dnfo^i_AB$, for all $1 \leq i \leq n$.
\begin{proof}
We can suppose that $M$ is sufficiently saturated, and that $A= \dcl(A)$, $B= \dcl(B)$ . 

$(\Leftarrow)$: Suppose that $a \dnfo^i_AB$, for all $1 \leq i \leq n$. 
Observe that $a\dnfo^i_AB$ implies $a\dnfo^{ACF}_AB$. By Fact \ref{propforking} and Fact \ref{RCFextensionsbases} we can suppose that $trdeg(A(a)/A)=trdeg(B(a)/B)= |a|$.

Suppose by contradiction that $tp(a/B)$ forks over $A$. 
Then there exists $b \subseteq B$ and $\phi(x,b) \in \tp(a/B)$ such that $\phi(x,b)$ forks over $A$. 
By Theorem \ref{forkequalsdiv}, $\phi(x,b)$ divides over $A$. 
Let $\kappa$ be a sufficiently large cardinal; by compactness there exists $k \in \mathbb{N}$ and an indiscernible sequence over $A$, $(b_{j})_{j \in \kappa}$ such that $b_{0}= b$ and $\{\phi(x,b_{j}): j \in \kappa\}$ is $k$-inconsistent.

As $M \models \phi(a,b)$ and  $trdeg(B(a)/B)= |a|$, by Theorem \ref{descomposition2} there is a multi-cell $C:= \displaystyle{\bigcap_{i=1}^n (C^i \cap M^{|a|})}$, such that $C^i$ is quantifier-free $\Li(b)$-definable, $a \in C$ and $\{x \in C: M \models \phi(x, b)\}$ is multi-dense in $C$.
Let $\psi(x, b)$ be a quantifier free $\LC$-formula such that $M \models x \in C \leftrightarrow \psi(x, b)$ and let $C_j= \displaystyle{\bigcap_{i=1}^n(C_j^i\cap M^{|a|})}$ be the multi-cell definable in $M$ by the formula $\psi(x,b_j)$.
Then by indiscernibility; for all $j \in \kappa$, $\{x \in C_j: M \models \phi(x,b_j)\}$ is multi-dense in $C_j$.

By hypothesis for all $1 \leq i \leq n$, $\tp^{\clos{M}{i}}(a/B)$ does not divides over $A$. As $``x \in C^i"\in \tp^{\clos{M}{i}}(a/B)$, it follows that $``x \in C^i"$ does not divide over $A$.
This implies that $\{x \in C^i_j: j \in \kappa\}$ is consistent.

Since for all $j \in \kappa$, $C^i_j$ is $<_i$-open in $M^{|a|}$, by saturation of $M$ and density of $M$ in each real closure $\clos{M}{i}$, there exists a multi-box $D^i \subseteq (\clos{M}{i})^{|a|}$ definable over $M$, such that $D^i \cap M^{|a|}  \subseteq \displaystyle{\bigcap_{j \in \kappa}(C^i_j\cap M^{|a|})}$.
Let $\alpha_i \subseteq M$ such that $D^i$ is $\Li(\alpha_i)$-definable. 

As before using Erd\H{o}s-Rado we can find a countable sequence $(c_j)_{j \in \omega}$, indiscernible over $A'= \acl(A(\alpha_i)_{i \leq n})$  and such that the first $k$ elements are in $(b_{j})_{j \in \kappa}$. So we have that $\{\phi(x, c_j)\}_{j \in \omega}$ is $k$-inconsistent.
Then $(c_{j})_{j \in \omega}$ is indiscernible over $A'$, $D:= \displaystyle{\bigcap_{i=1}^n(D^i\cap M^{|a|})}$ is definable over $A'$ and for all $j \in \omega$, $\{x \in D: M \models \phi(x, c_{j})\}$ is multi-dense in $D$.
Then by Theorem \ref{lemsq+v} $\{\phi(x,c_{j}): j \in \omega\}$ is consistent. This contradicts the $k$-inconsistency.

$(\Rightarrow)$
Suppose that $a \dnfo_{A}B$; this implies that for all $i \in \{1, \ldots, n\}$, $\tp^M_{\Li}(a/B)$ does not fork over $A$.  By Fact \ref{propforking} we can suppose that $\acl(A)=A$.
Suppose by contradiction that there exists $i \in \{1, \ldots,n\}$ such that $\tp_{\Li}^{\clos{M}{i}}(a/B)$ forks over $A$.
Then by Fact \ref{RCFextensionsbases} there is $\phi(x, b) \in \tp_{\Li}^{\clos{M}{i}}(a/B)$ which divides over $A$.
By quantifier elimination in $RCF$, we can suppose that $\phi(x,b)$ is a quantifier-free $\Li$-formula.

Then there exists $k \in \mathbb{N}$ and an $\Li$-indiscernible sequence over $A$, $(b_j)_{j \in \omega} \subseteq \clos{M}{i}$ such that:
$b_0= b$ and $\{\phi(x,b_j): j \in \omega\}$ is $k$-inconsistent.
By Lemma 5.35 of \cite{Sim} we can suppose that $b_j \dnfo^{i}_Ab_0\ldots b_{j-1}$ for all $j \in \omega$, and so $b_j \dnfo^{ACF}_Ab_0\ldots b_{j-1}$, for all $j \in \omega$.
\begin{claim}
Each $<_t$ extends to an order on $A(b_j:j \in \omega)$. 
\begin{proof}
We have that $b_0=b$, $b \in M$, and ${(b_j)}_{j \in \omega}$ is $\Li$- indiscernible. Then $b$ and $b_j$ satisfy the same algebraic formulas over $A$. This implies that we can extend each order $<_t$ in $A(b_j)$, since $A(b)$ is $<_t$-ordered. Suppose that each $<_t$ extends to an order on $A(b_0, \ldots, b_m)$. As $b_{m+1} \dnfo^{ACF}_Ab_0\ldots b_{m}$ and $A(b_{m+1})/ A$ is regular, $A(b_{m+1})$ is linearly disjoint of $A(b_0, \ldots, b_m)$ over $A$.
By Amalgamation theorem for ordered fields (\ref{Amalord}) each $<_t$ extends to $A(b_0, \ldots, b_{m+1})$.
\end{proof}
\end{claim}

Since $M/A$ and $A(b_j:j \in \omega)/A$ are regular, we can suppose that they are linearly disjoint over $A$. 
By the Amalgamation Theorem for ordered fields (\ref{Amalord}), $M(b_j:j \in \omega)$ is a totally real regular extension of $M$. 
So by Fact \ref{PRC} $M$ is existentially closed in $M(b_j:j \in \omega)$. 
Then by saturation there exists an $\Li$-indiscernible sequence $(b'_j)_{j \in \omega} \subseteq M$ over $A$ such that: $\tp(b'_j/A)= \tp(b/A)$  for all $j \in \omega$ and $\{\phi(x,b'_j): j \in \omega\}$ is $k$-inconsistent.
As $\phi(x,b)\in \tp_{\Li}^M(a/B)$, this implies that $\tp_{\Li}^M(a/B)$ fork over $A$. 
This is a contradiction.
\end{proof}
\end{thm}

\subsection{Lascar types}\label{Lascarsection}

\begin{defn}Let $T$ be an $\mathcal{L}$-theory and $\mathbb{M}$ be a model monster of $T$. Let $a$ and $b$ be tuples in $\mathbb{M}$ and $A \subseteq \mathbb{M}$. We write $\Lstp(a/A) = \Lstp(b/A)$ ($a$ and $b$ \emph{have the same Lascar strong type over $A$}) if there are $n \in \omega$, $a=a_0, \ldots, a_n= b$ such that $a_i, a_{i+1}$ start a $A$-indiscernible sequence for each $i <n$.

We let $d_A(a,b)$ be the \emph{Lascar distance}, that is the smallest $n$ as in the definition or $\infty$ if it does not exist.
\end{defn}

\begin{fact}\cite[Lemma 2.9]{HrusPill}\label{distLascarNIP}
If $T$ is NIP and $A$ is an extension base, then $\Lstp(a/A)= \Lstp(b/A)$ if and only if $d_A(a,b) \leq 2$.
\end{fact}

\begin{lem}\label{qfindisseq}
Let $M$ be a sufficiently saturated bounded $PRC$ field with exactly $n$ orders. We consider $M$ as an $\LC$-structure. 
Let $A \subseteq M$ and let $(a_j)_{j \in \omega}$ be a quantifier-free $\LC$-indiscernible sequence over $A$ such that $a_j \ind^{ACF}_A a_0, \ldots, a_{j-1}$ and $\tp(a_0/A)=\tp(a_1/A)$. Then there is an $\LC$-indiscernible sequence $(c_j)_{j \in \omega}$ over $A$ such that $\tp_{\LC}(c_l,c_j/A) = \tp_{\LC}(a_0,a_1/A)$ for all $l < j$.
\begin{proof}
Let $p(x,y)=\tp_{\LC}(a_0,a_1/A)$.
\begin{claim}
 There is a sequence $(b_j)_{j<\omega}$ such that $\qftp_{\LC}((b_j)_{j< \omega}/A)=\qftp_{\LC}((a_j)_{j< \omega}/A)$, and for every $l<j, (b_l,b_j)$ realizes $p$. 
\end{claim}
\begin{proof}
We will construct the sequence $(b_j)_{j<\omega}$ by induction on $j\geq 1$. For $j=1$, there is nothing to prove, we take $b_0=a_0$, $b_1=a_1$.
Assume $b_0,\ldots, b_j$ constructed; it suffices to find $b_{j+1}$ such that $\qftp_{\LC}(b_0,\ldots,b_{j+1}/A)=\qftp_{\LC}(a_0,\ldots,a_{j+1}/A)$, $\tp_{\LC}(b_{j+1}/A(b_0,\ldots,b_{j-1}))=\tp_{\LC}(b_j/A(b_0,\ldots,b_{j-1}))$ and  $\tp_{\LC}(b_j,b_{j+1}/A)=p$.

Let $c$ be such that $\qftp_{\LC}(b_0, \ldots, b_{j}, c/A)= \qftp_{\LC}(a_0, \ldots,a_j, a_{j+1}/A)$. 
Let $B:=A(b_0, \ldots, b_{j-2})$, $c$ realizes $\qftp_{\LC}(c/Bb_j) \cup \qftp_{\LC}(b_j/Bb_{j-1})$. By Theorem \ref{thamalgamation}, $\tp_{\LC}(c/Bb_j) \cup \tp_{\LC}(b_j/Bb_{j-1}) \cup \qftp_{\LC}(c/Bb_{j-1}b_{j})$ is realized by some $c^*$. Then $\tp_{\LC}(c^*/Bb_{j-1})= \tp_{\LC}(b_j/Bb_{j-1})$. 

Since $\qftp_{\LC}(b_j,c^*/B)= \qftp_{\LC}(b_{j-1},c^*/B)$, by Corollary \ref{IT2} there is $b_{j+1}$ such that: 
\begin{enumerate}
 \item $\qftp_{\LC}(b_{j+1}/Bb_{j-1}b_{j})= \qftp_{\LC}(c^*/Bb_{j-1}b_{j})$,
 \item $\tp_{\LC}(b_{j-1},b_{j+1}/B)= \tp_{\LC}(b_{j-1},c^*/B)$,
\item $\tp_{\LC}(b_j,b_{j+1}/B)= \tp_{\LC}(b_{j-1},b_{j}/A)$.
 
 \end{enumerate}

 Then  $\qftp_{\LC}(b_{j+1}/Bb_{j-1}b_{j})= \qftp_{\LC}(a_{j+1}/Aa_0, \ldots, a_{j})$,  $\tp_{\LC}(b_{j-1},b_{j+1}/B)= \tp_{\LC}(b_{j-1},b_j/B)$, and $\tp_{\LC}(b_j,b_{j+1}/A)= \tp_{\LC}(b_{0},b_{1}/A)$.
\end{proof}

By compactness we can extend the sequence of the claim to $(b_j)_{j \in \kappa}$ with $\kappa$ a sufficiently large cardinal. 
Since $\kappa$ is large enough, we can find an $\LC$-indiscernible sequence $(c_j)_{j \in \omega}$ over $A$ such that $\tp_{\LC}(c_0,c_1/A) = \tp_{\LC}(b_0,b_1/A)= \tp_{\LC}(a_0,a_1/A)$.  

\end{proof}

\end{lem}

\begin{thm}\label{LascarPRC}
Let $n\geq1$ and let $M$ be a sufficiently saturated bounded $PRC$ field with exactly $n$ orders. We consider $M$ as an $\LC$-structure. Let $a,b$ be tuples in $M$ and $A \subseteq M$. Then $\Lstp_{\LC}(a/A) = \Lstp_{\LC}(b/A)$ if and only if $d_A(a,b) \leq 2$ if and only if $\tp_{\LC}(a/A)= \tp_{\LC}(b/A)$. 
\begin{proof} 

Suppose that $\tp_{\LC}(a/A)= \tp_{\LC}(b/A)$. Then by Fact \ref{distLascarNIP}, for all $i \in \{1, \ldots,n\}$, $\Lstp_{\Li}^{\clos{M}{i}}(a/A) = \Lstp_{\Li}^{\clos{M}{i}}(b/A)$ and $d_A(a,b) \leq 2$ (in $\clos{M}{i}$).
Then there are $a=a^i_0, a^i_1, a^i_2= b$ such that $a^i_0, a^i_{1}$ and $a^i_1, a^i_{2}$ start an $\Li$-indiscernible sequence (in $\clos{M}{i}$) over $A$, for each $1 \leq i \leq n$.

By the proof of Proposition 5.25 of \cite{Sim} and since $Th(\clos{M}{i})$ is NIP, we can suppose that $a^i_0 \ind^{ACF}_A a^i_1$ and that $a^i_1 \ind_A^{ACF}a^i_2$.

For each $i \in \{1, \ldots,n\}$, let $(c^i_j)_{j \in \omega}$ be an $\Li$-indiscernible sequence (in $\clos{M}{i}$) over $A$, such that $c^i_0=a$, $c^i_1= a^i_1$ and  $c^i_j \ind^{ACF}_A c^i_0, \ldots, c^i_{j-1}$ (see proof of Proposition 5.25 of \cite{Sim}).

For each $i \geq 2$, let $\phi_i: A(c^i_j:j \in \omega)\rightarrow A(c^1_j: j \in \omega)$ be the bijection
 such that $\phi|_A$ is the identity and $\phi_i(c^i_j)= c^1_j$. 
Equip $A(c^1_j: j \in \omega)$ with the unique $\LC$-structure which makes each $\phi_i (i \geq 2)$ into an $\Li$-isomorphism (this not necessary coincide with the $\LC$-structure inherited from $M$).
It follows that $A(c^1_j: j \in \omega)/A$ is a totally real regular extension, and so there is $M^*\succeq M$ which contains an $\LC(A)$-isomorphic copy of $A(c^1_j: j \in \omega)$ (with this new $\LC$-structure).

Since the sequence $(c^1_j)_{j \in \omega}$ is quantifier-free $\LC$-indiscernible over $A$, by Lemma \ref{qfindisseq} we can find an $\LC$-indiscernible sequence  $(d_j)_{j \in \omega}$ over $A$, such that $\tp_{\LC}(d_0,d_1/A)= \tp_{\LC}(c_0^1, c_1^1/A)= \tp_{\LC}(a,c_1^1/A)$.

In the same way there is an $\LC$-indiscernible sequence $(e_j)_{j \in \omega}$ over $A$ such that $\tp_{\LC}(e_0,e_1/A)= \tp_{\LC}(c_1^1,b/A)$.
This implies that $d_A(a,b) \leq 2$. The rest of the assertions are clear.

\end{proof}
\end{thm}

\begin{rem}\label{IndTheoNTP2} 
Chernikov and Ben Yaacov showed in \cite{ChBY} the following Independence Theorem for NTP$_2$ theories (Theorem 3.3 of \cite{ChBY}): \emph{Let $T$ be NTP$_2$ and $A$ an extension base. Assume that $c \ind_A ab$, $a \ind_A bb'$ and $\Lstp(b/A)=\Lstp(b'/A)$. Then there is $c'$ such that $c'\ind_A ab'$, $\tp(c'a/A)=\tp(ca/A)$ and $\tp(c'b'/A)= tp(cb/A)$.}

It is easy to see that in the case of PRC bounded fields the Amalgamation Theorem (Theorem \ref{thamalgamation}) implies the independence theorem for NTP$_2$ theories:
Suppose that $M$ is a bounded $PRC$ field with exactly $n$ orders. We consider $M$ as an $\LC$-structure. Let $A \subseteq M$, and let  $c \ind_A ab$, $a \ind_A bb'$ and $\Lstp_{\LC}(b/A)=\Lstp_{\LC}(b'/A)$. Let $\phi \in Aut(M/A)$ such that $\phi(b)=b'$.

Observe that as $Th(\clos{M}{i})$ is NIP, Theorem 3.3 of \cite{ChBY} is true in each $\clos{M}{i}$, for all $i \in \{1, \ldots,n\}$.
This implies that for all $i \in \{1, \ldots, n\}$, there is $c'_i\ind^i_A ab'$ such that $c'_i$ realizes $\tp^{\clos{M}{i}}_{\Li}(c/Aa) \cup \tp_{\Li}^{\clos{M}{i}}(\phi(c)/Ab')$.

It follows that there is $c' \in M$ which realizes $\qftp^M_{\LC}(c/Aa) \cup  \qftp^M_{\LC}(\phi(c)/Ab')$ and such that $c'\ind^i_A ab'$ for all $i \in \{1, \ldots,n\}$. In particular, $c \ind^{ACF}_A ab'$.

By Theorem \ref{thamalgamation} there is $d \ind^{ACF}_A ab'$ realizing $\tp^M_{\LC}(c/Aa) \cup  \tp^M_{\LC}(\phi(c)/Ab')$.
Then $\tp(da/A) = \tp(ca/A)$ and $\tp(db'/A)=\tp(\phi(c)b'/A)= \tp(cb/A)$.

\end{rem}

%% file: PreliminariesPpC.tex
\section{Preliminaries on pseudo $p$-adically closed fields}\label{PreliminariesPpC}
\subsection{$p$-adically closed fields} 

\begin{defn}
Let $(M,v)$ be a valued field. The valuation $v$ is called  \emph{$p$-adic} if the residue field is $\mathbb{F}_p$ and $v(p)$ is the smallest positive element of the value group $v(M)$. A field $M$  which admits a $p$-adic valuation is called  \emph{formally $p$-adic}; it must be of characteristic $0$.

If $(M,v)$ is a valued field and $v$ is a $p$-adic valuation on $M$, then we say that $(M,v)$ is a  \emph{$p$-adically valued field}.
A $p$-adically valued field $(M,v)$ which has no proper $p$-adically valued algebraic extension is called  \emph{$p$-adically closed}. 
A \emph{$p$-adic closure of $(M,v)$} is an algebraic extension $(\overline{M}^p,\overline{v}^p)$, which is $p$-adically closed.
Each $p$-adically valued field $(M,v)$ has a \emph{$p$-adic closure}, but as opposed to real closures, in the $p$-adic case, in general the $p$-adic closure is not unique. However we have the following criterion: 

If $(L_1,v_1)$ and $(L_2,v_2)$ are $p$-adic closures of $(M,v)$, then $(L_1,v_1)$ and $(L_2,v_2)$ are isomorphic over $M$ if and only if for each $n \in \mathbb{N}$, $L_1^n \cap M = L_2^n\cap M$ (\cite[page 57]{PR}).

\end{defn}

\begin{para} \textbf{Quantifier elimination:}\label{defpval}
Let $\mathcal{L}_{Mac}:= \mathcal{L_R} \cup \{O_v\} \cup \{P_m: m > 1\}$, where $\mathcal{L_R}$ is the language of rings, and $O_v$ and $P_m$ are unary relation symbols. 
If we interpret $O_v$ as the valuation ring and $P_m$ as the set of $m$-th powers, then we can axiomatize the class of $p$-adically closed fields $(pCF$ fields) in the language $\mathcal{L}_{Mac}$.  
By Theorem 1 of \cite{Mac}\label{EQmac} the theory of $p$CF fields admits quantifier elimination in $\mathcal{L}_{Mac}$.

Observe that if $(M,v)$ is a $p$CF field, then the valuation ring $O_v$ is quantifier-free definable in $\mathcal{L_R} \cup \{P_m: m >1\}$:

$(*)$ $M \models O_v(a)$ if and only if $M \models P_m(1 + pa^m)$, where $m= 2$ if $p\not = 2$ and $m=3$ if $p=2$.

It follows that $O_v$ is $\mathcal{L_R}$-definable.
\end{para}

 \begin{fact}\label{aclpCF} 
Let $(M, v)$ be a $p$CF field and let $A \subseteq M$. Then $A^{alg} \cap M = acl(A)= dcl(A)\prec M$.
\begin{proof}
By Proposition 3.4 of \cite{Van1} $acl(A)=dcl(A)$, and its proof shows that $A^{alg} \cap M \models pCF$.
Since $p$CF is model complete (in the language of rings), then $A^{alg} \cap M \prec M$. As $A \subseteq A^{alg} \cap M$, then $acl(A) \subseteq A^{alg} \cap M$. 
\end{proof}
\end{fact}

\begin{defn}\label{def1cell}
Let $(M,v)$ be a $p$-adically closed field, and denote by $v(M)$ the value group.
A \emph{$1$-cell} in $M$ is either a singleton or a set of the form
\[\{x \in M: \gamma_1 < v(x-a) < \gamma_2 \wedge P_n(\lambda(x-a))\wedge \lambda(x-a)\not= 0 \}\]
where $\gamma_1, \gamma_2 \in v(M) \cup\{-\infty, +\infty\}$, $a \in M$, $n \in \mathbb{N}$ and $\lambda \in \mathbb{Z}$.  
Observe that if $V$ is a $1$-cell which is not a singleton, then $V$ is open in the topology generated by the valuation, and we call $V$ an \emph{open $1$-cell}. 

\end{defn}

\begin{fact}\cite[Lemma 4.3]{HM}\cite[Lemma 4.1]{SV} \label{cellpadics}
Let $M$ be a $p$-adically closed field and $A \subseteq M$. 
Then every $A$-definable subset of $M$ is a finite union of disjoint $1$-cells definable with parameters in $A$.
  
\end{fact}

\subsection{Pseudo $p$-adically closed fields}

\begin{defn}A field extension $N/M$ is called \emph{totally $p$-adic} if for every $p$-adic closure $\overline{M}^p$ of $M$, there exists a $p$-adic closure $\overline{N}^p$ of $N$ such that $\overline{M}^p \subseteq \overline{N}^p$. Observe that this is equivalent to: each $p$-adic valuation on $M$ can be extended to a $p$-adic valuation on $N$.

\end{defn}

\begin{fact}\cite[Lemma 13.9]{HJ}\label{PpcExtclose}
For a field $M$ the following are equivalent:
\begin{enumerate}
\item $M$ is existentially closed (relative to $\mathcal{L_R}$) in every totally $p$-adic regular extension.
\item Every non-empty absolutely irreducible variety $V$ defined over $M$ has an $M$-rational point, provided that it has a simple rational point in each $p$-adic closure of $M$.
\end{enumerate}
\end{fact}

\begin{defn}
A field $M$ of characteristic $0$ that satisfies the conditions of Fact \ref{PpcExtclose} is called \emph{pseudo $p$-adically closed} (P$p$C). 
By Lemma 10.1 of $\cite{J4}$  the class of P$p$C fields is elementary in the language $\mathcal{L}_{\mathcal{R}}$.

\end{defn}

\begin{fact}\cite[Theorem 10.8]{J4} \label{PPcfact}
Let $M$ be a P$p$C field and let $v$ be a $p$-adic valuation on $M$. Then: 
\begin{enumerate}
\item The $p$-adic closure of $M$ with respect to $v$ is exactly its henselization. In particular all $p$-adic closures of $M$ with respect to $v$ are $M$-isomorphic.
\item $M$ is dense in the $p$-adic closure $\overline{M}^p$ with respect to $v$.
\item If $v_1$ and $v_2$ are distinct $p$-adic valuations on $M$, then $v_1$ and $v_2$ are independent (i.e. $v_1$ and $v_2$ generate different topologies).
\end{enumerate}
\end{fact}

\begin{lem}\label{PpCacl}
Let $M$ be a P$p$C field and $A \subset M$. Then $A^{alg}\cap M = acl^{M}(A)$.
If in addition $M$ has a definable $p$-adic valuation, then $acl^{M}(A)= dcl^{M}(A)$.
\begin{proof}
The proof is identical to the one of Lemma \ref{PRCacl}, we only need to replace the amalgamation Theorem of orders by Lemma 4 of \cite{Ku}.
\end{proof}
\end{lem}

\subsection{The theory of P$p$C fields with $n$ $p$-adic valuations}
As in the case of PRC fields we are interested in bounded pseudo $p$-adically closed fields and these fields have finitely many $p$-adic valuations: If $M$ is a P$p$C field with infinitely many $p$-adic valuations, by  \ref{defpval} $(*)$ if $p \not = 2 \; (p= 2)$, $M$ would have infinitely many extensions of degree $2$ (resp degree $3$), and so $M$ would not be bounded.  
For this reason we will restrict our attention to pseudo $p$-adically closed fields with exactly $n$ $p$-adic valuations, for a fixed $n \in \mathbb{N}$.   

\begin{defn}\label{defnPpc} 
Let $M$ be a field and let $v_1, \ldots, v_n$ be $n$ $p$-adic valuations on $M$. 
The field $(M, v_1, \ldots, v_n)$ is \emph{$n$-pseudo $p$-adically closed} ($n$-P$p$C) if:
\begin{enumerate}
 \item $M$ is a P$p$C field,
 \item if $i \not= j$, then $v_i$ and $v_j$ are different valuations on $M$,
 \item  $v_1, \ldots, v_n$ are the only $p$-adic valuations of $M$. 
\end{enumerate}
\end{defn}

\begin{notation}
If $(M,v_1, \ldots, v_n)$ is an $n$-P$p$C field, we denote by $\closp{M}{i}$ a fixed $p$-adic closure of $M$ respect to $v_i$.
Recall that by Fact \ref{PPcfact} $\closp{M}{i}$ is unique up to isomorphism.
\end{notation}

\begin{fact}\cite[Lemma 3.6]{EJ}\label{simpointp}
Let $(M,v_1, \ldots, v_n)$ be an $n$-P$p$C field and let $V$ be an absolutely irreducible variety defined over $M$.
For each $1\leq i\leq n$, let $q_i \in V(\closp{M}{i})$ be a simple point. 
Then $V$ contains an $M$-rational point $q$, arbitrary $v_i$-close to $q_i$, for all $i \in \{1, \ldots,n\}$.
\end{fact}

\begin{fact}\cite[Proposition 10.4]{J4}\label{ElenequivP}
Let $(M, v_1, \ldots, v_n)$ and $(N, w_1, \ldots, w_n)$ be $n$-P$p$C fields. Let $L$ be a common subfield of $M$ and $N$. 
Suppose that there exists and isomorphism $\varphi: G(N)\rightarrow G(M)$ such that  $res_{L^{alg}}\varphi(\sigma)= res_{L^{alg}}\sigma$ for each $\sigma \in G(N)$. 
Suppose further that $\closp{M}{i}$(resp $\closp{N}{i})$ is a $p$-adic closure of $M$ (resp $N$) with respect to $v_i$ (resp $w_i$) such that $\varphi(G(\closp{N}{i}))= G(\closp{M}{i})$, for $i \leq n$. Then $(M, v_1, \ldots, v_n) \equiv_L (N, w_1, \ldots, w_n)$. 
\end{fact}

\begin{cor}\label{PpCel}
Let $(M, v_1, \ldots, v_n) \subseteq (N, w_1, \ldots, w_n)$ be two $n$-P$p$C fields.
If $res: G(N)\rightarrow G(M)$ is an isomorphism, then $(M, v_1, \ldots, v_n) \prec (N, w_1, \ldots, w_n)$.
\end{cor}

%% file: BoundedPpCfield.tex
\section{Bounded pseudo $p$-adically closed field}\label{SectionPpC}
In this section we show that the strategies used and the results obtained in section \ref{SectionPRC} for PRC bounded fields can be generalized without much difficulty for P$p$C bounded fields.
The biggest difference is the need to extend the language, since it is necessary to distinguish the $n$-th powers in each $p$-adic closure with respect to each $p$-adic valuation.
For this we work with a generalization of the language of Macintyre for fields with $n$ $p$-adic valuations.

\begin{lem}\label{DefvalP_m}
Let $K$ be a bounded P$p$C field, and let $K_0$ be a countable elementary substructure of $K$. Let $\LCRp$ be the language of rings with constant symbols for the elements of $K_0$.
Then $K$ has only finitely many $p$-adics valuations and each one is definable by an existential $\LCRp$-formula.

\begin{proof}
Let $v$ be a $p$-adic valuation on $K$, and let $\overline{K}^p$ be a $p$-adic closure of $M$ with respect to $v$. Let $\{P_m\}_{ m\in \mathbb{N}}$ be such that $K \models P_m(a)$ if and only if   $\overline{K}^p \models \exists y (y^m= a \wedge a \not = 0).$
By Fact \ref{PPcfact} (1) $\overline{K}^p$ is the henselization of $K$ with respect to the valuation $v$.
Let $K_m$ be the composite field of all the extensions of $K$ of degree $m$. Observe that as $K$ is bounded, $K_m$ is a finite extension of $K$. In $K$ we can interpret without quantifiers in the language $\LCRp$, the structure $(K_m, +, \cdot, G)$ with $G =\{ \sigma|_{K_m}: \sigma \in \G(\overline{K}^p)\}$, and then: 
\[K \models P_m(a) \; \mbox{if and only if} \; K_m\models \exists y (y^m= a \wedge a \not = 0 \wedge \forall \sigma\in G (\sigma(y)= y)).\] 

This implies that for all $m \in \mathbb{N}$, $P_m$ is definable by an existential $\LCRp$-formula, and then by \ref{defpval}$(*)$ $v$ is  definable by an existential $\LCRp$-formula. As all the $p$-adic valuations are independent and are definable in $K_2$ or $K_3$, $K$ has only finitely many $p$-adic valuations.
\end{proof}
\end{lem}

\begin{notation}\label{PpCB}
We fix a bounded P$p$C field $K$, which is not $p$-adically closed and a countable elementary substructure $K_0$ of $K$. 
Then $K_0^{alg}K= K^{alg}$  and $\G(K_0)\cong \G(K)$.

Since $K$ is bounded, by Lemma \ref{DefvalP_m} there exists $n \in \mathbb{N}$ such that $K$ has exactly $n$ distinct $p$-adic valuations. Thus $K$ is an $n$-P$p$C field. We will suppose that $n \geq 1$.

In this section we will work over $K_0$, thus we denote by $\LCRp$ the language of rings with constant symbols for the elements of $K_0$, $\Lip:=  \LCRp \cup \{O_i\} \cup \{P^i_m\}_{m \in \mathbb{N}, m>1}$ and $\LCp:= \LCRp \cup \{O_i\}_{i \leq n} \cup \{P^i_m\}_{i \leq n, m \in \mathbb{N}, m>1}$, where $O_i$ and $P^i_m$ are unary relation symbols. We interpret $O_i$ as the valuation ring corresponding to $v_i$ and define $P_m^i$ so that:
\[K \models P^i_m(a) \; \mbox{if and only if} \; \closp{K}{i}\models \exists y y^m= a \wedge a \not = 0,\] 
where $\closp{K}{i}$ is a $p$-adic closure of $K$ with respect to $K$.

We let $\PpCB:=Th_{\LCp}(K)$. If $M$ is a model of $\PpCB$ we denote by $\closp{M}{i}$ the $p$-adic closure of $M$ with respect to $v_i$.
As in \ref{MCPRC}, using Corollary \ref{PpCel} we obtain that $\PpCB$ is model complete.
Observe that by Lemma \ref{DefvalP_m} the predicates $P^i_m$ and the valuation ring  $O_i$ are definable in the language $\LCRp$ by an existential formula.
\end{notation}
\begin{cor}\label{PpCRmc}
 $Th_{\LCRp}(K)$ is model complete.
\end{cor}

\begin{para}\label{equivtypePpC}
\textbf{Types:}
As in \ref{typePrc2}, using Fact \ref{ElenequivP}, we have a good description of the types in $\PpCB$:
Let $M$ be a model of $\PpCB$, $A$ a subfield of $M$ and $a, b$ tuples in $M$. Then $\tp(a/A)=\tp(b/A)$ if and only if there is an $\LCRp$-isomorphism $\Phi$ between $\acl(A(a))$ and $\acl(A(b))$, which sends $a$ to $b$ and is the identity on $A$. 
\end{para}

\subsection{Density theorem for bounded P$p$C fields} \label{DensityTheoremPpC}
\subsubsection{Density theorem for one variable definable sets}

\begin{defn} \label{defnmulti1cell}
Let $(M, v_1, \ldots, v_n)$ be a model of $\PpCB$.
\begin{enumerate}
 \item A subset of $M$ of the form $C= \displaystyle{\bigcap_{i=1}^n (C^i\cap M)}$, with $C^i$ a non-empty open $1$-cell in $\clos{M}{i}$ (see  definition \ref{def1cell}), is called a \emph{multi-$1$-cell}. By \ref{ApTh} (Approximation Theorem) and Fact \ref{PPcfact}(2) every multi-$1$-cell is non-empty.
 \item A multi-$1$-cell $C= \displaystyle{\bigcap_{i=1}^n (C^i\cap M)}$ such that each $C^i$ is a non-empty $v_i$-ball in $\clos{M}{i}$, is called a \emph{multi-ball}.
 \item A definable subset $S$ of a multi-$1$-cell $C= \displaystyle{\bigcap_{i=1}^n (C^i\cap M)}$ is called \emph{multi-dense} in $C$ if for any multi-ball $J \subseteq C$, $J \cap S \not = \emptyset.$
Note that multi-density implies $v_i$-density in $C^i$, for all $i \in \{1, \ldots,n\}$.
 \end{enumerate}
\end{defn}

\begin{rem}
Observe that if $(M,v)$ is a $p$-adically valued field and $(\overline{M}^p, \overline{v}^p)$ is a $p$-adic closure of $(M,v)$, then for any $1$-cell $C$ in $\overline{M}^p$ definable with parameters $\bar{a} \subseteq M$, the set $C \cap M$ is definable in $M$ by a quantifier-free $\mathcal{L}_{Mac}(\bar{a})$-formula, where $\mathcal{L}_{Mac}=\mathcal{L_R} \cup \{O_v\} \cup \{P_m: m \in \mathbb{N}\}$.
\begin{proof}
As in Remark $\ref{definM}$ using the quantifier elimination of $Th(\overline{M}^p)$ (Fact \ref{EQmac}) and Fact \ref{aclpCF} which says $\acl^{\overline{M}^p}= \dcl^{\overline{M}^p}$.	
\end{proof}
\end{rem}

\begin{prop}\label{thmdensitep}
Let $(M, v_1, \ldots, v_n)$ be a model of $\PpCB$. Let $\phi(x, \bar{y})$ be an $\LCp$-formula, $\bar{a}$ a tuple in $M$ and $b \in M$ such that $M\models \phi(b, \bar{a})$ and $b \notin \acl(\bar{a})$.
Then there is a multi-$1$-cell $C= \displaystyle{\bigcap_{i=1}^n (C^i\cap M)}$ such that:
\begin{enumerate}
\item $b \in C$,
\item $\{x \in C: M \models \phi(x, \bar{a})\}$ is multi-dense in $C$,
\item the set $C^i\cap M$ is definable in $M$ by a quantifier-free $\Lip(\bar{a})$-formula, for all $1 \leq i \leq n$.
\end{enumerate}
\begin{proof}
As in the proof of Proposition $\ref{thmdensite}$ using Theorem \ref{DefvalP_m} and the fact that $\PpCB$ is model complete, we can find $d \in \mathbb{N}$, $\bar{y_0} \in M^d$ and an absolutely irreducible variety $V$ defined over $\acl(\bar{a})$, such that $(b,\bar{y_0})\in V^{sim}(M)$ and  $\{x \in M: \exists \bar{y}  (x,\bar{y})\in V(M)\} \subseteq \phi(M, \bar{a})$.

For each $i \in \{1, \ldots, n\}$ we define:
\[A_i := \{x \in \closp{M}{i}: \exists(y_1, \ldots, y_{d}) \in {(\closp{M}{i})}^d:(x, y_1 \ldots, y_{d}) \; \mbox{is a simple point of} \; V \}.\]
Observe that $A_i$ is $\Li(\bar{a})$-definable in $\closp{M}{i}$ and $b \in A_i$.
By Lemma \ref{cellpadics} there exists a $1$-cell $C^i$ in $\closp{M}{i}$, $\Li(\bar{a})$-definable, such that $b\in C^i$ and $C^i\subseteq A_i.$

As $b \not \in \acl(\bar{a})$, $C^i$ is a $v_i$-open set.
Define $C:= \displaystyle{\bigcap_{i=1}^n (C^i\cap M)}$ and $S:= \{x \in C: M \models \phi(x, \bar{a})\}$. 
As in the proof of Proposition $\ref{thmdensite}$, using Fact $\ref{simpointp}$, we obtain that $S$ is multi-dense in $C$.

\end{proof}
\end{prop}
 
\begin{thm} \label{descompositionp}
Let $(M, v_1, \ldots, v_n)$ be a model of $\PpCB$, let $\phi(x, \bar{y})$ be an $\LCp$-formula and let $\bar{a}$ be a tuple in $M$.
Then there are a finite set $A\subseteq \phi(M, \bar{a})$, $m \in \mathbb{N}$ and $C_1, \ldots, C_m$, with $C_j = \displaystyle{\bigcap_{i=1}^n (C^i_j\cap M)}$ a multi-$1$-cell such that:
\begin{enumerate}
\item $A \subseteq \acl(\bar{a})$,
\item $\phi(M, \bar{a}) \subseteq \displaystyle{\bigcup_{j=1}^m C_ j \cup A}$,
 \item$\{x\in C_j: M \models \phi(x, \bar{a})\}$ is multi-dense in $C_j$ for all $1  \leq j \leq m$,
 \item the set  $C^i_j \cap M$ is definable in $M$ by a quantifier-free $\Lip(\bar{a})$-formula, for all $1  \leq j \leq m$ and $1 \leq i \leq n$.
 \end{enumerate}
\end{thm}
\begin{proof}
Exactly the same proof as in Theorem $\ref{descomposition}$, replacing o-minimality of each real closure by Fact \ref{cellpadics} and Fact \ref{simpoint} by Fact \ref{simpointp}. 
\end{proof}

\subsubsection{Density theorem for several variable definable sets}

\begin{defn}\label{defnmulti1open}
Let $(M, v_1, \ldots, v_n)$ be a model of $\PpCB$ and $r \in \mathbb{N}$. 
\begin{enumerate}
\item  A subset of $M^r$ of the form $U= \displaystyle{\bigcap_{i=1}^n (U^i\cap M^r)}$ with $U^i$ a non-empty $v_i$-open set in $(\closp{M}{i})^r$ is called a \emph{multi-open set in $M^r$} (or only \emph{multi-open set} when $r$ is clear).
Observe that by Remark \ref{ApThC} and density of $M$ in each $\clos{M}{i}$ (Fact \ref{PPcfact}(2)) every multi-open set is not empty.
\item  A definable subset $S$ of a multi-open set $U = \displaystyle{\bigcap_{i=1}^n (U^i \cap M^r)}$ is called \emph{multi-dense} in $U$ if for any multi-open $V \subseteq U$, $V \cap S \not = \emptyset.$

Note that multi-density in $U$ implies $v_i$-density in $U^i$, for all $i \in \{1, \ldots,n\}$.

\end{enumerate}

\end{defn}

\begin{fact}\cite[2.1]{Mac} \label{EQmac2}
Let $(M,v)$ be a $p$-adically closed field. Let $r \in \mathbb{N}$ and $\phi(x_1, \ldots, x_r, \bar{a})$ be an $\mathcal{L}_{Mac}$-formula.
Then $\phi(M,\bar{a})$ is a finite union of $\mathcal{L}_{Mac}(\bar{a})$-definable sets each of which is $v$-open in $M^r$ (in the product topology) or is of the form $\{(x_1, \ldots, x_r): x \in U \wedge p(x_1, \ldots, x_r)=0\}$, where $p(x_1, \ldots, x_r) \in \acl(\bar{a})[\bar{x}], p \not = 0$ and $U$ is $v$-open in $M^r$.  
Observe that by quantifier elimination in the theory of $p$-adically closed fields these sets are in fact quantifier-free $\mathcal{L}_{Mac}(\bar{a})$-definable.
\end{fact}

Theorem \ref{thmdensite2} of section \ref{DensityTheorem} is generalized without difficulty to the class of bounded P$p$C fields. In the proof we only need to replace o-minimality by Fact \ref{EQmac2}, real closures by $p$-adic closures, $<_i$-open cells in $M^d$ by $v_i$-open sets in $M^d$, and multi-cells by multi-open sets. We thus obtain:

\begin{thm}\label{descompositionp2} \label{thmdensite2p}
Let $(M, v_1, \ldots, v_n)$ be a  model of  $\PpCB$ and let $r \in \mathbb{N}$.
Let $\phi(x_1, \ldots, x_r, \bar{y})$ be an $\LC$-formula and $\bar{a}$ be a tuple in $M$.
Then there are a set $V$, $m \in \mathbb{N}$, and $U_1, \ldots, U_m$ with $U_j= \displaystyle{\bigcap_{i=1}^n (U^i_j\cap M^r)}$ a multi-open set such that:

\begin{enumerate}
\item $\phi(M, \bar{a}) \subseteq \displaystyle{\bigcup_{j=1}^m U_ j \cup V }$,
\item the set $V$ is contained in some proper Zariski closed subset of $M^r$, which is definable over $\acl(\bar{a})$,
\item $\{x \in U_j: \phi(\bar{x}, \bar{a})\}$ is multi-dense in $U_j$ for all $1 \leq j \leq m$,
\item the set $U^i_j$ is definable in $M$ by a quantifier-free $\Li(\bar{a})$-formula, for all $1 \leq j \leq m$, $1 \leq t \leq l_i$.
 \end{enumerate}
\end{thm}

\begin{lem}\label{lemqftdensep}
Let $(M, v_1, \ldots, v_n)$ be a model of $\PpCB$.
Let $A \subseteq M$ and let $\bar{a}$ be a tuple of $M$ such that $trdeg(A(\bar{a})/A)= |\bar{a}|$.
For all $i \in \{1, \ldots,n\}$, let $\bar{b}_i \in M^{|\bar{a}|}$ be such that $\qftp_{\Li}(\bar{b_i}/A)= \qftp_{\Li}(\bar{a}/A)$, and let $U^i$ be a non-empty $<_i$-open set in $(\clos{M}{i})^{|\bar{a}|}$ such that $\bar{b_i} \in U^i$.
Then the type $p(\bar{x}):= \{\bar{x} \in \displaystyle{\bigcap_{i=1}^n U^i}\} \cup \tp_{\LC}(\bar{a}/A)$ is consistent. 
\begin{proof}
As in Lemma \ref{lemqftdense}, replace Theorem \ref{descomposition2} by Theorem \ref{descompositionp2}.
\end{proof}

\end{lem}

\subsection{Amalgamation theorems for bounded P$p$C fields}\label{AmalgamationTheoremPpC}

Proposition \ref{lemamalgamation} and Theorems \ref{thamalgamation} and Corollary \ref{IT2} of section \ref{AmalgamationTheorems} can be easily generalized to the class of bounded P$p$C fields. We thus obtain:

\begin{thm}\label{thamalgamationp}
 Let $(M, v_1, \ldots, v_n)$ be a model of  $\PpCB$. Let $E = \acl(E) \subseteq M$. Let $a_1, a_2, c_1,c_2$ be tuples of $M$ such that $E(a_1)^{alg}\cap E(a_2)^{alg}=E^{alg}$ and $\tp_{\LC}(c_1/E)=\tp_{\LC}(c_2/E)$. Assume that there is $c$ $ACF$-independent of $\{a_1,a_2\}$ over $E$ realizing $\qftp_{\LC}(c_1/E(a_1)) \cup \qftp_{\LC}(c_2/E(a_2))$.
Then $\tp_{\LC}(c_1/Ea_1) \cup \tp_{\LC}(c_2/Ea_2) \cup \qftp_{\LC}(c/E(a_1,a_2))$ is consistent.
\end{thm}

\begin{cor}\label{IT2p}
Let $(M, v_1, \ldots, v_n)$ be a model of $\PpCB$. Let $E = \acl(E) \subseteq E$.
Let $a_1, a_2, d$ be tuples of $M$, such that $\tp_{\LCp}(a_1/E)=\tp_{\LCp}(a_2/E)$, $d$ is ACF-independent of $\{a_1,a_2\}$ over $E$ and $\qftp_{\LCp}(d,a_1/E) = \qftp_{\LCp}(d,a_2/E)$. Suppose that $E(a_1)^{alg}\cap E(a_2)^{alg}= E^{alg}$.

Then there exists a tuple $d^{*}$ in some elementary extension $M^*$ of $M$ such that:
\begin{enumerate}
  \item $\qftp_{\LC}(c^*/E(a_1,a_2)) = \qftp_{\LC}(c/E(a_1,a_2)),$
	\item $\tp_{\LCp}(d^{*}, a_1/E) =\tp_{\LCp}(d^{*}, a_2/E),$
	\item $\tp_{\LCp}(d^{*},a_1/E)= \tp_{\LCp}(d, a_1/E)$.
\end{enumerate}
\end{cor}

The proofs of Theorems \ref{thamalgamationp} and \ref{IT2p} are exactly the same as those of Theorems \ref{thamalgamation} and \ref{IT2} respectively.
It is only required to replace the orders by $p$-adic valuations and real closures for $p$-adic closures.
We also note the following:

\begin{para} \label{Remvaluation}
Let $(M, v_1, \ldots, v_n)$ be a model of $T$.  Denote by $\clos{M}{i}$ a fixed $p$-adic closure of $M$ for the valuation $v_i$.
Then:
\begin{enumerate}

 \item If $A,B \subseteq M$ and $\Phi:A \rightarrow B$ is an $\LC$-isomorphism, by Fact \ref{aclpCF} for each $1 \leq i \leq n$ we can extend $\Phi$ uniquely to an $\Li$-isomorphism $\Phi^{i}: \clos{A}{i} \rightarrow \clos{B}{i}$, where  $\clos{A}{i}= A^{alg} \cap \clos{M}{i}$ and $\clos{B}{i} = B^{alg} \cap \clos{M}{i}$.
 \item Let $L/M$ be an algebraic field extension. If there exists a conjugate $H$ of $\G(\clos{M}{i})$ such that $L \subseteq Fix(H)$, then the valuation $v_i$ can be extended to a $p$-adic valuation on $L$. Since all the $p$-adic closures for the valuation $v_i$ are isomorphic, we can extend the predicates $P^i_m$ to $L$ such that for all $a \in M$, $M \models P_m(a)$ if and only if $L \models P_m(a)$. Then $L$ is an $\Li$-extension of $M$. 

\end{enumerate}
\end{para}

%% file: PpCstability.tex
\section{Independence property in P$p$C fields}\label{IPppC}

Theorem \ref{IPPRC} says that the complete theory of a PRC field which is neither real closed nor algebraically closed is not NIP.
Contrary to PRC fields, the algebraic extensions of a P$p$C field are not necessarily P$p$C fields.
So the proof of Theorem \ref{IPPRC} cannot be generalized to P$p$C fields.
For this reason, to prove that the theory of P$p$C fields is not NIP we will give an explicit example of a formula with the independence property.

\begin{thm}\label{nPpCIP}
Let $p>2$ be a prime number.
Let $M$ be a bounded P$p$C field with two distinct $p$-adic valuations $v_1, v_2$. 
Then the formula \[\phi(x, y):= \bigwedge_{i=1}^2 v_{i}(x)> 0 \; \wedge \; \bigwedge_{i=1}^2 v_{i}(y)> 0 \; \wedge \; \exists z (\bigwedge_{i=1}^2 v_i(z-1)> 0 \; \wedge z^2=x+y+1)\] has the independence property.

\begin{proof} Recall that by Theorem \ref{DefvalP_m}, $v_1$ and $v_2$ are existentially $\LCRp$-definable (see \ref{PpCB} for the definition of $\LCRp$). So $\phi(x,y)$ is an $\LCRp$-formula.
Let $m \in \mathbb{N}$ and $k=m+2^m$.
Let $\Gamma:= \mathbb{Z}^k$; then $\Gamma$ is an ordered abelian group with the lexicographic order. 
Let $t$ be an indeterminate, and let $M((t^\Gamma))$ be the set of elements of the form $\displaystyle{\sum_{\gamma \in \Gamma}a_\gamma t^\gamma}$, with $a_\gamma \in M$ and such that $\{\gamma \in \Gamma: a_\gamma \not = 0\}$ is well-ordered.
Then $F:=M((t^\Gamma))$ is a field and the $t$-adic valuation $v_t: F^{*} \rightarrow \Gamma$ given by $v_t(\displaystyle{\sum_{\gamma \in \Gamma}a_\gamma t^\gamma})= \min\{\gamma \in \Gamma: a_\gamma \not=  0\}$ is such that $(F,v_t)$ is Henselian.

For each $1 \leq r \leq m+2^m$, let $z_{r}:=(z_{1r}, \ldots, z_{kr}) \in \Gamma$ such that $z_{jr}= 0$ if $j \not = r$ and $z_{rr}=1$.
For each $0 \leq j< m$, let $x_j=t^{z_{j+1}} \in F$. 
Let $(A_{l})_{1 \leq l \leq 2^m}$ an enumeration of $\mathcal{P}(m)$, for each $1 \leq l \leq 2^m$ let $y_l=t^{z_{m+l}}$.
Then the elements $\{x_j,y_l:0 \leq j < m, 1 \leq l \leq 2^m\}$ are transcendental and algebraically independent over $M$, satisfying $v_t(x_j)>0$ and $v_t(y_l)>0$.

Define $M_0 := M(x_j, y_l: 0 \leq j < m, A_l\in \mathcal{P}(m))$ and $L := M_0(\sqrt{x_j+ y_l+1}:0 \leq j < m, A_l\in \mathcal{P}(m))$.
As $p \not = 2$, for all $0 \leq j<m$ and $A_l\in \mathcal{P}(m)$, $1$ is a residual simple root of $z^2= x_j+y_l+1$; as $(F,v_t)$ is Henselian there exists $z_{j,l} \in F$ such that $z_{j,l}^2= x_j+y_l+1$ and $v_t(z_{j,l}-1)>0$.
Therefore $L \subseteq F$.

Let $v$ be a $p$-adic valuation on $M$. Define the valuation $w$ on $F$ as follows:
If $a = \sum a_{\gamma}t^\gamma$ and $v_t(a)= \gamma_0$, then $w(a):= (v(a_{\gamma_0}),\gamma_0)$.
Then the value group of $w$ is $v(M)\times \Gamma$, and it is ordered with the anti-lexicographic order.
If $v(M)$ has a smallest positive element $1$, then $(1,0)$ is the smallest positive element of $w(F)$. 
This implies in particular that $w$ is also a $p$-adic valuation on $F$. Therefore $F$ is a totally $p$-adic extension of $M$, and as $L \subseteq F$ we obtain that $L$ is a totally $p$-adic extension of $M$. 
In particular $v_1, v_2$ extend to $p$-adic valuations $w_1,w_2$ on $L$ satisfying 
$w_i(x_j)>0$, $w_i(y_l)>0$, and $w_{i}(z_{j,l}-1)>0$ for all $i \in \{1,2\}$, $0 \leq j <m$ and $A_l \in \mathcal{P}(m)$.
\begin{claim}
The fields $M_0(\sqrt{x_j+ y_l+1})$, with $(j,A_l) \in m \times \mathcal{P}(m)$, are linearly disjoint over $M_0$
\begin{proof}

Let $H:= \langle x_j+y_l +1: 0 \leq j < m, A_l\in \mathcal{P}(m), {(M_0^*)}^2\rangle / {(M_0^*)}^2$, a subgroup of $M_0^*/(M_0^*)^2$.

By Kummer theory, $L:=M_0(H^{\frac{1}{2}})$ is a Galois extension of $M_0$ and $\Gal(L/M_0)\cong H$.
So we have that $[L:M_0] = |\Gal(L/M_0)| = |H|$.

Since $H$ is a $\mathbb{Z}/2\mathbb{Z}$-vector space, it is sufficient to show that for all $C \subseteq  m \times \mathcal{P}(m)$, $C \not = \emptyset$,  $\displaystyle{\prod_{(j,A_l) \in C}(x_j+y_l+1)}$ is not a square in $M_0^*$. 

Suppose there are $z \in M_0^*$ and $C \subseteq  m \times \mathcal{P}(m)$, $C \not = \emptyset$ such that: 
\[\prod_{(j,A_l) \in C}(x_j+y_l+1)= z^2.\] 

We have that $ R= M[x_j, y_l: 0 \leq j < m, A_l\in \mathcal{P}(m)]$ is integrally closed over $M_0$, hence $z$ belong to $R$ and in $R$, $(x_j+y_l+1) \mid z^2$, for all $(j, A_l)\in C$.

Let $(r, A_s)\in C$; since $x_r+y_s+1$ is irreducible in $R$, ${(x_r+y_s+1)}^2 \mid z^2$. 
Then, \[(x_r+y_s+1)^2 \mid \prod_{(j,A_l) \in C}(x_j+y_l+1),\] which gives us, by canceling, \[(x_r+y_s+1) \mid \prod_{(j,A_l) \neq (r, A_s), (j,A_l) \in C}(x_j+y_l+1).\]

This contradicts the irreducibility of the $x_r + y_s+1$ in $M_0$. 

Hence $|H|= 2^{m2^m}$, and  $[L:M_0] = 2^{m2^m}$.
\end{proof}
\end{claim}

By the Claim, $\Gal(L/M) \cong (\mathbb{Z}/2\mathbb{Z})^{m2^m}$, so there exists $\sigma \in \Gal(L/M_0)$ such that $\sigma(z_{j,l}) = z_{j,l} \Leftrightarrow j \in A_l$.
Observe that as $w_2$ is a $p$-adic valuation on $L$, $w_2 \circ \sigma$ is also a $p$-adic valuation on $L$ extending $v_2$.
Replace $w_2$ by $w_2 \circ \sigma$.
Then in $(L, w_1, \ldots, w_n)$ we have that: 

$\displaystyle{\bigwedge_{i=1}^2 w_{i}(x_j)> 0 \; \wedge \; \bigwedge_{i=1}^2 w_{i}(y_l)> 0 \; \wedge \; [\exists z (\bigwedge_{i=1}^2 w_i(z-1)> 0 \; \wedge z^2=x_j+y_l+1)}] \Leftrightarrow j \in A_l$.

Observe that since each $v_i$ is $\LCR$-definable, and $\LCR$ contains constant symbols for a fixed submodel of $M$, each $w_i$ is also $\LCR$-definable.

Since $(L, w_1, \ldots, w_n)$ is a totally $p$-adic regular extension of $(M, v_1, \ldots, v_n)$, Lemma \ref{DefvalP_m} and Fact \ref{PpcExtclose} imply that $(M, v_1, \ldots, v_n)$ is existentially closed in $(L, w_1, \ldots, w_n)$. 
Then there are $a_j$, $b_{l}$ in $M$ for every $0 \leq j<m$ and $A_l \in \mathcal{P}(m)$, such that $M\models \phi(a_j, b_l)$ if and only if $ j \in A_l$.
\end{proof}

\end{thm}

\begin{rem}\label{nPpCIPRem}
 The proof for $p=2$ is similar, using instead the formula 
 \[\phi(x, y):= \bigwedge_{i=1}^2 v_{i}(x)> 0 \; \wedge \; \bigwedge_{i=1}^2 v_{i}(y)> 0 \; \wedge \; \exists z (\bigwedge_{i=1}^2 v_i(z-1)> 0 \; \wedge z^3=x+y+1)\]
 and working in $M(\omega)$, with $\omega^2 +\omega +1=0$.
\end{rem}

\begin{defn}
Let $M$ be a field and let $\mathcal{M}$ be a family of separable algebraic extensions of $M$. Assume that $\mathcal{M}$ is closed under the action of $\G(M)$.  We say that $M$ is \emph{pseudo $\mathcal{M}$-closed} $(P\mathcal{M}C)$ if every non-empty absolutely irreducible variety $V$ defined over $M$ 
with an $\overline{M}$-simple rational point for each $\overline{M} \in \mathcal{M}$, has an $M$-rational point. 
\end{defn}

\begin{cor}\label{nPpCIPC}
Let $M$ be a bounded P$p$C field which is not $p$-adically closed. Then $T=Th_{\LCRp}(M)$ (see Notation \ref{PpCB}) is not NIP.
\begin{proof}
Since $M$ is bounded there exists $n \in \mathbb{N}$ such that $M$ has exactly $n$ distinct $p$-adic valuations.
If $n\geq 2$, by Theorem \ref{nPpCIP} and Remark \ref{nPpCIPRem}; $T$ has the $IP$.
Suppose that $n= 1$, let $v$ be the unique valuation on $M$; by Theorem \ref{DefvalP_m} $v$ is definable. Since $M$ is not $p$-adically closed it is not Henselian, and there exists a finite algebraic extension $N$ of $M$ such that $N$ admits two distinct valuations $v_1$, $v_2$, extending $v$.
Observe that as $\LCRp$ contain constant symbols for a fixed submodel of $M$, in $M$ we can interpret the structure $(N,v_1, v_2)$ (see Appendix 1 of \cite{Cha} for more details).
Therefore it is enough to show that $(N, v_1, v_2)$ has the $IP$.

Let $\closp{M}{1}$ be a fixed $p$-adic closure of $(M,v)$ and let $\mathcal{M} := \{\sigma(\closp{M}{1}): \sigma \in \G(M)\}$.
Then $M$ is $P\mathcal{M}C$. By Lemma 7.2 of \cite{J4} $N$ is $P\mathcal{M}(N)C$, where $\mathcal{M}(N)= \{\sigma(\closp{M}{1})N: \sigma \in \G(M) \}$.

We do the case $p>2$, the case $p=2$ can be adapted similarly. Let $m \in \mathbb{N}$ and $k = m + 2^m$. Take $\Gamma$, $x_j$, $y_l$ for $0 \leq j < m$ and $1 \leq l \leq 2^m$  as in the proof of Theorem \ref{nPpCIP}.
Let $N_0 := N(x_j, y_l: 0 \leq j < m, A_l\in \mathcal{P}(m))$ and $L := N_0(\sqrt{x_j+ y_l+1}:0 \leq j < m, A_l\in \mathcal{P}(m))$.

Let $\{v_i\}_{i \in I}$ be the set of valuations in $N$ extending $v$. Denote by $\clos{N}{i}$ the Henselian closure of $(N,v_i)$, and let $T_i:= Th(\clos{N}{i})$.  
As in the proof of Theorem \ref{nPpCIP} we can find definable valuations $w_i$ on $L$ extending $v_i$ such that: 
\begin{enumerate}
 \item $(L,w_i) \models (T_i)_{\forall}$
 \item $L \models \displaystyle{\bigwedge_{i=1}^2 w_{i}(x_j)> 0 \; \wedge \; \bigwedge_{i=1}^2 w_{i}(y_l)> 0 \; \wedge \; [\exists z (\bigwedge_{i=1}^2 w_i(z-1)> 0 \; \wedge z^2=x_j+y_l+1)}]$ if and only if $j \in A_l$
\end{enumerate}

Observe that $N$ is existentially closed in $L$: we can suppose that $L = N(\bar{a})$, and let $V$ be an absolutely irreducible variety such that $\bar{a}$ is a generic point of $V$. Then $V$ has a simple point in each $\clos{N}{i}$, for all $i \in I$, and as $N$ is $P\mathcal{M}(N)C$, it follows that $V$ has an $N$-rational point.

Since the valuations are definable in $N$ (in the language of rings expanded by constant symbols) and $N$ is existentially closed in $L$, we obtain that $(N, v_1, v_2)$ has the $IP$.

\end{proof}

\end{cor}

\section{P$p$C fields and their stability theoretic properties}\label{PpCClasification}

\begin{prop} \label{SQNIPp}
Let $n \geq 1$. In $n$-P$p$C every quantifier-free $\LC$-formula is NIP.  
\begin{proof}
By Lemma \ref{NIPreducts}, since by 1.32(*) every atomic formula is of the form $P^i_m(p(\bar{x}, \bar{y}))$, with $i \in \{1, \ldots, n\}$ and $p(\bar{x}, \bar{y}) \in \mathbb{Q}[\bar{x}, \bar{y}]$, and $pCF$ is NIP.
\end{proof}
\end{prop}

\begin{notation}
We work with the notation of \ref{PpCB}.
Let $n \geq 1$. Fix $K$ a bounded P$p$C field with exactly $n$ $p$-adic valuations, which is not $p$-adically closed and let $\PpCB:=Th_{\LC}(K)$.
Let $M$ be a monster model $\PpCB$ and $\clos{M}{i}$ a fixed $p$-adic closure of $M$ for the valuation $v_i$. 
Denote by $a\dnfo^i_AB$ if $\tp^{\closp{M}{i}}(a/AB)$ does not fork over A and by  $a\dnfo^{ACF}_AB$ if $a$ is $ACF$-independent of $B$ over $A$.
Observe that $a\dnfo^i_AB$ implies $a\dnfo^{ACF}_AB$, for all $i \in \{1, \ldots,n\}$. 

\end{notation}

\begin{thm}\label{lemsq+vp} 
Let $E =\acl(E)\subseteq M$ and $(a_j)_{j\in \omega}$ an indiscernible sequence over $E$.
Let $r \geq 1$, let $\phi(x_1, \ldots,x_r, \bar{y})$ be an $\LC$-formula and $C$ a multi-open set in $M^r$ definable over $E$, such that $\{(x_1, \ldots,x_r) \in C: \phi(x_1, \ldots,x_r, a_0)\}$ is multi-dense in $C$.
Then $p(x_1, \ldots,x_r):= \{\phi(x_1, \ldots,x_r, a_j)\}_{j \in \omega}$ is consistent.
\begin{proof}
The proof of Lemma \ref{lemsq+v} generalizes almost immediately to our context.
The proof of the Claim is identical: replace Corollary \ref{SQNIP} by Proposition $\ref{SQNIPp}$ and observe that by Fact \ref{EQmac2} if $\psi(x_1, \ldots, x_r,\bar{a})$ is a quantifier-free $\LCp$-formula, $c \in M^r$, $M \models \psi(c,\bar{a})$ and $trdeg(\acl(\bar{a})c/ \acl(\bar{a}))=r$, then there is for all $i \in \{1, \ldots, n\}$, a $v_i$-open set $C^i$ in $(\closp{M}{i})^r$, definable over $\acl(\bar{a})$, such that:
$\displaystyle{\bigcap_{i=1}^n}(C^i\cap M^r) \subseteq \psi(M, \bar{a})$ and $c \in \displaystyle{\bigcap_{i=1}^n}(C^i\cap M^r)$.
The rest of the proof is identical, replace Theorem \ref{IT2} by Theorem \ref{IT2p}, Theorem \ref{descomposition2} by Theorem \ref{descompositionp2}, and multi-cells in $M^r$ by multi-open sets in $M^r$.
\end{proof}
\end{thm}

\begin{fact}\cite[Theorem 6.6]{DLG} \label{PCFstrong} 
The theory of $p$-adically closed fields $(pCF$ fields) is strong and if $r \in \mathbb{N}$, $\bdn(x_1= x_1 \wedge \ldots \wedge, x_r= x_r)=r$.
\end{fact}

\begin{thm}\label{PpCstrong}\label{PpCresilient}
The theory $\PRCB$ is strong, resilient and $\bdn(\{x=x\}) = n$.
If $r \geq 1$, and $p(\bar{x}):= \{x_1= x_1 \wedge \ldots \wedge x_r = x_r\}$, then $\bdn(p(\bar{x}))= nr$.
\begin{proof}
We work in a monster model $(M,v_1, \ldots, v_n)$ of $\PRCB$.
For $l \in \{0, \ldots,n-1\}$, define the formula $\varphi_l(x,y):= v_{l+1}(x-y )> 0$.
Find $((a_{l,j})_{j \in \omega})_{l \leq n-1}$, such that $\varphi_l(M,a_{lj_1}) \cap \varphi_l(M,a_{lj_2}) = \emptyset$, for all $j_1 \not = j_2$ and for all $0 \leq l \leq n-1$. 
Using the Approximation Theorem ($\ref{ApTh}$) we obtain that $(\bar{a_l}, \varphi_l(x,y), 2)_{l<n}$, with $\bar{a_l}= (a_{l,j})_{j\in \omega}$ is an inp-pattern of depth $n$. 
It follows that the $\bdn(\{x=x\})$ is greater or equal to $n$.
Observe that this can easily be generalized to several variables and then $\bdn(p(\bar{x}))\geq nr$.

To show that $\bdn(p(\bar{x})) \leq nr$ it is only necessary to replace in the proof of Theorem \ref{bdntrivialtype} multi-cell by multi-open set, Theorem \ref{descomposition2} by Theorem \ref{descompositionp2} and Theorem \ref{lemsq+v} by Theorem \ref{lemsq+vp}. 

The proof of resilient is as in Theorem \ref{PRCresilient}, using the fact that the $p$-adically closed fields are NIP and replacing Theorem \ref{descomposition} by Theorem \ref{descompositionp2} and Theorem \ref{strongdensity} by Theorem \ref{lemsq+vp}.

\end{proof}
\end{thm}

\begin{cor}
If $M$ is a bounded P$p$C field, then $Th(M)$ is NTP$_2$.
\begin{proof}
 This follows immediately from Theorem \ref{PpCstrong} using the fact that if $M$ is bounded, then there exists $n \in \mathbb{N}$ such that $M$ has exactly $n$ $p$-adic valuations, and they are definable.
\end{proof}

\end{cor}

Observe that the proof of the converse of Theorem \ref{PRCNTP2} cannot be generalized to bounded P$p$C fields, as algebraic extensions of P$p$C fields are not necessarily P$p$C.

\begin{thm} \label{bdntypesindp}
Let $r \in \mathbb{N}$ and $\bar{a}:=(a_1, \ldots, a_r) \in M^r$. Then $\bdn(\bar{a}/A)= n \cdot trdeg(A(\bar{a})/A)$.
Therefore the burden is additive $(i.e. \, \bdn(\bar{a}\bar{b}/A)= \bdn(\bar{a}/A) + \bdn(\bar{b}/A\bar{a}))$.

\begin{proof}
Exactly the same proof as for Lemma \ref{bdntypesind} and Theorem \ref{bdntypesgen}, replacing $<_i$-open interval by $1$-cell for the valuation $v_i$ and Lemma \ref{lemqftdense} by Lemma \ref{lemqftdensep}.
\end{proof}
\end{thm}

\begin{fact}\label{PCFextensionsbases}
 In the theory of $p$-adically closed fields ($pCF$) all sets (in the real sort) are extensions bases and forking equals dividing.
\begin{proof}
If $M$ is $p$-adically closed and $A \subseteq M$, then by Proposition 3.4 of \cite{Van1} $\dcl(A)$ is $p$-adically closed.
Using Corollary 3.22 of \cite{CK} we obtain that forking equals dividing. 
\end{proof}

 \end{fact}

\begin{thm}\label{PpCextebases}
In $\PpCB$ all sets are extensions bases and forking equals dividing, and if $a$ is a tuple in $M$ and $A \subseteq B \subseteq M$, then  $a \dnfo_AB$ if and only if $a \dnfo^i_AB$, for all $1 \leq i \leq n$.
\begin{proof}
The proof of Theorems \ref{forkequalsdiv} and \ref{forkingPRC} generalizes immediately to our context. 
We need to replace multi-intervals by multi-1-cells, multi-cells and multi-boxes by multi-open sets and real closures by $p$-adic closures. It is also necessary to change Theorems \ref{descomposition} and \ref{descomposition2} by Theorems \ref{descompositionp} and \ref{descompositionp2}, Theorem \ref{strongdensity} and \ref{lemsq+v}  by Theorem \ref{lemsq+vp} and the amalgamation theorem of orders \ref{Amalord} by Lemma 4 of \cite{Ku}.
\end{proof}
\end{thm}

\begin{thm}
Let $n\geq1$ and let $M$ be a sufficiently saturated bounded $PpC$ field with exactly $n$ $p$-adique valuations. We consider $M$ as a $\LC$-structure. Let $a,b$ tuples in $M$ and $A \subseteq M$. Then $\Lstp_{\LC}(a/A) = \Lstp_{\LC}(b/A)$ if and only if $d_A(a,b) \leq 2$ if and only if $\tp_{\LC}(a/A)= \tp_{\LC}(b/A)$.
\begin{proof}
As in Theorem \ref{LascarPRC}, using the fact that if $\clos{M}{i}$ is the $p$-adic closure of $M$ for the valuation $v_i$, then
$\Lstp^{\closp{M}{i}}_{\Li}(a/A) = \Lstp^{\closp{M}{i}}_{\Li}(b/A)$ if and only if $d_A(a,b) \leq 2$ if and only if $\tp_{\Li}^{\closp{M}{i}}(a/A)= \tp^{\closp{M}{i}}_{\Li}(b/A)$. 
\end{proof}

\end{thm}

%% file: PRC,PpCfieldsandNTP2.bbl
\begin{thebibliography}{99}


\bibitem{Adl} H. Adler, \emph{Strong theories, burden, and weight.} Preprint (2007).  Available in http://www.logic.univie.ac.at/$\sim$adler/docs/strong.pdf.
\itemsep=\smallskipamount

\bibitem{Ax} J. Ax, \emph{The elementary theory of finite fields}. Annals of Mathematics, ser 2, vol 88(1968), 239-271.  
\itemsep=\smallskipamount

\bibitem{Ba0} Serban A. Basarab, \emph{Definite functions on algebraic varieties over ordered fields}. Revue Roumaine de Math\'ematiques Pures et Appliqu\'ees, vol 29(1984), 527-535.  
\itemsep=\smallskipamount

\bibitem{Ba1} Serban A. Basarab, \emph{The absolute galois group of a pseudo real closed field with finitely many orders}. Journal of Pure and Applied Algebra, vol 38(1985), 1-18. 
\itemsep=\smallskipamount

\bibitem{Ba} Serban A. Basarab, \emph{Transfer Principles for Pseudo Real Closed E-Fold Ordered Fields.}. Journal of Symbolic Logic 51 (4)(1986), 981-991. 
\itemsep=\smallskipamount

\bibitem{Cha0} Z. Chatzidakis, \emph{Simplicity and Independence for Pseudo-algebraically closed fields}. Models and Computability, S.B. Cooper, J.K. Truss Ed., London Math. Soc. Lect. Notes Series 259, Cambridge University Press, Cambridge (1999), 41 - 61.
 
\bibitem{Cha} Z. Chatzidakis, \emph{Properties of forking in $\omega$-free pseudo-algebraically closed fields}. J. of Symb. Logic 67, Nr 3 (2002), 957 - 996.

\bibitem{ChPi} Z. Chatzidakis, A. Pillay \emph{Generic structures and simple theories}. Annals of Pure and Applied Logic, vol 95 (1998), 71 - 92.


\bibitem{EaOn} C. Ealy and A. Onshuus, \emph{Characterizing rosy theories}. Journal of Symbolic Logic, vol 72, number 3 (2007), 919-940.

\bibitem{ChBY} I. Ben Yaacov, A. Chernikov, \emph{An independence theorem for NTP$_2$ theories}. J. Symb. Log. 79 (2014), no 1, 135-153.  
\itemsep=\smallskipamount

\bibitem{Che1} A. Chernikov, \emph{Theories without the tree property of the second kind}. Annals of Pure and Applied Logic 165 (2014), 695-723.  
\itemsep=\smallskipamount


\bibitem{Che3} A. Chernikov, \emph{More on NTP$_2$}. (2014) Private communication. 
\itemsep=\smallskipamount

\bibitem{CheHil} A. Chernikov and M. Hils, \emph{Valued difference fields and NTP$_2$}. Israel Journal of Mathematics 204 (1), 299-327 (2014).   
\itemsep=\smallskipamount

\bibitem{CK} A. Chernikov and I. Kaplan, \emph{Forking and dividing in NTP$_2$ theories}.  J. Symbolic Logic 77 (2012), no 1, 1-20.  
\itemsep=\smallskipamount

\bibitem{CKS} A. Chernikov, I. Kaplan and P. Simon, \emph{Groups and fields with NTP$_2$}.  Proceedings of the AMS, 143 (2015), no 1, 395-406.  
\itemsep=\smallskipamount

\bibitem{DLG} A. Dolich, D. Lippel, and J. Goodrick, \emph{Dp-minimal theories: basic facts and examples}.  Notre Dame Journal of
Formal Logic, vol 52 (2011), no 3, 267-288. 
\itemsep=\smallskipamount

\bibitem{Du} J-L. Duret, \emph{Les corps faiblement alg\'ebriquement clos
non s\'eparablement clos ont la propri\'et\'e d'ind\'ependance}. Model theory of Algebra and Arithmetic, Pacholski et al. ed,
Springer Lecture Notes 834 (1980), 135 -157.
\itemsep=\smallskipamount


\bibitem{EaOn} C. Ealy and A. Onshuus, \emph{Characterizing rosy theories}. Journal of Symbolic Logic, vol 72, number 3 (2007), 919-9.

\bibitem{EJ}I. Efrat and M. Jarden, \emph{Free pseudo $p$-adically closed fields of finite corank}. Journal of Algebra 133 (1990), 132-150  
\itemsep=\smallskipamount



\bibitem{Gr} C. Grob, \emph{Die Entscheidbarkeit der Theorie der maximalen pseudo p-adisch abgeschlossenen Korper}. Konstanzer Dissertationen, Band 202 (1987).
\itemsep=\smallskipamount


\bibitem{HM} D. Haskell, D. Macpherson,  \emph{A version of o-minimality for the $p$-adics}. The journal of Symbolic Logic, Volume 62, Number 4 (1997). 
\itemsep=\smallskipamount

\bibitem{Hem} N. Hempel, \emph{On n-dependent groups and fields}. ArXiv Nr: 1401.4880. 
\itemsep=\smallskipamount


\bibitem{Hrus} E. Hrushovski, \emph{Pseudo-finite fields and related structures}. Model theory and applications, Quad. Mat., vol 11, Aracne, Rome (2002), 151-212. 
\itemsep=\smallskipamount


\bibitem{HrusPill} E. Hrushovski, and A. Pillay\emph{On NIP and invariant measures}.J. Eur. Math.
Soc. (JEMS), 13(4):1005?1061, 2011
\itemsep=\smallskipamount


\bibitem{HJ} D. Haran, M. Jarden, \emph{The absolute Galois group of a pseudo $p$-adically closed field }. J.reine angew.Math.383 (1988), 147-206. 
\itemsep=\smallskipamount

\bibitem{J1} M. Jarden, \emph{The elementary theory of large e-fold ordered fields}. Acta mathematica
149 (1982), 239-259. 
\itemsep=\smallskipamount

\bibitem{J2} M. Jarden, \emph{On the model companion of the theory of e-fold ordered fields}. Acta
mathematica 150 (1983), 243-253
\itemsep=\smallskipamount

\bibitem{J3} M. Jarden, \emph{The algebraic nature of the elementary theory of PRC fields}. Manuscripta Mathematicae 60 (1988), 463-475.  
\itemsep=\smallskipamount

\bibitem{J4} M. Jarden, \emph{Algebraic realization of $p$-adically projective groups}. Compositio Mathematica 79 (1991), 21-62.
\itemsep=\smallskipamount

\bibitem{WJ} W. Johnson, \emph{Forking and Dividing in Fields with Several Orderings and Valuations}. Preprint 2014.
\itemsep=\smallskipamount


\bibitem{KOU} I. Kaplan, A. Onshuus and A. Usvyatsov, \emph{Additivity of the dp-rank}.  Trans. Amer. Math. Soc. 365 (2013), no. 11, 5783-5804.
\itemsep=\smallskipamount


\bibitem{Ku} U.-M. Kunzi, \emph{Decidable theories of pseudo-p-adic closed fields}. Algebra and Logic, Vol 28, No.6 (1989), 421-438.
\itemsep=\smallskipamount

\bibitem{Mac} A. Macintyre, \emph{On definable subsets of $p$-adic Fields}. The journal of Symbolic Logic, Vol 41, No.3(1976), 605-610.
\itemsep=\smallskipamount

\bibitem{Ons} A. Onshuus, \emph{Properties and Consequences of Thorn-Independence}. Journal of Symbolic Logic. vol 71, number 1 (2006). 1-21.  
\itemsep=\smallskipamount

\bibitem{Pre} A. Prestel, \emph{Pseudo real closed fields, Set Theory and Model Theory}. Lecture
Notes 872, Springer (1982). 
\itemsep=\smallskipamount

\bibitem{PR} A. Prestel and P. Roquette, \emph{Formally p-adic fields}. Lecture Notes in Mathematics 1050, Springer, Berlin (1984).
\itemsep=\smallskipamount

\bibitem{PRZI} A. Prestel and M. Ziegler, \emph{Model-theoretic methods in the theory of topological fields}. J. Reine Angew. Math. 299(300) (1978), 318-341.
\itemsep=\smallskipamount


\bibitem{Pp} F. Pop, \emph{Embedding problems over large fields}. Ann. of Maths. (2)144 (1996), no 1, 1-34.
\itemsep=\smallskipamount

\bibitem{SV} P. Scowcroft and L. van den Dries, \emph{On the structure of semi-algebraic sets over $p$-adic fields}. The Journal of Symbolic Logic, vol 53(1988), 1138-1164.
\itemsep=\smallskipamount

\bibitem{Shel} S. Shelah, \emph{Classification theory and the number of non-isomorphic models}. volume 92
of Studies in Logic and the Foundations of Mathematics. North-Holland Publishing
Co., Amsterdam, second edition (1990).
\itemsep=\smallskipamount

\bibitem{Sim0} P. Simon, \emph{On dp-minimal ordered structures}. J. Symbolic Logic Volume 76, Issue 2 (2011).  
\itemsep=\smallskipamount

\bibitem{Sim} P. Simon, \emph{Lecture notes on NIP theories}. Lecture Notes in Logic, 44 (2015).
\itemsep=\smallskipamount

\bibitem{Van} L. van den Dries, \emph{Model theory of fields}. Thesis, Utrecht (1978). 
\itemsep=\smallskipamount

\bibitem{Van1} L. van den Dries, \emph{Algebraic theories with definable Skolem functions}. J. Symbolic Logic 49 (1984), no 2, 625-629.
\itemsep=\smallskipamount

\bibitem{Van2} L. van den Dries, \emph{Tame Topology and o-minimal Structures}. London Mathematical Society Lecture Note Series 248. Cambridge: Cambridge University (1998).
\itemsep=\smallskipamount

\end{thebibliography}
